\documentclass[11pt]{article}
\usepackage{graphicx}
\usepackage{enumerate,amssymb,amsmath,setspace} 
\usepackage{amssymb,amsfonts,amsthm,amscd}
\date{}
\font\sml=cmr6

\font\itten=cmti10 

\newtheorem{mainthm}{{Theorem}}
\newtheorem{maincor}{{Corollary}}

\newtheorem{thm}{Theorem}[section]
\newtheorem{lem}[thm]{Lemma}
\newtheorem{prop}[thm]{Proposition}
\newtheorem{cor}[thm]{Corollary}

\newtheorem{remark}[thm]{{Remark}}
\newtheorem{definition}[thm]{{Definition}}

 \newcommand{\RR}{\mathbb{R}}
 \newcommand{\CC}{{\mathbb C}}
 \newcommand{\NN}{{\mathbb N}}

\newcommand{\acal}{\mathcal{A}}
\newcommand{\ccal}{\mathcal{C}}

\newcommand{\hcal}{\mathcal{H}}

\newcommand{\ocal}{\mathcal{O}}
\newcommand{\qcal}{\mathcal{Q}}

\newcommand{\vcal}{\mathcal{V}}\newcommand{\ucal}{\mathcal{U}}

\def\calQ{\qcal}  
 \def\calA{\acal} 
\def\calU{\ucal} \def\calH{\hcal} 
 \def\calV{\vcal} 
\def\calC{\ccal} \def\calO{\ocal}

\def\a{\alpha} \def\be{\beta}
\def\o{\omega} \def\O{\Omega}
\def\th{\theta} 
\def\vp{\varphi} \def\eps{\epsilon}

\def\bB{{\mathbf B}}

\newcommand{\dbar}{\bar\partial}
\newcommand{\ddbar}{\partial\dbar}
\newcommand{\diag}{{\operatorname{diag}}}

\def\Re{{\operatorname{Re}}}

\def\max{{\operatorname{max}}}

\def\tr{\hbox{\rm tr}}

\def\MA{Monge--Amp\`ere } 

\def\K{K\"ahler } 
\def\KE{K\"ahler--Einstein } \def\KEno{K\"ahler--Einstein}
\def\KRF{K\"ahler--Ricci flow }

\def\polishl{\char'40l}
\def\Kolodziej{Ko\polishl{}odziej} \def\Blocki{B\polishl{}ocki}
\def\Holder{H\"older }

\def\ovp{{\o_{\vp}}}
\def\on{\omega^n}
\def\ovpn{\omega_{\vp}^n}

\def\intM{\int_M}
\def\Ric{\hbox{\rm Ric}\,}

\def\vpKE{\vp_{\h{\sml KE}}}
\def\ovpn{\o^n_{\vp}}

\def\h#1{\hbox{#1}}

\def\strutdepth{\dp\strutbox}
\def\specialstar{\vtop to \strutdepth{
    \baselineskip\strutdepth
    \vss\llap{$\star$\ \ \ \ \ \ \ \ \  }\null}}
\def\marginalstar{\strut\vadjust{\kern-\strutdepth\specialstar}}

\def\marginal#1{\strut\vadjust{\kern-\strutdepth
    {\vtop to \strutdepth{
    \baselineskip\strutdepth
    \vss\llap{{ \small #1 }}\null} 
    }}
    }

\def\text{\textstyle}

 \def\qq{\qquad}

\def\b#1{\bar{#1}}
\def\MsmD{M\!\setminus\!  D}

\def\PSH{\mathrm{PSH}}

\def\ra{\rightarrow}
\def\pa{\partial}
\def\isom{\cong}
\def\w{\wedge}
\def\i{\sqrt{-1}}

\def\jbar{\bar \jmath}
\newcommand{\Dom}{\mathscr{D}}

\newcommand{\D}{D}

\newcommand{\e}{\epsilon}
\newcommand{\del}{\partial}

\newcommand{\phg}{\operatorname{phg}}
\newcommand{\calD}{\mathcal D}

\newcommand{\Diff}{\mathrm{Diff}}
\newcommand{\ff}{\mathrm{ff}}
\newcommand{\lf}{\mathrm{lf}}
\newcommand{\rf}{\mathrm{rf}}
\newcommand{\sm}{\setminus}

\def\id{{\operatorname{id}}}

\def\Dom{\calD}
\def\Ds{\calD_s^{0,\gamma}}
\def\De{\calD_e^{0,\gamma}}
\def\Dw{\calD_w^{0,\gamma}}
\let\s=\sigma

\def\beq{\begin{equation}}
\def\eeq{\end{equation}}
\def\beqno{\begin{equation*}}
\def\eeqno{\end{equation*}}
\def\eaeq{\end{aligned}}
\def\baeq{\begin{aligned}}

\def\bpf{\begin{proof}}
\def\epf{\end{proof}}
\def\G{\Gamma}
\def\phgs{polyhomogeneous}

\def\onebe{{\frac1\be}}

\begin{document}
\title{K\"ahler--Einstein metrics with edge singularities}

\author{T. Jeffres, Rafe Mazzeo, and Yanir A. Rubinstein
\\
\\
{\itten with an appendix by Chi Li and Yanir A. Rubinstein}
}

\maketitle

\begin{abstract}
This article considers the existence and regularity of \KE metrics on a compact \K manifold $M$ with 
edge singularities with cone angle $2\pi \be$ along a smooth divisor $D$.  
We prove existence of such metrics with negative, zero and some positive cases 
for all cone angles $2\pi \be \leq 2\pi$. The results in the positive case parallel those in the smooth case. 
We also establish that solutions of this problem are polyhomogeneous, i.e., have a complete asymptotic 
expansion with smooth coefficients along $D$ for all $2\pi \be < 2\pi$. 
\end{abstract}

\section{Introduction}
Let $\D\subset M$ be a smooth divisor in a compact \K manifold. A \K edge metric on $M$ 
with angle $2\pi \be$ along $\D$ is a \K metric on $\MsmD$ that is asymptotically equivalent at $\D$ 
to the model edge metric
$$
g_\be:=|z_1|^{2\be-2}|dz_1|^2 + \sum_{j=2}^n |dz_j|^2;
$$
here $z_1,z_2,\ldots,z_n$ are holomorphic coordinates such that $\D = \{z_1 = 0\}$ locally.  We always 
assume that $0 < \be \le 1$. 

Of particular interest is the existence and geometry of metrics of this type which are also Einstein.
The existence of \KE (KE) edge metrics was first conjectured by Tian in the mid '90's \cite{T}. 
In fact, Tian conjectured the existence of KE metrics with `crossing' edge singularities 
when $\D$ has simple normal crossings. 
One motivation was his observation that these metrics could be used to prove various inequalities in 
algebraic geometry; in particular, the Bogomolov--Miyaoka--Yau inequality could be proved by deforming
 the cone angle of \KE edge metrics with negative curvature to $2\pi$. Furthermore, these metrics can be used to bound the 
degree of immersed curves in general type varieties.  He also anticipated that the 
complete Tian--Yau KE metric on the complement of a divisor should be the limit of the \KE
edge metrics as the angle $2\pi \beta$ tends to $0$. Recently, Donaldson \cite{D} 
proposed using these metrics in a similar way to construct smooth \KE metrics on Fano manifolds by deforming 
the cone angle of \KE metrics of positive curvature, and more generally to relate this approach
to the much-studied obstructions to existence of smooth \KE metrics. 

One of the main results in this article is a proof of Tian's conjecture on the existence of \KE edge metrics 
when $\D$ is smooth. In a sequel to this article we shall prove the general case \cite{MR}; 
this involves substantial additional complications due to the singularities of the divisor. 

In the lowest dimensional setting, $M$ is a Riemann surface and the problem is to find constant 
curvature metrics with prescribed conic singularities (with cone angle less than $2\pi$) at a finite collection of points. This was 
accomplished in general by McOwen and Troyanov \cite{McO,Tr}; as part of this, Troyanov found 
some interesting restrictions on the cone angles necessary for the existence of spherical cone 
metrics. Later, Luo and Tian \cite{LT} established the uniqueness of these metrics.  For the 
problem in higher dimensions, we focus only on the case where $\D$ is smooth, unless explicitly 
stated. A preliminary study of the case of \KE edge metrics with negative curvature appeared
in the thesis of the first named author \cite{J2}, where it was already realized 
that some of the a priori estimates of Aubin and Yau \cite{Au,Y2} should carry over to this setting
when $\be\in(0,\frac12]$. An announcement for the existence in that negative case 
with $\be\in(0,\frac12]$ was made over ten years ago by the 
first and second named authors \cite{Ma0}.  There were several analytic issues described in that 
announcement which seemed to complicate the argument substantially and details never appeared.

Recently there has been a renewed interest in these problems stemming from an 
important advance by Donaldson \cite{D}, alluded to just above, whose insightful observations make it possible to
establish good linear estimates. He proves a deformation theorem, showing that the set of attainable 
cone angles for KE edge metrics is open. The key to his work is the identification of a function space in the
space of bounded functions on which the linearized \MA equation is solvable. 

We realized, immediately following the appearance of \cite{D}, that only a slight change of perspective suggested by his
advance makes it possible to apply the theory of elliptic edge operators from \cite{Ma1} in a rather direct
manner so as to circumvent the difficulties surrounding the openness part of the argument proposed in \cite{Ma0}.
Indeed, we show that estimates equivalent to those of Donaldson (but on slightly different function spaces) follow 
directly from some of the basic results in that theory, and we explain this at some length in this paper. This alternate
approach to the linear theory allows us to go somewhat further, and we use it to show that solutions 
are polyhomogeneous, i.e., have complete asymptotic expansions in possibly noninteger powers of the
distance to the divisor and positive integer powers of the log of this distance function, with all coefficients
smooth along the divisor. This was announced in \cite{Ma0} and speculated on in \cite{D}, and the existence 
of this higher regularity should be very helpful in the further study of these metrics. 

In fact, this higher regularity plays a fundamental role in the nonlinear a priori estimates. In order to obtain our 
existence results it is necessary to work with minimal assumptions on the reference geometry, and one theorem
we prove, that solutions of the problem are automatically polyhomogeneous, is a key ingredient allowing 
us to do so. Consequently, we are able to establish the existence theorem for all cone angles less than $2\pi$, 
which we carry out in the negative, zero and in certain positive curvature cases. 

More precisely, what we achieve here is the following. We prove existence of \KE edge (KEE) metrics with cone angle 
$2\pi\be$ that have negative, zero and positive curvature, as appropriate, for all cone angles $2\pi \be \leq 2\pi$, 
when $\D$ is smooth. Existence in the positive case is proved 
under the condition that the twisted Mabuchi K-energy is proper, in parallel to Tian's result in the smooth
case \cite{T97}.  Next, we prove that solutions of a general class of complex \MA
equations are polyhomogeneous, i.e., have complete asymptotic expansions with smooth coefficients.
We provide a sharper identification of the function space defined by Donaldson for his deformation result.  As we 
have briefly noted above, there are two slightly different scales of H\"older spaces which play a role in this type of
problem. One, used in \cite{D}, we call the wedge H\"older spaces; the other, from \cite{Ma1}, 
are the edge H\"older spaces. Functions in the wedge H\"older spaces are slighty more regular, which
is crucial in certain parts of the argument; on the other hand, the edge H\"older spaces are invariant 
with respect to the dilation structure inherent in this problem, which makes the linear theory, and 
certain parts of the nonlinear theory, more transparent. 

We shall employ these spaces at various points in the argument. What makes it possible to
go from the edge spaces back to the wedge spaces is 
Tian's regularity argument from Appendix B, which shows that any solution to the \MA equation which is bounded
along with its Laplacian is automatically in a wedge H\"older space. The results in \S 4 then show that the solution is polyhomogeneous. 

The key new ingredient for deriving the nonlinear a priori estimates is the new {\it Ricci continuity method},
which can be considered as a continuity method analogue of the Ricci flow. This was introduced in the context 
of the Ricci iteration by the third named author \cite{R}, and one point of this article is to show that it 
is perhaps the best suited for proving existence of \KE metrics. Indeed, we derive our estimates also for more classical
continuity paths studied in the literature, and at the appropriate junctures indicate how these break down 
unless $\be$ is in the restricted ``orbifold range" $(0,\frac12]$, while this new continuity method works for all $\be\in(0,1]$.

In proving the a priori estimates we have made an effort to extend various classical arguments and bounds to this singular 
setting with minimal assumptions on the background geometry. In particular, the Ricci continuity method together with 
the Chern--Lu inequality allows us to obtain the a priori estimate on the Laplacian assuming only 
that the reference edge metric has bisectional curvature bounded above.
We then explain how the Evans--Krylov theory together with our asymptotic expansion imply a priori H\"older bounds
on the second derivatives for all cone angles with no further curvature assumptions.
Reducing the dependence of the estimates for the existence of a \KE metric
to only an upper bound on the bisectional curvature of the reference 
metric does not seem to have been observed previously even in the smooth setting,
where traditionally a lower bound on the bisectional curvature is required,
or at least an upper bound on the bisectional curvature together with a lower bound on some curvature.
Thus, as a by-product, we also obtain a new and unified proof of the classical results of Aubin, Yau, and Tian,
on existence of KE metrics on smooth compact \K manifolds.
Finally, in the case of positive curvature, we show how to control the Sobolev constant and infimum of the
Green function, which are both needed for the uniform estimate.
In an Appendix it is shown that the bisectional curvature of one
reference metric is bounded from above on $\MsmD$ whenever $\be\in(0,1]$.
The somewhat miraculous calculations to establish this were obtained by the third named
author and Chi Li, and this appendix constitutes yet another necessary component of this work.

Before stating our results, let us mention some other recent articles concerning existence. As already outlined in
\cite{Ma0}, if one were to have linear estimates such as the ones obtained by Donaldson \cite{D}, and if
$\be\in(0,\frac12]$, so that the curvature of the reference metric is bounded, then it is possible
to adapt the classical Aubin--Yau a priori estimates, and hence obtain existence when $\mu\le 0$. 
This was carried out in \cite{Br}. Another quite different approach to existence for $\be\in(0,\frac12]$ and $\mu\le0$ but
allowing divisors with simple normal crossings, based on approximation by 
smooth metrics (and thus avoiding the linear estimates), is due to Campana, Guenancia and P\v aun \cite{CGP}. 
Both \cite{Br,CGP} appeared around the same time as the present article. 
Finally, in a different direction, Berman \cite{Berm} 
showed how to bypass the linear estimates and produce KE metrics whose volume form is asymptotic 
to that of an edge metric using a variational approach.  However, neither of these methods 
give good information about the regularity or the geometry of the solution metric near the divisor.

We now state our main results more precisely. Since some of the terminology in these two Theorems 
is perhaps unfamiliar in complex geometry, we recall the notion of polyhomogeneity described
briefly earlier in this introduction. The existence of a polyhomogeneous expansion should be regarded
as an optimal regularity statement for a solution, and is the natural and unavoidable replacement for smoothness for 
these types of degenerate problems. Just as with the Taylor expansions for smooth functions, the asymptotic
expansions we use in this paper are rarely convergent. We refer to  Sections \ref{PrelimSection} and \ref{LinearSection} 
for more on this and for all relevant notation. 
\begin{mainthm} {\rm (Asymptotic expansion of solutions)}
\label{PhgMainThm}
Let $\o$ be a polyhomogeneous K\"ahler edge metric with angle $2\pi \beta\in(0,2\pi]$. Suppose that, for some H\"older 
exponent $\gamma \in (0,1)$, $u \in \Dom^{0,\gamma}_s\cap \PSH(M,\o)$, $s=w$ or $e$, is a solution of the complex \MA equation
\begin{equation}
\label{genma}
{\o_u^n} = {\o^n} e^{f - s u} , \quad \h{\  on $\MsmD$},
\end{equation}
where $\o_u = \o + \i\,\ddbar u$
and $f \in \calA_{\phg}^0(X)$.  
Then $u$ is polyhomogeneous, i.e., $u \in \calA_{\phg}^0(X)$. 
\end{mainthm}
This result admits a straightforward generalization if the exponential on the right hand side is replaced by a function 
$F(z,u)$ which is polyhomogeneous in its arguments and is such that if $u \in \calA^0_{\phg}$ then $F(z,u) \in \calA_{\phg}^0$. 
\begin{mainthm} {\rm (\KE edge metrics)}
\label{ConicKEMainThm}
Let $(M,\o_0)$ be a compact \K manifold with $D\subset M$ a smooth divisor, and suppose that 
$\mu[\o_0]+(1-\be)[D]=c_1(M)$, where $\be\in(0,1]$ and $\mu\in\RR$. 
If $\mu>0$, suppose in addition that the twisted K-energy $E_0^\beta$ is proper.
Then there exists a \KE edge metric $\o_{\vpKE}$ with Ricci curvature $\mu$ and with angle $2\pi\be$ along $D$.
This metric is unique when $\mu < 0$, unique in its \K class when $\mu = 0$, 
and unique up to automorphisms that preserve $D$ when $\mu>0$.
This metric is polyhomogeneous, namely, $\vpKE$ admits a complete
asymptotic expansion  with smooth coefficients as $r \to 0$ of the form 
\begin{equation}
\label{KEExpEq}
\vpKE(r,\th,Z) \sim \sum_{j, k \geq 0} \sum_{\ell=0}^{N_{j,k}} a_{ j k \ell }(\th,Z)r^{j+k/\be}(\log r)^{\ell},
\end{equation}
where $r = |z_1|^\beta/\beta$ and $\th = \arg z_1$, and with each $a_{jk\ell} \in \calC^\infty$. 
There are no terms of the form $r^\zeta (\log r)^\ell$ with $\ell > 0$ if $\zeta \leq 2$. 
In particular, $\vpKE$ has infinite conormal regularity and a precise H\"older regularity as measured 
relative to the reference edge metric $\o$, which is encoded by $\vpKE \in \calA^0 \cap \Dw$.
\end{mainthm}

We refer to Proposition \ref{PreciseAsympExpansionProp} for the determination of the first several terms 
in the expansion \eqref{KEExpEq}.

To clarify the conclusions about regularity in these theorems, we first prove infinite `conormal' regularity ($\varphi \in \calA^0$), 
which means simply that the solution is tangentially smooth and also infinitely differentiable with respect to the vector 
field $r\del_r$; we then establish H\"older continuity of some second derivatives with respect to the model metric ($\varphi \in 
\Dw$); finally, we prove the existence of an asymptotic expansion in powers of the distance to the edge ($\varphi \in 
\calA^0_{\phg}$). This expansion also leads to the precise asymptotics of the curvature tensor and its covariant derivatives.
For example, when $\be\le\frac12$ we have $\vpKE\in \calC^{2,\onebe-2}_w$, all third derivatives of the form $(\vpKE)_{i\b j k}$
belong to $\calC^{0,\onebe-2}_w$, and therefore so do all Christoffel symbols, and the curvature tensor of $\o_{\vpKE}$ is 
\Holder continuous.  However, assuming only that $\be\le1$, we have $\Delta_\o\vpKE\in \calC^{0,\gamma}_w$
for some $\gamma\in(0,\frac1\be-1]$, but in general the curvature tensor does not lie in $L^\infty$ (this follows
readily from the calculations of the Appendix). Again, we refer to Proposition \ref{PreciseAsympExpansionProp} for
more precise information.

Theorem \ref{ConicKEMainThm} is the generalization to the edge setting of the classical theorems
of Aubin, Yau ($\mu\le0$), and Tian ($\mu>0$) on existence of KE metrics in the compact smooth setting 
\cite{Au,Y2,T97}. Its proof gives a new and unified treatment for all $\mu$ even in the smooth setting.
It is also a satisfactory generalization of Troyanov's theorem on the existence of constant curvature 
metrics with conic singularities on Riemann surfaces \cite{Tr} inasmuch as the cone angle restrictions which
appear in his work arise only in the positive curvature case, and they are the same as the 
properness of the twisted $K$-energy in that setting.  Just as for the smooth setting \cite{T97}, the 
properness assumption should be a necessary condition for existence.

Finally, consider the special case that $M$ is Fano and $D$ is a smooth anticanonical
divisor (the existence of such a divisor is related to the so-called Elephant Conjectures
in algebraic geometry, and is known when $n\le3$ by work of Shokurov and others). 
Then, as noted by Berman \cite{Berm}, the twisted K-energy is proper for small $\mu=\be$. 
Theorem \ref{ConicKEMainThm} thus gives the following corollary conjectured by Donaldson \cite{Don2009}.
\begin{maincor}
\label{FanoCor}
Let $M$ be a Fano manifold, and suppose that there exists a smooth anticanonical divisor
$D\subset M$. Then there exists some $\be_0\in(0,1]$ such that for all $\be\in(0,\be_0)$
there exists a KEE metric with angle $2\pi\be$ along $D$ and with 
positive Ricci curvature equal to $\be$.
\end{maincor} 

{\it Added in revision:}
There has been substantial work in this area in the years following the initial appearance of this article, 
cf.\ in particular, the papers \cite{CMR,RZ,CDS,T12}. We refer the reader to the survey \cite{R14} for 
further references and background.  In both \cite{CDS,T12}, the construction of a smooth KE metric is carried out 
by studying the deformations of a KEE metric as the cone angle increases. 

Our original proof of the $\calD^{0,\gamma}_w$ estimate had an error, now corrected by Appendix B. 
The papers \cite{CDS} contain a different approach to this estimate. 

\subsubsection*{Acknowledgements} 
The authors are grateful to Gang Tian for his advice and encouragement throughout the course of this project, and
also for explaining to us his argument which appears in Appendix B. Various referees made numerous valuable
remarks which led to a better exposition and which drew our attention to errors, some rather important, in earlier versions.
The authors also offer many thanks to Simon Donaldson for his interest in this work, and for numerous comments which
helped sharpen, clarify and in some instances correct the exposition. The NSF supported this research through 
grants DMS-0805529 (R.M.) and DMS-0802923,1206284 (Y.A.R.).

\section{Preliminaries}\label{PrelimSection}
We set the stage for the rest of the article with a collection of facts and results needed later. First consider the flat model 
situation, where $M = \CC^n$ with linear coordinates $(z_1, \ldots, z_n)$,
and $\D$ is the linear subspace $\{z_1 = 0\}$. For brevity we often write $Z = (z_2, \ldots, z_n)$. The model 
singular K\"ahler form and singular \K metric are given by
\begin{eqnarray}
\label{modelform} 
\omega_\beta & = & \frac12 \sqrt{-1} \Big(
|z_1|^{2\beta-2} dz^1 \w \overline{dz^1} + \sum_{j=2}^n dz^j \w \overline{dz^j}
\Big), 
\ \mbox{and}\\
g_\beta & =  & |z_1|^{2\beta-2} |dz^1|^2 + \sum_{j=2}^n |dz^j|^2.
\label{modelmetric}
\end{eqnarray}
This is the product of a flat one complex dimensional conic metric with cone angle $2\pi \beta$
with $\CC^{n-1}$. We always assume that $0 < \beta \leq 1$; the expressions above make sense for any real $\beta$, but
their geometries are quite different for $\beta$ outside of this range. 

Now suppose that $M$ is a compact K\"ahler manifold and $\D$ a smooth divisor. Fix $\beta \in (0,1]$ and
$\mu \in \RR$, and assume that there is a K\"ahler class $\Omega = \Omega_{\mu, \beta}$ such that
\begin{equation}
\mu \Omega + 2\pi(1-\beta) c_1(L_\D) = 2\pi c_1(M).
\label{cohom}
\end{equation}
Here, $L_\D$ is the line bundle associated to $\D$.  Thus, $c_1(M) - (1-\be)c_1(L_D)$ is a positive or negative 
class if $\mu > 0$ or $\mu<0$. If $\mu = 0$, $\O$ is an arbitrary \K class.

Let $g$ be any K\"ahler metric which is smooth (or of some fixed finite regularity) on $M \setminus \D$. 
We shall say that $g$ is a \K edge metric with angle $2\pi \beta$ if, in any 
local holomorphic coordinate system near $\D$ where $\D = \{z_1 = 0\}$, and $z_1 =  \rho e^{\i\theta}$,  
\begin{equation}
\label{pertgb}
g_{1\bar{1}} = F \rho^{2\beta-2},\ g_{1 \jbar} = g_{i \bar{1}} = O(\rho^{\beta - 1 + \eta'}),\ \mbox{and all other}\ g_{i\jbar} = O(1),
\end{equation}
for some $\eta' > 0$, where $F$ is a bounded nonvanishing function which is at least continuous at $\D$ 
(and which will have some 
specified regularity).  If this is the case, we say that $g$ is asymptotically equivalent to $g_\beta$, and that its associated 
K\"ahler form $\omega$ (which by abuse of terminology we sometimes also refer to as a metric) is asymptotically equivalent to 
$\omega_\beta$.  There are slightly weaker hypotheses under which it is reasonable to say that $g$ has angle $2\pi \be$ at 
$\D$, but the definition we have given here is sufficient for our purposes.
We denote by $\Ric\o$ the Ricci current (on $M$) associated to $\o$, namely, in local coordinates
$\Ric\o=-\i\ddbar\log\det[g_{i\bar j}]$ if $\o=\i g_{i\bar j} dz^i\wedge \overline{dz^j}$.
Thus, $\Ric\o-2\pi(1-\be)[D]$ is a $(1,1)$ 
current on $M$ with a continuous potential,
where $[\D]$ is the current associated to integration along $\D$. 

\begin{definition}
With all notation as above, a K\"ahler current $\omega$, with associated singular K\"ahler metric $g$, is called a 
K\"ahler--Einstein edge current, respectively metric, with angle $2\pi \beta$ along $D$ and Ricci
curvature $\mu$ if $\omega$ and $g$ are asymptotically equivalent to $\omega_\beta$ and $g_\beta$, and if 
\begin{equation}
\label{EdgeKEEq}
\Ric \omega - 2\pi(1-\beta) [\D]= \mu \omega.
\end{equation}
\label{Kcurrent}
\end{definition}
In this section we present some preliminary facts about the geometry and analysis of the class of \K edge metrics.
We first review some different coordinate charts 
near the edge $\D$ used extensively below. Many calculations in this article are most easily done 
in a singular real coordinate chart, although when the complex structure is particularly relevant to a calculation, 
we use certain adapted complex coordinate charts. While all of this is quite elementary, there are some 
identifications that can be confusing, so it is helpful to make all of this very explicit.  
We calculate the curvature tensor for any one such metric $g$, assuming it is sufficiently regular. 
We then introduce the relevant class of \K edge potentials and describe the continuity method that will be used for the 
existence theory. As we recall, this particular continuity method is 
closely related to the Ricci iteration, 
that, naturally, we also treat simultaneously in this article. We conclude the section with a fairly lengthy description of 
the various function spaces that will be used later. Rather than a purely technical matter, this discussion 
gets to the heart of some of the more important analytic and geometric issues that must be faced here.  
There are two rather different choices of H\"older spaces; one is naturally associated to this 
class of K\"ahler edge metrics and was employed, albeit in a slightly different guise, by Donaldson \cite{D}, while the
other, from \cite{Ma1}, is well adapted to this edge geometry because of its naturality under dilations and has 
been used in many other analytic and geometric problems where edges appear. Use of these latter function
spaces is central to our method.

\subsection{Coordinate systems}
As above, fix local complex coordinates $(z_1, \ldots, z_n) = (z_1, Z)$ with $\D= \{z_1 = 0\}$ locally.  
There are two other coordinate systems which are quite useful for certain purposes.
The first is a singular holomorphic coordinate chart, 
where we replace $z_1$ by $\zeta = z_1^\beta/\beta$. Of course, 
$\zeta$ is multi-valued, but we can work locally in the logarithmic Riemann surface which uniformizes this variable. 
Thus if $z_1 =  \rho e^{\i\theta}$, then $\zeta = r e^{\i\tilde{\theta}}$, 
where $r = \rho^\beta/\beta$ and $\tilde{\theta} = \beta \theta$. 
The second is the real cylindrical coordinate system $(r, \theta, y)$ around $\D$, where $r = |\zeta|$ as above, $\theta$ is 
the argument of $z_1$, and $(y_1, \ldots, y_{2n-2}) = (\mbox{Re}\, Z, \mbox{Im}\, Z)$. Note
that $re^{\i\theta} = z_1|z_1|^{\beta-1}/\beta$. We use either $(z_1, Z)$ or $(\zeta, Z)$ in situations where the formalism of 
complex analysis is useful, and $(r,\theta,y)$ elsewhere.  For later purposes, note that
\begin{equation}
d\zeta = z_1^{\be-1} dz_1 \Leftrightarrow dz_1 = (\be \zeta)^{\frac{1}{\be}-1} d\zeta, \qquad
\frac{\del \zeta}{\del z_1} = z_1^{\be-1} \Leftrightarrow \frac{\del z_1}{\del \zeta} = (\be \zeta)^{ \frac{1}{\be}-1}.
\label{zzeta}
\end{equation}

One big advantage of either of these other coordinate systems is that they make the model metric $g_\beta$ appear less singular. 
Indeed, 
\begin{equation}
g_\beta = |d\zeta|^2 + |dZ|^2 = dr^2 + \beta^2 r^2 d\theta^2 + |dy|^2.
\label{modmetric2}
\end{equation}
In either case, one may regard the coordinate change as encoding the singularity of the metric via a singular coordinate 
system.  This is only possible for edges of real codimension two, and there are many places, both in 
\cite{D} and here, where we take advantage of this special situation. For edges of higher codimension, one cannot conceal 
the singular geometry so easily, see \cite{Ma1}.  The expression for $g_\beta$ in cylindrical coordinates makes clear that 
for any $\beta, \beta'$, we have $C_1 g_{\be} \leq g_{\be'} \leq C_2 g_\be$; the corresponding inequality in the original $z$ coordinates
must be stated slightly differently, as $C_1 g_{\be} \leq \Phi^* g_{\be'} \leq C_2 g_\be$, where $\Phi(z_1, \ldots, z_n) = (z_1^{\be'/\be}, z_2, 
\ldots, z_n)$. 

We now compute the complex derivatives in these coordinates. We have
\begin{eqnarray}  
\del_{z^1} & = & \frac12 e^{-\i\theta}( \del_\rho - \frac{\i}{\rho} \del_\theta) 
=  \frac12 e^{-\i\theta} (\beta r)^{1-\frac{1}{\beta}} (\del_r - \frac{\i}{\beta r} \del_\theta), 
\label{del1}  
\end{eqnarray}
and then
\begin{equation}
\del^2_{z^1 \overline{z^1}}  =  (\beta r)^{2 - \frac{2}{\beta}} \left( \del_r^2 + \frac{1}{r} \del_r + \frac{1}{\beta^2 r^2} \del_\theta^2\right).
\label{del11bar}
\end{equation}
The other mixed complex partials $\del^2_{z^1 \overline{z^j}}$, $\del^2_{z^i \overline{z^1}}$ and $\del^2_{z^i \overline{z^j}}$ are compositions of 
the operators in \eqref{del1} and their conjugates
and certain combinations of the $\del_{y_\ell}$. From this we obtain that 
\begin{equation}
\Delta_{g_\beta}u = \sum_{i, j = 1}^n (g_\beta)^{i\jbar}u_{i\jbar} = \left(\del_r^2 + \frac{1}{r} \del_r + \frac{1}{\beta^2 r^2} 
\del_\theta^2 + \Delta_y \right) u,
\label{modellap}
\end{equation}
since $(g_\beta)^{1\bar{1}} = \rho^{2-2\beta} = (\beta r)^{\frac{2}{\beta}-2}$ and 
$(g_\beta)^{1\jbar}, (g_\beta)^{i\bar{1}} = 0$  and all other
$(g_\beta)^{i \jbar} = \delta^{ij}$.   

As already described, we shall work with the class of K\"ahler metrics $g$ that 
satisfy condition \eqref{pertgb}, and which
we call asymptotically equivalent to $g_\beta$.  If $g$ is of this type, then 
\begin{equation}
\label{gupperijEq}
g^{1\bar{1}} = F^{-1}\rho^{2-2\beta},\quad g^{1 \jbar}, g^{i \bar{1}} = O(\rho^{\eta'+1-\beta}),
\quad \mbox{and all other}\ g^{i\jbar} = O(1)
\end{equation}
for some $\eta' > 0$, hence
\begin{equation}
\Delta_g = F^{-1}\left(\del_r^2 + \frac{1}{r}\del_r + \frac{1}{\beta^2 r^2} \del_\theta^2  + 
\sum_{r, s = 1}^{2n-2} c_{rs}(r,\theta,y) \del^2_{y_r y_s}\right)  + E ,
\label{decompLap}
\end{equation}
where
\[
E :=  r^{\eta-2} \sum_{i+j+|\mu| \leq 2} a_{i j \mu}(r,\th,y) (r\del_r)^i \del_\th^j (r\del_y)^\mu .
\]
Here $\eta > 0$ is determined from $\eta'$ and $\beta$, and all coefficients have some specified regularity down to $r=0$.
In particular, the coefficient matrix $(c_{rs})$ is positive definite, with $c_{rs}(0,\th,y)$ independent of $\th$, and the 
coefficients $a_{ij\mu}$ are bounded as $r\to 0$. Thus there are no cross-terms to leading order, and the $1 \b1$ part of
the operator $\Delta_g$ is `standard' once we multiply the entire operator by $F$. 

One way that this asymptotic structure will be used is as follows.  Fundamental to this work is the role of the family of dilations
$S_\lambda: (r,\th,y) \mapsto (\lambda r, \th, \lambda y)$ centered at some point $p \in \D$ corresponding to $y=0$.  If we push 
forward this operator by $S_\lambda$, which has the effect of expanding a very small neighbourhood of $p$, then the principal
part scales approximately like $\lambda^2$ while $E$ scales like $\lambda^{2-\eta}$. Hence, after a linear change of
the $y$ coordinates, 
\begin{equation}
A \lambda^{-2}(S_\lambda)_* \Delta_g  \longrightarrow  \Delta_{g_\beta}\ \ \mbox{as}\ \ \lambda \to \infty
\label{scalLap}
\end{equation}
where $A = F(p)$. In particular, $E$ scales away completely in this limit. 

One important comment is that if the derivatives $u_{i\jbar}$ are all bounded, and if $g$ satisfies these asymptotic conditions, 
then so does $\tilde{g}$, where $\tilde{g}_{i\jbar} = g_{i\jbar} + u_{i \jbar}$. 

A key point in the treatment below, exploited by Donaldson \cite{D}, is that for any K\"ahler metric $g$,
$\Delta_g$ only involves combinations of the following second order operators: 
\begin{eqnarray*}
P_{1\bar{1}} & = & ( \del_r^2 + \frac{1}{r}\del_r + \frac{1}{\beta^2 r^2} \del_\theta^2), \\ 
P_{1 \bar{\ell}} & = & ( \del_r - \frac{\i}{\beta r} \del_\theta) \del_{\overline{z_\ell}},  \\ 
P_{\ell \bar{1}} & = & ( \del_r + \frac{\i}{\beta r} \del_\theta) \del_{z_\ell},\ \ \ell = 2, \ldots, n, \mbox{\ and} \\  
P_{\ell \bar{k}} & = & \del^2_{z_\ell \overline{z_k}}, \ \ \ell, k = 2, \ldots, n.
\end{eqnarray*}
Regularity properties in certain function spaces considered below involve precisely these derivatives, while others are
less sensitive about the decomposition of $\del_{z_j}$ and $\del_{\overline{z_j}}$ into their $(1,0)$ and $(0,1)$ parts. We therefore
introduce the following collections of differential operators:
\begin{equation}
\begin{aligned}
{\calQ} & = \{ \del_r, r^{-1} \del_\theta, \del_{y_\ell}, \del^2_{r y_\ell}, \del^2_{r \theta}, \del^2_{\theta y_\ell}\} \\[0.5ex]
\calQ^*  & = {\calQ} \sqcup \{\del^2_{y_k y_\ell}, P_{1 \bar{1}} \}.
\end{aligned}
\label{defQ}
\end{equation}
The reason for singling out the extra operators in $\calQ^*  \setminus {\calQ}$ is that the relevant boundedness properties are 
more subtle for these. 

As a final note, let us record the form of the complex \MA operator in these coordinates, for any K\"ahler metric
which satisfies the decay assumptions above. We have
\[
(\o + \i\, \ddbar u)^n/\o^n = \frac{\det (g_{i\jbar} + \sqrt{-1} u_{i\jbar})}{\det g_{i\jbar}} = \det ( \delta_i^j + \sqrt{-1} u_i^j),
\]
where $u_i^j = u_{i \bar{k}}g^{j \bar{k}}$. Using the calculations above, we have
\begin{eqnarray*}
u_1^{\  1} & = & F^{-1} P_{1 \bar{1}} u + O(r^{\eta})u_{1 \jbar}, \\ 
u_1^{\  j} & = & e^{-\i \, \theta} (\beta r)^{1 - \frac{1}{\beta}} g^{j \bar k} P_{1 \bar k}u   + O(r^{\eta + \frac{1}{\be} - 1})P_{1\bar{1}}u, \\ 
u_i^{\  1} & = & F^{-1} e^{\i \, \theta} (\beta r)^{\frac{1}{\beta} - 1} P_{i \bar{1}}u + O( r^{\eta + \frac{1}{\be}-1}) u_{i\jbar}, \\
u_i^{\ j} & = & g^{j \bar{k}} P_{i \bar{k}}u  + O(r^\eta) u_{i \bar{1}}.
\end{eqnarray*}
Hence if we multiply every column but the first in $(\delta_i^{\ j} + \sqrt{-1}u_i^{\ j})$ 
by $e^{\i\theta}(\beta r)^{1/\beta - 1}$ 
and every row but the first by $e^{-\i \theta} (\beta r)^{1-1/\beta}$, then the determinant remains the same, and we have
shown that
\begin{equation}
\label{RafeRewriteCMAEq}
\frac{\det (g_{i\jbar} + \sqrt{-1} u_{i\jbar})}{\det g_{i\jbar}} = 
\det \left( \begin{matrix}
1 + F^{-1} P_{1 \bar{1}}\, u & F^{-1}P_{2 \bar{1}}\, u  & \ldots & F^{-1} P_{n \bar{1}}u \cr
\vdots & \vdots & \vdots  & \vdots \cr
g^{n\b k} P_{1 \b k} u & \ldots & \ldots  & 1 + g^{n\bar{\ell}}P_{n \bar{\ell}} \, u 
\end{matrix} \right) + R,
\end{equation}
where $R = r^{\eta}R_0( u_{p \bar{q}})$, with $R_0$ polynomial in its entries. 

\subsection{\K edge potentials}
\label{KEdgePotentialsSubSec}
Fix a smooth \K form $\o_0$ with $[\o_0] \in\O\equiv \Omega_{\mu,\be}$. Consider the space of all \K potentials relative to $\o_0$, asymptotically equivalent to the model metric,
\begin{multline}
\label{KEdgePotentialsSpace}
\calH_{\o_0} := \{\vp\in \calC^\infty(\MsmD)\cap \calC^0(M) 
\, :\,   \o_{\vp} := \o_0+\i\ddbar\vp>0 \h{\ on\ } M  \\
\mbox{and $\o_\vp$ asymptotically equivalent to}\  \omega_\be\}.
\end{multline}
Note that in our notation $\calH_{\o_0}\isom\calH_{\eta}$ for any smooth $\eta$ cohomologous to $\o_0$
but not for any $\eta\in\calH_{\o_0}$.
The first observation is that such \K edge metrics exist. 
\begin{lem}
\label{ReferenceMetricExistsLemma}
Let $\be\in(0,1]$. Then $\calH_{\o_0}$ is non-empty.
\end{lem}
\begin{proof}
Let $h$ be a smooth Hermitian metric on $L_D$, and let $s$ be a global holomorphic section of $L_D$
so that $\D = s^{-1}(0)$. We claim that for $c>0$ sufficiently small, the function
\begin{equation}
\label{PhizeroEq}
\phi_0:=c |s|_h^{2\beta}=c(|s|_h^2)^\be 
\end{equation}
belongs to $\calH_{\o_0}$. To prove this, it suffices to consider $p\in\MsmD$ near
$D$. Use a local holomorphic frame $e$ for $L_D$ and local holomorphic coordinates $\{z_i\}_{i=1}^n$ valid
in a neighborhood of $p$, such that $s=z_1e$, so that locally $\D$ is cut out by $z_1$. Let 
\begin{equation}
\label{afunctionEq}
a:=|e|^2_h, 
\end{equation}
and set $H:=a^\be$, so  $|s|^{2\be}_h=H|z_1|^{2\be}$. Note that
$H$ is smooth and positive.  Then
\begin{equation}
\label{ddbarPhioneFirstEq}
\begin{aligned}
\i \, \ddbar |s|_h^{2\be} &= \beta^2 H |z_1|^{2\be - 2} \i\, dz_1\w \overline{dz_1} \\ 
+ & 2\be\Re(|z_1|^{2\be}{z_1}^{-1}\i\, dz_1\w \dbar H) +|z_1|^{2\be}\i\, \ddbar H.
\end{aligned}
\end{equation}
For $c>0$ small, the form $\o_0 + \i\ddbar \phi_0$ is positive definite and satisfies the 
conditions of \eqref{pertgb}, hence is asymptotically equivalent to $g_\be$. 
\end{proof}

It is useful to record the form of $\o_{\phi_0}$ in the $(\zeta,Z)$ coordinates as well. First note that if $\psi_0$ is a \K potential
for $\o_0$, then using \eqref{zzeta},
\begin{equation}
\begin{aligned}
\i\, \ddbar \psi_0 =  & (\psi_0)_{z_1 \bar{z_1}}  |\be \zeta|^{\frac{2}{\be}-2} \i\, d\zeta \w \overline{d\zeta}  \\
& + 
\sum_{j>1}2 \Re ((\psi_0)_{z_1 \bar{z_j}} (\be \zeta)^{\frac{1}{\be}-1} \i\, d\zeta \w \overline{dz_j})
 + \sum_{i,j>1} (\psi_0)_{z_i \bar{z_j}} \i\, dz_i \w \overline{dz_j}.
\end{aligned}
\label{psizeta}
\end{equation}
Next, $|s|_h^{2\be} = \be^2 H |\zeta|^2$, hence
\begin{equation}
\i\, \ddbar (c |s|_h^{2\be}) = c \be^2 \big( \i\, H d\zeta \w \overline{d\zeta} + 2 \Re ( \bar{\zeta} \i\, d \zeta \w \dbar H) + 
|\zeta|^2 \i\, \ddbar H\big).
\label{ddbars2b}
\end{equation}
From these two expressions, it is clear once again that $\phi_0 \in \calH_{\o_0}$ when $c$ is sufficiently small. 
Putting these expressions together shows that $\o_0 + \i\, \ddbar \phi_0$ is locally equal to
\begin{equation}
\label{ddbarPhioneSecondEq}
\begin{aligned}
 & \left( |\be\zeta|^{\frac{2}{\be}-2} (\psi_0)_{z_1\overline{z_1}} 
         + c \beta^2 H+ c |\be\zeta|^{\frac{2}{\be}}
H_{z_1\overline{z_1}}
+ 2 c \be^{\frac{1}{\be}+1} \Re({\bar\zeta}^{\frac1\be}H_{\overline{z_1}})  \right) 
\i \, d\zeta\w \overline{d\zeta}  \\
& \quad
+2\Re \sum_{j>1}  \left( (\be\zeta)^{\frac1\be-1}(\psi_0)_{z_1\overline{z_j}} 
+ c \be^2 {\bar\zeta} H_{\overline{z^j}} +c \be^{\frac{1}{\be}+1} 
\zeta^{\frac{1}{\beta}} {\bar\zeta}H_{z_1 \overline{z^j}} \right) \i\, d\zeta\w \overline{dz^j} \\
& \quad +\i\, \del_Z\dbar_Z \psi_0 +|\zeta|^{2}\i\, \del_Z\dbar_Z H.
\end{aligned}
\end{equation}
The reason for writing the derivatives of $\psi_0$ and $H$ with respect to $z_1$ rather than $\zeta$ is 
because we know that both of 
these functions are smooth in the original $z$ coordinates, and hence so are 
its derivatives with respect to $z$. 

We now use this expression to deduce some properties of the curvature tensor of $g$. 
This turns out to be simple in this singular holomorphic 
coordinate system. The coefficients of the $(0,4)$ curvature tensor are given by
\begin{equation}
\label{CurvatureCoeffEq}
R_{i\jbar k\b l} = -g_{i\jbar,k\b l}+g^{s\b t}\, g_{i\b t,k} \, g_{s\b j,\b l},
\end{equation}
where the indices after a comma indicate differentiation with respect to a variable. 
In the following, contrary to previous notation, we temporarily use the subscripts $1$ and $\bar{1}$ to 
denote components of the metric or derivatives with respect 
to $\zeta$ and $\bar{\zeta}$, not $z_1$ or $\overline{z_1}$. 

\begin{lem}
\label{CurvRefMetricLemma}
The curvature tensor of $\o = \o_0+\i\,\ddbar \phi_0$ is 
uniformly bounded on $\MsmD$ provided $\be\in(0,\frac12]$.
\end{lem}
\begin{proof}
Since in the $(\zeta, Z)$ coordinates, $cI<[g_{i\jbar}]<CI$, it suffices to show that $|R_{i\jbar k\b l}|<C$.
From \eqref{ddbarPhioneSecondEq},
\begin{equation*}
\begin{aligned}
g_{1\bar{1}, 1}  & = O( |\zeta|^{\frac{2}{\be}-3}), \\
g_{1\bar{1}, k} & =  O(1), 
\end{aligned}
\qquad
\begin{aligned}
g_{1\jbar, 1} & = O(|\zeta|^{\frac{1}{\be}-2}), \\
g_{1\jbar, k} & = O(|\zeta| + |\zeta|^{\frac{1}{\be}-1}),
\end{aligned}
\qquad 
\begin{aligned}
g_{i\jbar, 1}  & = O(|\zeta|^{\frac{1}{\be}-1}),  \\
g_{i\jbar, k} & =  O(1). 
\end{aligned}
\end{equation*}
Similarly,
$|g_{i\jbar,k\b l}|\le C(1+|\zeta|^{\frac1\be-2}+
|\zeta|^{\frac2\be-4})$. 
\end{proof}
As indicated by Donaldson, there seem to be genuine cohomological obstructions to finding reference
edge metrics with bounded curvature when $\be > 1/2$.  Nevertheless, in Proposition~\ref{AppendixProp} 
it is shown that the bisectional curvature of $\o$ is bounded from above on $\MsmD$ provided $\be \leq 1$.
This fact comes out of the proof as some kind of miracle, yet it would be enlightening 
to have a more geometric explanation for it. In a related vein, we remark in passing that it is not difficult to 
write down local expressions (near $D$) for metrics equivalent to the model metric which have bisectional 
curvature unbounded from above or below or both. For instance, when $n=1$, the curvature of 
$(|z_1|^{2\be-2}-1)|dz_1|^2$ equals $(1-\be)^2\rho^{-2\be}/(1-\rho^{2-2\be})^2$ (here $\rho=|z_1|$), 
and hence tends to $+\infty$ as $\rho \searrow 0$. More generally, one can easily choose $\psi$ 
polyhomogeneous but not smooth so that the curvature of the metric $|z_1|^{2\be-2}e^\psi|dz_1|^2$ is 
unbounded, either above or below or both. In other words, the asymptotics of the curvature depend
on the higher order terms in the expansion of the metric. 

\subsection{The twisted Ricci potential}
\label{RicciPotentialSubSec}
From now on (except when otherwise stated) we denote
\begin{equation}
\label{oDefEq}
\o:=\o_0+\i\, \ddbar\phi_0\in \calH_{\o_0},
\end{equation}
with $\phi_0$ given by \eqref{PhizeroEq}. In the remainder of this article, we refer to $\o$ as the reference metric.
Define $f_\o$ by
\begin{equation}
\label{foDefEq}
\i\ddbar f_\o=\Ric\o-2\pi(1-\be)[D]-\mu\o,
\end{equation}
where $[D]$ denotes the current of integration along $D$,
and with the normalization
\begin{equation}
\frac1V\int_M e^{f_\o}\o^n=1, \quad \h{where\ \ } V\!:=\int_M\o^n.
\label{normalization}
\end{equation}
We call this the twisted Ricci potential (see
Lemma \ref{foAsympExpansionLemma} for precise regularity of $f_\o$); this terminology refers to the fact that the adjoint bundle 
$K_M+(1-\be)L_D$ takes the place of the canonical bundle $K_M$. Alternatively, one can
also think of $\Ric\o-2\pi(1-\be)[D]$ as a kind of Bakry-\'Emery Ricci tensor.

\subsection{The Ricci continuity method for the twisted \KE equation}
\label{CMSubSec}
The existence of \KE metrics asymptotically equivalent to $g_\be$ is governed by the 
\MA equation
\begin{equation}
\label{KEMAEq}
\o_\vp^n=e^{f_\o-\mu \vp}\o^n. 
\end{equation}
We seek a solution $\vp\in \calH_{\o_0}$, and shall do so using a particular continuity method. We consider a continuity path 
in the space of metrics $\calH_{\o}$ (with some specified regularity) obtained from the Ricci flow via a backwards 
Euler discretization, as first suggested in \cite{R}. Alternatively, it can be obtained essentially by concatenating (and extending) two previously 
studied paths, one by Aubin \cite{Aubin1984} in the positive case and the other by Tian--Yau \cite{TY} in the negative case.
The path is given by
\begin{equation}
\label{RCMEq}
\o_\vp^n=\on e^{f_\o-s\vp}, \quad s\in(-\infty,\mu],
\end{equation}
where $\vp(-\infty)=0$, and $\o_{\vp(-\infty)}=\o$. We call this the {\it Ricci continuity path.}
Adapting the proof of a result of Wu \cite{Wu}, we prove later that there exists a solution $\vp(s)$ for $s \ll 0$
of the form $s^{-1}f_\o+o(1/s)$. A key feature of this continuity path is that
\begin{equation}
\label{RicAlongCMEq}
\Ric\o_\vp=s\o_\vp+(\mu-s)\o+2\pi(1-\be)[D], 
\end{equation}
which implies the very useful property that for all 
solutions $\vp(s)$ along this path, the Ricci curvature is bounded below on $\MsmD$, i.e.,\ 
$\Ric\o_\vp>s\o_\vp$.  As we explain in \S\ref{TianMabuchiSubSec},  
another important property is that the Mabuchi K-energy is monotone along this path. 

Much of the remainder of this article is directed toward analyzing this family of \MA equations:
\S 3 describes the linear analysis needed to understand the openness part of the continuity
argument as well as the regularity theory; \S 4 uses this linear analysis to prove that solutions are 
automatically `smooth' 
at $\D$, by which we mean that they are polyhomogeneous (see below); the a priori estimates 
needed to obtain the closedness of the continuity argument are derived in the remaining
sections of the article, and the proof is concluded in \S\ref{ProofKESection}.

We will pursue a somewhat parallel development of this proof using two different scales of
H\"older spaces since we hope to illustrate the relative merits of each of these classes of function spaces, 
with future applications in mind.  Certain aspects of the proof work much more easily in one setting
rather than the other, but we give a complete proof of the a priori estimates in either framework. The 
proof of higher regularity, which shows that these two approaches are ultimately equivalent and facilitates
the continuity argument, relies directly on only one of these scales of spaces. 

\begin{remark}
\label{OneTwoParamRemark}
{\rm  The continuity path \eqref{RCMEq} has several useful properties, some already
noted above, which are necessary for the proof of Theorem \ref{ConicKEMainThm}
when $\be > 1/2$. However, we also consider the two-parameter family of equations
\begin{equation}
\label{TwoParamCMEq}
\o_{\vp}^n=e^{tf_\o+c_t-s\vp}\on, 
\quad 
c_t:=-\log\frac1V\int_M e^{tf_\omega}\omega^n, \quad (s,t)\in A,
\end{equation}
where $A:=(-\infty,0]\times[0,1]\;\cup [0,\mu]\times\{1\}$. This incorporates the
continuity path $s = \mu, 0 \leq t \leq 1$, which is the common one in the literature.  
The analysis required to study this two-parameter family requires little extra effort, 
and has been included since it may be useful elsewhere. It provides an opportunity to 
use the Chern--Lu inequality in its full generality (see \S\ref{LaplacianSection}).
In addition, we have already noted that one cannot obtain openness for \eqref{RCMEq} 
at $s=-\infty$ directly, but must produce a solution for $s$ (very) negative by
some other method. Wu \cite{Wu} accomplishes this by a perturbation argument; 
the augmented continuity path \eqref{TwoParamCMEq} gives yet another means to do this, 
but works only when $\be \leq 1/2$. We refer to \S\ref{ProofKESection} for more details. 

We emphasize that our proof of Theorem \ref{ConicKEMainThm} when $\be > 1/2$ or when 
$\mu > 0$, requires the path \eqref{RCMEq} (i.e., fixing $t=1$). }
\end{remark}

\subsection{The twisted Ricci iteration} \label{RISubSec}
The idea of using the particular continuity method \eqref{RCMEq} 
to prove the existence of \KE metrics for all $\mu$ (independently of sign) was suggested in \cite[p. 1533]{R}.
As explained there and recalled below, this path arises from discretizing the Ricci flow via the
Ricci iteration. After treating this continuity path we will be in a 
position to prove smooth edge convergence of the (twisted) Ricci iteration to the \KE edge metric. 
 
One \KRF in our setting is 
$$
\frac{\pa \o(t)}{\pa t} =-\Ric\,\o(t)+2\pi(1-\be)[D]+\mu\o(t),\quad \o(0)  =\o\in\calH_{\o_0}.
$$
Let $\tau\in(0,\infty)$.
The (time $\tau$) Ricci iteration, introduced in \cite{R}, is the sequence $\{\o_{k\tau}\}_{k\in\NN}\subset\calH_\o$,
satisfying the equations
$$
\o_{k\tau} =\o_{(k-1)\tau}+\tau\mu\o_{k\tau}-\tau\Ric\o_{k\tau}+\tau2\pi(1-\be)[D],
\qquad \o_{0\tau}  =\o,
$$
for each $k\in\NN$ for which a solution exists in $\calH_\o$.  
This is the backwards Euler discretization of the K\"ahler--Ricci flow. Equivalently,
let $\o_{k\tau}=\o_{\psi_{k\tau}}$, with $\psi_{k\tau}=\sum_{l=1}^k\vp_{l\tau}$. Then,
\begin{equation}
\label{RIEq}
\o_{\psi_{k\tau}}^n=\on e^{f_\o-\mu\psi_{k\tau}+\frac1\tau\vp_{k\tau}}.
\end{equation}
Since the first step is simply $\o_{\vp_{\tau}}^n=\on e^{f_\o+(\frac1\tau-\mu)\vp_{\tau}},$ 
the Ricci iteration exists (uniquely) once a solution exists (uniquely) for \eqref{RCMEq} for $s=\mu-\frac1\tau$. 
Thus, much like for the Ricci flow,  a key point is to prove uniformity of the a priori estimates 
as $k$ tends to infinity. The convergence to the \KE metric then follows essentially by the monotonicity of the twisted K-energy
if the \KE metric is unique.

As noted above, our choice of the particular continuity path (\ref{RCMEq}) allows us to treat the continuity method and the 
Ricci iteration in a unified manner.  When $\mu\le0$,  our estimates for \eqref{RCMEq}, the arguments of \cite{R}, 
and the higher regularity developed in \S\ref{HigherRegSection}, 
imply the uniqueness, existence and edge smooth convergence of the iteration, for all $\tau$.
When $\mu>0$ the uniqueness of the (twisted) Ricci iteration was proven recently
by Berndtsson \cite{Bern}, and it follows from his result that whenever the twisted K-energy $E_0^\be$ is proper 
then also the \KE edge metric must be unique. Given this, our analysis here and in \cite{R} immediately 
implies smooth (in the edge spaces) convergence of the iteration for large enough times steps, more specifically, 
provided $\tau>1/\mu$ and $E_0^\be$ is proper, or else provided 
$\tau> 1/\a_{\O,\o}$, and $\a_{\O,\o}>\mu$, where $\a_{\O,\o}$
is Tian's invariant defined in \S\ref{TianMabuchiSubSec}
(note that by Lemma \ref{alphainvariantepsLemma} this assumption implies
$E_0^\be$ is proper).
As pointed out to us by Berman, given the results of \cite{R}, 
the remaining cases follow immediately in the same manner
by using one additional very useful pluripotential estimate
contained in \cite[Lemma 6.4]{BBGZ} and stated explicitly in \cite{Berm}, and recalled 
in Lemma \ref{BermanLemma} below. As already observed in \cite{BBEGZ}
this estimate gives in an elegant manner a uniform estimate on the oscillation
of solutions along the iteration, and is 
used in \cite{BBEGZ} to prove convergence of the twisted Ricci iteration
and flow in very general singular settings, smoothly away from the singular
set, and global $\calC^0$ convergence on the level of potentials.
Our result below, in the case $\mu>0$, is complementary to theirs
since it shows how to use their uniform estimate and our analysis to obtain
smooth convergence near the edge. 
We thank Berman for his encouragement to include this result here,
prior to the appearance of \cite{BBEGZ}.

To summarize, we have the following statement.

\begin{thm}
\label{RIThm}
Under the assumptions of Theorem \ref{ConicKEMainThm}, the Ricci iteration \eqref{RIEq} exists uniquely 
and subconverges in $\Dw\!\cap\! \calA^0$ to a \KE edge metric in $\calH_{\o}$.
Whenever the KEE metric is unique, in particular when there are no holomorphic vector fields tangent to $D$, 
the iteration itself converges.
\end{thm}
These function spaces encode the strongest possible convergence for this problem, and are defined next. 
The proof of Theorem \ref{RIThm} is given in Section \ref{ProofKESection}.

\subsection{Function spaces}

To conclude this section of preliminary material, we review the various function spaces used 
below. These are the `wedge and edge' H\"older spaces, as well as the spaces of conormal and 
polyhomogeneous functions necessary for our treatment of the higher regularity theory.  The 
wedge H\"older spaces are the ones used in \cite{J2,D}, and are naturally associated to the 
incomplete edge geometry. The edge spaces, introduced in \cite{Ma1} are also naturally 
associated to this geometry and have some particularly favorable properties stemming from their 
invariance under dilations. Using this, certain parts of the proofs below become quite simple. 
The wedge H\"older spaces, on the other hand, are closer to standard H\"older spaces, and 
indeed reduce to them when $\beta = 1$. They impose stronger regularity conditions. Since we 
use both types of spaces here, we describe many of the proofs below in both settings. 
This is important for applications and should also give the reader a better sense of their relative advantages.

Before giving any of the formal definitions below, let us recall that a H\"older space is naturally associated
to a distance function $d$ via the H\"older seminorm 
\[
[ u ]_{d; 0, \gamma} :=  \sup_{ p \neq p' \atop d(p, p') \leq 1} \frac{ |u(p) - u(p')|}{ d(p, p')^\gamma}.
\]
We only need to take the supremum over points with distance at most $1$ apart, since if $d(p, p') > 1$, 
then this quotient is bounded by $2 \sup |u|$.  The two different spaces below differ simply through
the different choices of distance function $d$.

\subsubsection{Wedge H\"older spaces}
First consider the distance function $d_1$ associated to the model metric $g_\beta$; note that it is clearly
equivalent to replace the actual $g_\beta$ distance function with any other function on $M \times M$ which
is uniformly equivalent, and it is simplest to use the one defined in the coordinates $(r,\theta,y)$ by
\[
d_1( (r,\theta,y), (r', \theta', y')) = \sqrt{ |r-r'|^2 + (r + r')^2| \theta - \theta'|^2 + |y-y'|^2}.
\]
Note that the angle parameter $\beta$ does not appear explicitly in this formula, but if we were to have included
it, there would be a factor of $\beta^2$ before $(r+r')^2|\theta - \theta'|^2$. 
This changes $d_1$ at most by a factor, so we may as well omit it altogether.

Now define the wedge H\"older space 
$\calC^{0,\gamma}_w\equiv\calC^{0,\gamma}_w(M)$ to consist of all functions $u$ on $M \setminus \D$ 
for which  
\[
||u||_{w; 0,\gamma} :=  \sup |u| + [u]_{d_1; 0, \gamma} < \infty.
\]

The spaces with higher regularity are defined using differentiations with respect to unit length vector fields
with respect to $g_\beta$; these vector fields are spanned by $\del_r$, $r^{-1}\del_\theta$ and $\del_{y_j}$. Thus
\[
\calC^{k,\gamma}_w(M) = \{ u:  \del_r^i (r^{-1}\del_\theta)^j \del_y^\mu u \in \calC^{0,\gamma}_w(M)\ 
\quad \forall\;\, i+j+|\mu| \leq k\}.
\]

There a few potentially confusing points about these spaces. The first is that the spaces $\calC^{k,\gamma}_w$ with $k > 0$
seem to depend on the choice of coordinates, or choice of frame. It is not hard to untangle 
the dependence or lack thereof, but since we only use these spaces when $k = 0$, this discussion
is relegated to another paper.  Second, it is worth comparing this definition with the equivalent one given 
in \cite{D}. As is evident from the definition above, the space $\calC^{0,\gamma}_w$ above does {\it not} 
depend on the cone angle parameter $\beta$ 
(at least so long as $\beta$ stays bounded away from $0$ and $\infty$). However, suppose we 
consider the (apparently) fixed function $f = |z_1|^a = \rho^a$ for some $a>0$ in terms of the 
original holomorphic coordinates. In terms of the cylindrical coordinates $(r,\theta,y)$, we 
have $f = \be^{a/\be} r^{a/\beta}$, and hence $f \in \calC^{0,\gamma}_w$ if and 
only if $a/\beta \geq \gamma$, i.e., $a \geq \beta \gamma$. Inequalities of this type appear in 
\cite{D}. This seems inconsistent with the claim that the H\"older space is independent of 
$\beta$; the discrepancy between these statements is explained by observing that the singular 
coordinate change {\it does} depend on $\beta$, and while the function $\rho^a$ is independent 
of 
$\beta$, its composition with this coordinate change is not.  Equivalently, if we pull back the 
function space $\calC^{0,\gamma}_w$ via this coordinate change, then we get a varying family of 
function spaces on $M$.  We prefer, however, to think of $M \setminus \D$ as a fixed but 
singular geometric object, with smooth structure determined by the coordinates $(r,\theta,y)$, 
and with a single scale of naturally associated H\"older spaces.

\subsubsection{Edge H\"older spaces}
Now consider the distance function $d_2$ associated to the {\it complete} metric
\[
\hat{g}_\beta := r^{-2} g_\beta = \frac{ dr^2 + |dy|^2}{r^2} + \beta^2 d\theta^2.
\]
As before, the distance $d_2$ is replaced by the metric 
\[
d_2( (r, \theta, y), (r', \theta', y') ) = (r+r')^{-1}\sqrt{ (r-r')^2 + (r + r')^2(\theta - \theta')^2 + |y-y'|^2}, 
\]
which is uniformly equivalent to it. It suffices to consider only $r, r' \leq C$. As before, no factor 
of $\beta$ is included.  

The H\"older norm $||u||_{e; 0, \gamma}$ is now defined using the seminorm associated to $d_2$. 
The higher H\"older norms are defined using unit-length vector fields with respect to $\hat{g}_\beta$,
which are spanned by $\{ r\del_r, \del_\theta, r \del_y\}$. The corresponding spaces of functions
for which these norms are finite are denoted $\calC^{k,\gamma}_e\equiv \calC^{k,\gamma}_e(M)$. 

The key property of this distance function is that it is invariant with respect to the scaling 
\[
(r,\theta,y) \mapsto (\lambda r, \theta, \lambda y)
\]
for any $\lambda > 0$.  The vector fields $r \del_r$, $\del_\theta$ and $r\del_y$ are also invariant with
respect to these dilations. This means that if $u_{\lambda,y_0}( r, \theta, y) = u( \lambda^{-1} r, \theta, \lambda^{-1} y + y_0)$,
then $|| u_{\lambda,y_0}||_{e; k, \gamma} = ||u ||_{e; k, \gamma}$.  
(We assume, of course, that both $(r,\theta, y)$ and $(\lambda^{-1} r, \theta, \lambda^{-1} y + y_0)$ lie in the domain of $u$.)  
One way to interpret this is as follows. Consider the annular region 
\[
B_{\lambda ,y_0} := \{(r,\theta,y):  0 < \lambda < r < 2\lambda,\ |y-y_0| < \lambda\},
\]
for $\lambda$ small. The image of this annulus under translation by $y_0$ and dilation by $\lambda^{-1}$
is the standard annulus $B_{1,0}$. Hence if $u$ is supported in $B_{\lambda, y_0}$, then $u_{\lambda, y_0}$
is defined in $B_{1,0}$ and $||u||_{e; k, \gamma} = ||u_{\lambda, y_0}||_{e; k,\gamma}$. 

For any $\nu\in\RR$, we also define weighted edge H\"older spaces 
\[
r^\nu C^{k,\gamma}_e(M) = \{u = r^\nu v: v \in  \calC^{k,\gamma}_e(M)\}.
\]

Although $\calC^{0,\gamma}_e(M)\subset L^\infty(M)$, elements of $\calC^{0,\gamma}_e(M)$ need not be continuous at $r=0$; 
an easy example is the function $\sin \log r$, which lies in $\calC^{k,\gamma}_e$ for all $k$.  On the other hand, 
elements of $r^\nu \calC^{0,\gamma}_e$ are continuous and vanish at $D$ if $\nu > 0$. 

\subsubsection{Comparison between the wedge and edge H\"older spaces}
Let us now comment on the relationship between these spaces. Since $r, r' \leq C$, we have
$d_1 \leq C^{-1} d_2$, and hence 
\[
|| u || _{e; k, \gamma} \leq C^{-\gamma} || u ||_{w; k,\gamma},
\]
or equivalently,
\begin{equation}
\label{ComparisonSpacesEq}
\calC^{k,\gamma}_w \subset \calC^{k,\gamma}_e.
\end{equation}

Elements in the wedge H\"older space are more regular than those in the edge H\"older space.  For example, unlike
elements of $\calC^{0,\gamma}_e$, elements of $\calC^{0,\gamma}_w$ are continuous up to $\D$. Moreover, if 
$u \in \calC^{0,\gamma}_w$, then $u(0,\theta, y)$ is independent of $\theta$ and lies in $\calC^{0,\gamma}(\D)$; 
by contrast, if $u \in \calC^{0,\gamma}_e$, then the `tangential' difference quotient $| u( r, \theta, y) - 
u(r, \theta, y')|/ |y-y'|^\gamma$ is bounded by $Cr^{-\gamma}$.

However, there is a direct relationship between the two spaces.  Define
\[
\calC^{0,\gamma}_w(M)_0 = \{u \in \calC^{0,\gamma}_w(M): u|_D = 0\}.
\]
Next, if $u \in \calC^{0,\gamma}_w(M)$, write $u_0 = u|_D \in \calC^{0,\gamma}(D)$. There exists an extension operator 
$\calC^{0,\gamma}(D) \ni u_0 \mapsto E(u_0) = U \in \calC^{0,\gamma}_w$. Fixing an identification of a neighborhood 
$\calV$ of $D$ with a  bundle of truncated cones over $D$ and collapsing the $S^1$ cross-sections of these conic fibers
yields a map $\calV \to D \times [0,r_0)$. Requiring any local $(r, \theta, y)$ coordinates to be coherent with this extension, 
we may choose $U$ to be independent of $\theta$ and to equal the `ordinary' harmonic  extension of $u_0$ in the $(r,y)$ 
coordinates, i.e., $(\del_r^2 + \Delta_y)U = 0$. Actually, the only properties of $U$ needed later are that 
$U \in \calC^\infty(M \setminus D)$, and 
\begin{equation}
|\del_r U| + |\del_y U| \leq C r^{-1 + \gamma}.
\label{extu0}
\end{equation}
For the harmonic extension, these bounds are a classical characterization of H\"older spaces, and this characterization of H\"older
spaces is explained carefully in \cite [Chapter V, Section 4.2]{Stein}. We note that it is straightforward
to choose such an extension in a less ad hoc way using the theory of edge Poisson operators developed in \cite{MV2}.

Now decompose any $u \in \calC^{0,\gamma}_w$ as
\[
u = U  + \tilde{u}, \qquad \tilde{u} \in \calC^{0,\gamma}_w(M)_0. 
\]
This is useful because the two components have different characterizations. We have already explained the relevant
regularity properties of $U = E(u_0)$. As for the other component, we assert that
\begin{equation}
\calC^{0,\gamma}_w(M)_0 = r^\gamma \calC^{0,\gamma}_e(M).
\label{eqw0e}
\end{equation}
To explain this, note that if $u$ lies in the space on the left, then $|u(r,y,\theta)| \leq C r^\gamma$,
so the function $v = r^{-\gamma} u$ is at least bounded. 
The proof of \eqref{eqw0e} is an elementary calculation checking that 
$||u||_{w; 0,\gamma} \leq C ||v||_{e; 0, \gamma}$ and $||v||_{e, 0, \gamma} \leq C ||u||_{w; 0, \gamma}$.  

There are certain advantages to using the edge H\"older spaces. First observe that if $\mu\in(0,1)$, then 
$r^{\mu} \in \calC^{0,\gamma}_w$ only when $\mu \leq \gamma$, while $r^{\mu} \in 
\calC^{0,\gamma}_e$ for all $\gamma \in (0,1)$. Furthermore, $r^\mu\not\in\calC^{k,\gamma}_w$ for any 
$k \geq 1$, but since $(r\del_r)^j r^\mu = \mu^j r^\mu$, we see that $r^\mu \in \calC^{k,\gamma}_e$ for all 
$k \geq 0$. In other words, the edge spaces more naturally accomodate noninteger exponents. This is 
important when dealing with singular elliptic equations because solutions of such equations typically 
involve noninteger powers of $r$, and it is quite reasonable to think of these solutions as being infinitely 
differentiable in a suitable sense.   One final point is that basic H\"older regularity theory for elliptic
differential edge operators is phrased in terms of the edge spaces; these are scale-invariant estimates.
The pseudodifferential parametrices in the edge calculus, discussed in Section 3 below, are most easily
shown to be bounded on edge spaces; their boundedness on the wedge spaces is a consequence of
that result. 

In the remainder of this article, whenever our discussion applies to both of these spaces,
we refer to the `generic' singular H\"older space $\calC^{k,\gamma}_s$, where
$$
\h{\sl $s$ equals either $w$ or $e$}.
$$
This $s$ should not be confused either with the parameter $s$ along the continuity path (\ref{RCMEq}),
nor with the holomorphic section $s$ defined in Lemma \ref{ReferenceMetricExistsLemma}.

\subsubsection{Conormal and polyhomogeneous functions}
The final set of spaces we define are the spaces of conormal and polyhomogeneous functions.
\begin{definition}
For any $\nu \in \RR$, define 
\[
\calA^\nu = \bigcap_{\ell \geq \ell' \geq 0} r^\nu \calC^{\ell,\gamma;\ell'}_e. 
\]
This is the space of conormal functions (of weight $\nu$). Thus elements of $\calA^\nu$ are bounded
by $r^\nu$, as are all of their derivatives with respect to the vector fields $r\del_r$, $\del_\theta$ and $\del_y$. 

Next, we say that $u \in \calA^\nu$ is polyhomogeneous, and write $u \in \calA^\nu_{\phg}$,  if it has an 
expansion of the form
\[
u \sim \sum_{j=0}^\infty \sum_{p=0}^{N_j} a_{jp}(\theta, y) r^{\sigma_j} (\log r)^p 
\]
where the coefficients $a_{jp}$ are all $\calC^\infty$, and $\{\sigma_j\}$ is a discrete sequence of complex
numbers such that $\mbox{Re}\, \sigma_j \to \infty$, with $\mbox{Re}\, \sigma_j \geq \nu$ for
all $j$ and $N_j = 0$ if $\mbox{Re}\,\sigma_j = \nu$. 
This expansion can be differentiated arbitrarily many times with the corresponding differentiated remainder.
We say that $u$ has a nonnegative index set if $u \in \calA^0_{\phg}$ and if any exponent $\sigma$ in
its expansion has $\Re\, \sigma = 0$, then $\sigma = 0$. Note finally that if 
$u \in \calA_{\phg}^0$, then $u$ is bounded, and if any such $u$ has nonnegative index set, then $u$
is continuous up to the boundary. 
\label{co-phg}
\end{definition}
These function spaces accomodate behavior typical for solutions of degenerate elliptic edge problems, 
e.g.\ functions like  $r^\sigma (\log r)^p a(\theta,y)$ where $a$ is smooth, $p$ is a nonnegative integer 
and $\sigma\in\CC$.  We remark that these spaces are the correct analogues of the spaces of infinitely 
differentiable functions in this context. Note that when $\be=1$, $\calA^0$ {\it does not} correspond to 
$\calC^\infty(M)$. In this setting, we make a distinction between functions which are infinitely differential 
(conormal) and those which have ``Taylor series'' expansions (i.e., are polyhomogeneous).  We
remark also that the expansions of polyhomogeneous functions are rarely convergent, but only give
`order of vanishing' type estimates. It is usually difficult to control the size of the neighborhood on
which such an expansion provides a good approximation.  

\section{Linear analysis}
\label{LinearSection}
We now present the key facts about the linear elliptic theory needed to handle the existence, 
deformation and regularity theory for canonical edge metrics.  We discuss this from two points 
of view, reviewing the estimates obtained by Donaldson in the wedge H\"older spaces, and also 
describing how to obtain analogous estimates in the edge H\"older spaces. These latter estimates
are obtained through the use of edge pseudodifferential operators, as developed in \cite{Ma1}. 
This methodology, part of the general framework of geometric microlocal analysis, yields the 
most incisive results for the class of degenerate elliptic operators which arise here, and as we shall see, 
there are numerous places below where the more delicate parts of the linear {\it and} nonlinear analysis 
needed to prove our main results here require this full theory. In other words, the 
use of the edge calculus in this paper is an essential feature, rather than simply
a more systematic way of rephrasing estimates analogous to those obtained by Donaldson. 

Fix a K\"ahler edge metric $g$ on $M$ with cone angle $2\pi\beta$ along the smooth divisor $D$;
we initially suppose that the metric $g$ is polyhomogeneous along $\D$, though this will be relaxed later.
For the rest of this section, we consider the operator $L = \Delta_g + V$ where $V$ is polyhomogeneous 
with nonnegative index set (and hence is bounded); in certain places below we extend certain results 
to the case where $g$ and $V$ are not polyhomogeneous, but have some given H\"older regularity.  

Our method is based on the realization that the Schwartz kernel of the Green operator for the 
Friedrichs extension of $L$ has a fairly simple polyhomogeneous structure, and knowing this structure, one
can read off the estimates we need. This Green function is a pseudodifferential edge operator. The paper \cite{Ma1} 
contains a detailed development of this class of operators, their mapping properties, and the elliptic parametrix 
construction in this calculus. We review various aspects of this theory now, at all times maintaining focus 
on the particular problem at hand. We give specific references to the appropriate results and sections of \cite{Ma1}
so as to guide the interested reader to the details of the proofs of the results we need.  We also 
recall Donaldson's estimates, explain the essential differences between his and the ones obtained here through the
edge theory, and also describe the differences between these two approaches to proving these estimates.

\subsection{Edge structures and edge operators}
We have already indicated that it can be advantageous to think of $M$ with a \K edge metric as being
a singular object, but it is more convenient to formulate the edge theory via structures on the manifold with
boundary obtained by taking the real blowup of $M$ along $D$. 

The general notion of an edge structure on a manifold with boundary $X$ is defined in terms of a space of vector 
fields $\calV_e(X)$ on that manifold, where we assume that $\del X$ is the total space of a fibration $\pi: \del X \to Y$ 
with fiber $F$. The space $\calV_e(X)$ consists of all smooth vector fields on $X$ which are unconstrained in 
the interior, but which lie tangent to the fibers at the boundary. In our setting, the manifold $X$ is obtained 
by taking the real blowup of $M$ around $\D$, so $\del X$ is the unit normal circle bundle $\mathrm{SN}\D$ over $\D$.  
To be more specific, the real blowup $X := [M; D]$ is by definition the disjoint union $(M \setminus D) 
\sqcup \mathrm{SN}\D$, endowed with the unique smallest topological and differential structure so that 
the lifts of smooth functions on $M$ and polar coordinates around $D$ are smooth. There is a natural 
smooth blowdown map $X \to M$. 

Before proceding, we note a subtlety here related to the fact that there are actually two natural smooth structures:
one is induced by the holomorphic coordinates $(z_1, \ldots, z_n)$, where $D = \{z_1 = 0\}$ locally, and the 
other by the cylindrical coordinate system $(r,\theta,y)$ defined earlier. Indeed, since $r = |z_1|^\beta/\beta$, 
functions smooth with respect to $z$ are not necessarily smooth with respect to $(r,\theta,y)$ and vice versa. 
These structures are, of course, equivalent via the coordinate transformation. However, perhaps the correct
perspective is that it is {\it not} the smooth structure on $X$, but rather the `polyhomogeneous structure', i.e., 
the ring of polyhomogeneous functions, which is fundamental. Indeed, the polyhomogeneous structure is preserved
by this coordinate change.  At any rate, for $X = [M, D]$, $\del X = \{\rho = 0\}$, where $\rho = |z_1|$ 
(or equivalently, $\{r = 0\}$ where $r$ is defined as above), and the $S^1$ fibers of $\del X$ are the level sets 
$\{y = \mbox{const.}\}$.  Functions on $X$ are polyhomogeneous if and only if they are polyhomogeneous 
with respect to either of the coordinate systems $(r,\theta,y)$ or $(\rho,\theta, y)$. Finally, and here
the difference between $\rho$ and $r$ is important, we define $\calV_e(X)$ to be generated by the vector fields 
$r\del_r$, $\del_\theta$ and $r\del_y$. 

Next, the space of differential edge operators $\Diff_e^*(X)$ consists of all operators which can be written locally 
as finite sums of products of elements of $\calV_e(X)$.  Thus again for $X = [M, \D]$, if $m \geq 0$, then the 
typical element of $\Diff_e^m(X)$ has the form 
\begin{equation}
A = \sum_{j+k + |\mu| \leq m} a_{j k \mu}(r,\theta,y) (r\del_r)^j \del_\theta^k (r\del_y)^\mu.
\label{exeo}
\end{equation}
We now restrict attention exclusively to the case $X = [M;\D]$ and $m=2$, though there are suitable versions of 
all of the main linear results below in the general edge setting. 

If $g$ is an incomplete edge metric on $M$ with cone angle $\beta$, then $L=\Delta_g + V$ can be written as in 
\eqref{modellap}, as the sum of a principal part and an error term $E$.  However, it is $A = r^2 L$ which is an edge 
operator in the sense we have just defined.

A differential edge operator is called elliptic if it is an `elliptic combination' of elements of $\calV_e(X)$, for
example a sum of squares of a generating set of sections plus lower order terms. This is the case for the operator
$A$ here; we refer to \cite[\S 2]{Ma1} for the coordinate invariant formulation of edge ellipticity, and for more
on edge vector fields and their dual one-forms. 

\subsection{Normal and indicial operators}
If $A$ is an elliptic edge operator, its mapping properties are governed not only by its ellipticity, but also by two model 
operators, the indicial and normal operators $I(A)$ and $N(A)$, respectively, which are defined at each point of $D$.  
While these may be defined invariantly, let us simply record here that for $A = r^2 L$, with $L = \Delta_g + V$, and
after a certain natural identification which we explain below, 
\begin{equation}
N(A) = (s\del_s)^2 + \beta^{-2} \del_\theta^2 + s^2 \Delta_w, \quad \mbox{and}\quad I(A) =  (s\del_s)^2 + \beta^{-2} \del_\theta^2, 
\label{norA}
\end{equation}
where $(s,w)$ are global affine coordinates on a half-space $\RR^+_s \times \RR^{2n-2}_w$, $\Delta_w:=\sum_{i=1}^{2n-2}
\frac{\del^2}{\del w_i^2}$, and  $\theta \in S^1_{2\pi}$ (the circle of radius $2\pi$).  Note that 
\[
N(A) = s^2 L_\beta, \quad \mbox{where}\ L_\beta = \del_s^2 + \frac{1}{s} \del_s + \frac{1}{\beta^2 s^2}\del_\theta^2 + \Delta_w
\]
is the Laplacian of the flat model metric $g_\beta$. 

Informally, $N(A)$ is obtained by dropping the error term $r^2 E$, freezing coordinates at a given point $y_0 \in D$, and replacing the 
local coordinates $(r,y)$ by global affine coordinates $(s,w) \in \RR^+ \times \RR^{2n-2}$. More invariantly, $N(A)$ is the limit of 
rescalings of $A$ by the group of dilations based at a point $y_0 \in Y$, and it acts on functions defined on the inward-pointing normal 
bundle of the fiber of $\del X$ over $y_0$.  The indicial operator $I(A)$ is even simpler: it is defined by dropping the terms 
in $N(A)$ which have the property that they map any function $s^a v(\theta, w)$ (with $v$ smooth) to a function which 
vanishes faster than $s^a$. The only term in the operator $N(A)$ above which is discarded for this reason is $s^2 \Delta_w$.  

In general, both $N(A)$ and $I(A)$ could depend on $y_0$ (for example, if the cone angle were to vary along $D$). Fortunately,
in our case of interest, this dependence is quite simple, and as we have indicated above, can effectively be normalized away. 
Indeed, from \eqref{pertgb}, the term $\del^2_{z_1 \bar{z}_1}$ is multiplied by the factor $F^{-1}$, which depends on all variables, 
but is independent of $\theta$ at $r=0$. The normal operator of $F r^2 L$, obtained by this rescaling procedure above, has 
the form \eqref{norA}, but initially the terms involving derivatives in $w$ are a second order constant coefficient
elliptic operator on $\RR^{2n-2}$, multiplied by $s^2$. A linear change of variables in $w$, depending smoothly on $y_0$, puts this
into standard form $s^2 \Delta_w$. Thus the correct statement is that the normal operators at different points $y_0$ can be identified 
with one another, and similarly for the indicial operator.

A number $a \in \CC$ is called an indicial root of $A$ (and also of $L$) if there exists a nontrivial function
$\psi(\theta)$ such that $I(A) s^a \psi(\theta) = 0$; thus, for $A = r^2 L$, 
\[
I(A) s^a \psi(\theta) =  (\beta^{-2}\del_\theta^2 + a^2) \psi = 0  
\Longleftrightarrow 
\begin{cases} a \in \left\{ j/\beta : j \in \mathbb Z\right\}, \\ 
\psi_j(\theta) = a_j \cos j\theta + b_j \sin j\theta,\ j \geq 1. 
\end{cases}
\]
The case $j=0$ here is special since $0$ is a `double' indicial root, so  $\psi_0(\theta) = 1$ and both $I(A) (s^0 ) = I(A)( s^0 \log s) 
= 0$. This is special to the case that $D$ has codimension $2$. These indicial roots are just the square roots of the eigenvalues 
of $-\beta^{-2}\del_\theta^2$, which leads to the observation that it is quite important that $\theta$ lies on a compact manifold 
(namely, $S^1$), since otherwise the spectrum, and hence the set of indicial roots, would not be discrete. Note also that for 
any $a \in \CC$ and $\psi(\theta,y) \in \calC^\infty$, it is always true that $A(r^a \psi(\theta,y)) = O(r^{a})$, but $a$ is an 
indicial root if and only if $A (r^a \psi_j) = O(r^{a + 1})$ and $\psi(\theta,y) = a(y)\psi_j(\theta)$ where $a(y)$ is essentially
arbitrary. 

\subsection{Mapping properties and the Friedrichs domain}
We next describe the basic mapping properties of $L$ on weighted H\"older spaces; these are the content of 
\cite[Corollary 6.4]{Ma1} applied to the operator $A = r^2 L$. 
\begin{prop}
The mapping 
\begin{equation}
\label{mappingpropeq}
L: r^\nu \calC^{\ell+2,\gamma}_e \longrightarrow  r^{\nu-2}\calC^{\ell,\gamma}_e
\end{equation}
has closed range if and only if $\nu \notin \{\frac{j}{\beta},  j \in \mathbb Z\}$.  
\label{mappingprop}
\end{prop}
\noindent The indicial roots are excluded as weights here because for these values, \eqref{mappingpropeq} 
does not have closed range. 

Although this Proposition, and indeed the emphasis in all of \cite{Ma1}, is on the Fredholm (and semi-Fredholm) theory of 
operators such as $A = r^2 L$, 
it is more relevant for us to focus on $L$ and its action as an unbounded operator acting on a space with a fixed
weight, rather than between two differently weighted spaces.  The main new issue from this point of view
is to select a self-adjoint extension; we assume throughout that the term of order $0$ is real-valued 
so that $L$ is a symmetric operator on the core domain $\calC^\infty_0( \MsmD )$. 
Rather than reviewing the well-known classical theory of self-adjoint extensions, we
recall simply that since $L$ is semibounded, there is always a distinguished self-adjoint realization called the
{\it Friedrichs} extension, which is defined using the coercive quadratic form
\beq
\label{FriedrichsEq}
\langle u, v \rangle = \int_M \left(\nabla u \cdot \overline{\nabla v} - V u \overline{v}\right)\, dV_g.
\eeq

We can identify the domain $\Dom_{\mathrm{Fr}}(L)$ of this Friedrichs extension explicitly. It can be shown, 
see \cite[\S 7]{Ma1}, that any $u \in \Dom_{\mathrm{Fr}}(L)$ has a `weak' partial expansion $u \sim u_0(y) + \tilde{u}$, 
where $\tilde{u} = O(r^\mu)$ for some $\mu > 0$ and $u_0$ may be a distribution of negative order, but is independent 
of $r$; this expansion is called weak because it only becomes an asymptotic expansion in the usual sense (in particular, 
with decaying remainder) provided both sides are paired with a test function $\chi(y)$ (depending only on $y$). Thus
$u \in \Dom_{\mathrm{Fr}}(L)$ if and only if 
\[
(r,\th) \mapsto \langle u(r,\th,\cdot), \chi(y) \rangle = \langle u_0(y), \chi(y) \rangle + \calO(r^\mu)
\] 
for any $\chi \in \calC^\infty(Y)$.  To distinguish this from behaviour of more general solutions, it is also proved 
in \cite[\S 7]{Ma1} that if $u$ is any $L^2$ solution to $Lu = f$ with $f \in L^2$, then this expansion could 
contain an extra term $\langle u_{01}(y), \chi(y) \rangle \log r$ on the right. Hence Friedrichs domain is characterized
by the requirement that the coefficient $u_{01}$ of $\log r$ vanish. We note that a principal source of the difficulties 
reported in \cite{Ma0} revolved around some technicalities encountered when working with these weak expansions. 

Henceforth we work exclusively with the Friedrichs extension of $L$, and denote it simply by $L$. It is straightforward to deduce 
using Hardy-type estimates that the domain $\Dom_{\mathrm{Fr}}(L)$ is compactly contained in $L^2$, which proves 
that $L$ has discrete spectrum as an operator on this space. Its nullspace is finite dimensional, with every element bounded 
and polyhomogeneous. Thus there is a uniquely defined generalized inverse $G$ determined by 
\[
L G = G L = \mbox{Id} - \Pi,
\]
where $\Pi$ is the finite rank orthogonal projector onto the nullspace. Essentially by definition, if $K$ is
the $L^2$ nullspace of $L$, then $\Dom_{\mathrm{Fr}}(L) = G( L^2(M, dV_g)) \oplus K$. 

We now shift to the analogous but less-standard discussion for $L$ acting between H\"older spaces. 
Proceeding by analogy with these $L^2$ definitions, we define the H\"older-Friedrichs domains
\[
\Dom_s^{0,\gamma}(L) := \{u \in \calC^{0,\gamma}_s: Lu \in \calC^{0,\gamma}_s\}, \quad 
\mbox{for}\ s = w \ \mbox{or} \ e.
\]
We claim that 
\[
\Dom^{0,\gamma}_e \subsetneq \calC^{2,\gamma}_e, \quad \mbox{and} \quad 
\Dom^{0,\gamma}_w\supsetneq \calC^{2,\gamma}_w.
\]
To see these inclusions, note first that if $u \in L^\infty$ and $Lu \in \calC^{0,\gamma}_e$, then a basic edge regularity theorem, 
proved using the mapping properties of the Green function $G$, see \cite[Proposition \S 3.7]{Ma1}, gives that $u \in \calC^{2,\gamma}_e$. 
However, if $u \in \calC^{2,\gamma}_e$, then $Lu$ is usually not bounded, and in fact typically we only have
$Lu \in r^{-2}\calC^{2,\gamma}_e$.  On the other hand, we explain below that $G: \calC^{0,\gamma}_w \to \calC^{0,\gamma}_w$,
so $\Dw \subset \calC^{0,\gamma}_w$. Moreover, functions in this domain lie in $\calC^{2,\gamma}( X \setminus \del X)$.  
However, as we describe more carefully below, $\Dom^{0,\gamma}_w$ contains the function $v = r^{1/\beta} e^{i\theta}$, and hence 
if $\beta > 1/2$, then for example $\del_r^2 r^{1/\beta} e^{i\theta} \notin L^\infty$.   In order to accomodate functions with these fractional 
exponents in the wedge spaces, we henceforth assume that
\begin{equation}
\mbox{if $s = w$, then $\gamma\in (0,1)\cap (0,\frac1\be-1]$}.  
\label{gamma}
\end{equation}
Note that this guarantees at least that $r^{1/\beta}e^{i\theta} \in \calC^{1,\gamma}_w$. 

The mapping 
\begin{equation}
L: \Dom_s^{0,\gamma}(L) \longrightarrow \calC^{0,\gamma}_s
\label{sFH} 
\end{equation}
is invertible up to a possible finite dimensional nullspace. We need to obtain a more explicit characterization 
of these singular `H\"older-Friedrichs' domains. The first step in this direction uses the Green
function $G$ exactly as in the $L^2$ theory: 
\begin{prop}
The nullspace $K$ of $L$ in $L^2(M, dV_g)$ coincides with the nullspace of $L$ in $\calC^{2,\gamma}_s$, and we have
\[
\Dom^{0,\gamma}_s(L) = G ( \calC^{0,\gamma}_s) \oplus K = \{u = Gf: f \in \calC^{0,\gamma}_s\} \oplus K,
\]
\label{char1}
\end{prop}
\begin{proof}
To prove the first assertion about nullspaces, apply \cite[Proposition 7.17]{Ma1} to see that an element
of either nullspace is polyhomogeneous, and lies in both $L^2$ and $\calC^{2,\gamma}_s$. 
Since $\calC^{0,\gamma}_s \subset L^\infty(M) \subset L^2(M, dV_g)$, the space on the right in the
displayed equation is well-defined.  If $u$ is in the space on the right, 
then clearly $Lu = f \in \calC^{0,\gamma}_s$. Conversely, if $u \in \calC^{0,\gamma}_s$, $f \in \calC^{0,\gamma}_s$ and $Lu = f$ 
distributionally, then $u$ is in the $L^2$ Friedrichs domain.  Clearly $L(u - Gf) = 0$, and since both
$u$ and $Gf$ are in $L^2$, we can write $u - Gf = v$ for some $v\in K$. 
\end{proof}

\subsection{Finer properties of functions in the H\"older-Friedrichs domain}
This last proposition sets the stage for the more detailed study of the regularity of functions in these domains.  In this 
subsection we first recall Donaldson's estimates, which characterize $\Dom_w^{0,\gamma}(L)$, and then state the 
corresponding results for $\Dom_e^{0,\gamma}(L)$, with the proofs deferred to the next subsection. 
We include some auxiliary regularity results which are used later.

The characterizations of the domains $\calD^{0,\gamma}_s$ will be given in terms of which derivatives lie in $\calC^{0,\gamma}_s$.
We also show that either of these domains are independent of the operator $L$ in the sense 
that they remain the same if we replace the polyhomogeneous K\"ahler edge metric $\o$ by any metric $\o_u$, 
where the K\"ahler potential $u$ itself only lies in $\Dw$, and  $V \in \calC^{0,\gamma}_s$. 

To gain a sense of where we are headed, recall that on a closed smooth manifold $M$, the $L^2$ Friedrichs domain 
of the Laplacian is equal to the Sobolev space $W^{2,2}(M)$. This follows from the basic elliptic estimates,
of course, but is also a consequence of the boundedness on $L^2$ of the Riesz potential operator 
$\nabla^2 \circ \Delta^{-1}$.  The corresponding Schwartz kernels are pseudodifferential operators of order 
$0$ and we can appeal to the general boundedness properties of this class of operators. 
Pseudodifferential theory has its origins in attempts to answer questions of this type. 

We follow a similar route here. The Green operator $G$ represents $\Delta^{-1}$, and the problem 
becomes one of determining which second derivatives applied to $G$ yield `Riesz potential' operators which 
are bounded on $\calC^{0,\gamma}_s$.  As we now describe, if $u \in \Dom_s^{0,\gamma}(L)$, then {\it not} 
every second derivative term appearing in the operator $L$, written as a real operator, applied to $u$ lands in 
$\calC^{0,\gamma}_s$. Donaldson's simple yet crucial observation \cite{D} is that this is not necessary!  As 
described in \S 2, the \MA operator decomposes into the sum of the individual $(1,1)$-type terms $g^{i\jbar}u_{i\jbar}$, 
each of which involve particular combinations of real second derivatives which in the notation of \S 2 are the expressions $P_{i\jbar} u$.  
The next proposition shows that these simple `monomial' operators characterize $\Dw$ 
in the sense that they provide an equivalent norm on $\Dw(L)$, for each fixed operator $L$ associated to an
edge metric, see also Corollary~\ref{equivdomains} below. 

\begin{prop}
\label{mainlinearprop}
Let $\gamma$ satisfy \eqref{gamma}, and recall the set of operators $\calQ^*$ in \eqref{defQ}. Then 
\[
\Dw(L)=\{u\in \calC^{0,\gamma}_w:\; Q_i u\in \calC^{0,\gamma}_w,\;\; \forall\, Q_i\in\mathcal Q^* \}.
\]
Equivalently, each of the maps 
\begin{equation}
Q_i \circ G:  \calC^{0,\gamma}_w \longrightarrow \calC^{0,\gamma}_w, \quad Q_i \in \mathcal Q^*,
\label{composedmaps}
\end{equation}
is bounded. 
\end{prop}
The proof of Proposition \ref{mainlinearprop} is described in \S\ref{mainlinearpropproofsubsection}. 
\begin{remark}
\label{SchauderRemark}
{\rm 
\noindent
(i)
This is essentially the same as the result in \cite{D} that if $u \in \Dw(L)$ then 
\[
\sum_{i \jbar} || g^{i\jbar}u_{i \jbar} ||_{w; 0,\gamma} \leq C ( || Lu||_{w; 0,\gamma} + ||u||_{\calC^0} ),  
\]
where the important point is that on the left we have a sum of norms rather than a norm of the sum.

Another useful way to phrase this involves the norms
\begin{equation}
||u||_{\Dw} = ||u||_{w; 0,\gamma} + \sum_{Q_i \in \mathcal Q^*} ||Q_i u||_{w;0,\gamma}.
\label{Dsnorm}
\end{equation}
Later we also use the seminorm $[\,\cdot\,]_{\Ds}$ 
defined by omitting the initial $||\,\cdot\,||_{s;0,\gamma}$ term. Proposition~\ref{mainlinearprop} implies that $||\cdot ||_{\Dw}$ 
is a Banach norm on $\Dw$. The space $\Dom^{0,\gamma}_w$ is the same as the space $\calC^{2,\gamma, \be}$ introduced in \cite{D}.}
\end{remark}

There is an equivalence of norms:
\[
C_1 ||u||_{\Dw} \leq ||u||_{\calC^0} + || Lu||_{w; 0,\gamma} \leq C_2 ||u||_{\Dw},
\]
where the constants $C_1, C_2$ depend on the sizes of the coefficients $g^{i\jbar}$. In our application below, 
these metric coefficients are determined by the solution $\vp$ of the \MA equation. We will prove a 
uniform $\calC^0$ bound on $\Delta \vp$ which ensures that these constants remain uniform  
across the family of edge metrics which arise in the continuity argument.

Since the complex operators $P_{i\jbar}$ are sums of the real operators $Q_i$, one direction of Proposition~\ref{mainlinearprop} 
is trivial: if $u$ and every $Q_i u$ lie in $\calC^{0,\gamma}_w$, then trivially $Lu \in \calC^{0,\gamma}_w$ since $Lu$ is just 
a sum of these terms with coefficients in $\calC^{0,\gamma}_w$ (or better). The other direction is proved by
showing that the compositions $Q_i \circ G$ are bounded operators. This is accomplished by Donaldson for the 
model problem by direct scaling methods. 
Our proof here uses that each of these Riesz operators are pseudodifferential edge operators of order $0$ and then invokes 
basic boundedness results for such operators.  We explain this more carefully in the next subsection. 

\medskip

{\rm 
\noindent 
(ii) 
It is at this point that the theory in edge and wedge H\"older spaces differs significantly. Indeed, it turns out that
it is {\it not} true that certain of the Riesz potentials $Q_i \circ G$ are bounded on $\calC^{0,\gamma}_e$; in particular, this
boundedness fails when $Q_i = \del^2_{y_j y_\ell}$. This can be seen by a specific example in local coordinates: the function
$u = y_k y_\ell \log (r^2 + |y|^2)$ lies in $\calC^{0,\gamma}_e$, and it is not hard to check that when $k \neq \ell$, then 
$L u \in \calC^{0,\gamma}_e$ as well.  However, $\del^2_{y_k y_\ell} u \sim \log (r^2 + |y|^2) \notin \calC^{0,\gamma}_e$.
It turns out that $\calC^{0,\gamma}_e$ is a borderline space for this boundedness. Note that we could equally well
have replaced the $y_i y_j$ prefactor in $u$ by $\Re\, z_i \overline{z_j}$; this is still harmonic, so $Lu \in \calC^{0,\gamma}_e$, 
and it is also still true that $\del^2_{i\jbar} u \sim \log r$. 

Despite this defect, the spaces $\calC^{k,\gamma}_e$ still serve some important roles in the arguments in the rest of this paper. 
}

\medskip

The following result is the key to the higher regularity theory. Recall that $\phi \in \PSH(M,\o)$ means that $\o_\phi > 0$.
\begin{cor}
Suppose that $\phi \in \Dw(L)\cap \PSH(M,\o)$ is a limit in the topology of $\Dw(L)$ of a sequence of polyhomogeneous potentials, 
and let $L_\phi := \Delta_{\o_\phi} + V$ for some $V \in \calC^{0,\gamma}_s$. (To make the notation coherent, write $L_0 = L$.) Then 
\begin{equation}
\calD^{0,\gamma}_w(L_0) = \calD^{0,\gamma}_w(L_\phi). 
\label{eqdomains}
\end{equation}
\label{equivdomains}
\end{cor}
We only claim this result when $\phi$ is a limit in the appropriate H\"older norm of polyhomogeneous functions, 
but not for an arbitrary element of $\Dw(L) \cap \PSH(M,\o)$.  This is the classical distinction between the `little'
and `big' H\"older spaces, and is adequate in our setting since we shall only need to apply this result when $\phi$
lies along the continuity path, and hence is a limit in this sense.  This raises an interesting analytic question on which 
we comment after the proof.
\begin{proof}
Observe that $L_\phi u = f$ can be rewritten as $(\Delta_{\o_\phi} -1) u = f - (V +1) u \in \calC^{0,\gamma}_w$,  
so we may as well assume that $V \equiv - 1$, which is a convenient choice because $\Delta_{\o_\phi} - 1$ is invertible. 
Letting $G_\phi = L_\phi^{-1}$, then the assertion is equivalent 
to the fact that the range of $G_\phi$ is independent of $\phi$ (in the allowable space of functions). 

When $\phi$ is polyhomogeneous, then the inverse $G_\phi$ is a pseudodifferential edge operator and 
\eqref{eqdomains} follows from Proposition~\ref{mainlinearprop}.

To prove the assertion for $\phi$ which is a limit of polyhomogeneous functions, note first that the inclusion $\subseteq$ 
is obvious. Indeed, if $u \in 
\calD^{0,\gamma}_w(L_0)$, then $Q u \in \calC^{0,\gamma}_w$ for every 
$Q \in \mathcal Q^*$. Now write
$(g_\phi)_{i\jbar} = g_{i \jbar} + \sqrt{-1} \phi_{i \jbar}$, so that $(g_\phi)^{i\jbar} = g^{i \jbar} + \eta^{i \jbar}$ for some 
$\eta^{i\jbar} \in \calC^{0,\gamma}_w$. Then $L_{\phi} u \in \calC^{0,\gamma}_w$ as well, i.e., $u \in \calD^{0,\gamma}_w(L_\phi)$. 

These two facts together show that $\calD^{0,\gamma}_w(L_\phi)$ remains the same when $\phi$ varies in the dense set of 
polyhomogeneous functions, but might potentially jump up when $\phi$ is a limit of polyhomogeneous potentials. 

For the converse, we claim that there is an a priori estimate
\begin{equation}
\sum_{Q \in \mathcal Q^*} ||Qu ||_{w;0, \gamma} \leq C \left( || L_\phi u||_{w; 0,\gamma} + ||u||_{w; 0,\gamma} \right),
\label{apriori}
\end{equation}
which holds {\it only} for functions $u \in \Dw(L)$ (but {\it not} $\Dom^{0,\gamma}_s(L_\phi)$), where the constant $C$ 
is locally uniform in $\phi$. To prove this, note that this estimate is true for $L_\phi$ when $\phi = 0$ and $u \in \calD^{0,\gamma}_w(L)$.
Freezing coefficients of a more general $L_\phi$ locally near any point $q \in M$, we can approximate this operator by 
one with polyhomogeneous coefficients, with an error term which has coefficients small in $\calC^{0,\gamma}_w$. We prove
the estimate in small coordinate charts for the nearby operator and then by perturbation for $L_\phi$ itself, absorbing the small 
coefficients into the left hand side. These local estimates can then be pasted together with a partition of unity. This method makes
clear that the constant $C$ depends only on the $\calC^{0,\gamma}_w$ norm of the coefficients $\eta^{i \jbar}$. 

To complete the argument, we must prove that for any $f \in \calC^{0,\gamma}_w$, the unique solution $u$ in
$\calD^{0,\gamma}_w(L_\phi)$ to $L_\phi u = f$ necessarily lies in $\calD^{0,\gamma}_w(L_0)$.  Note that another way
to phrase this is that we must prove that $L_\phi: \calD^{0,\gamma}_w(L_0) \rightarrow \calC^{0,\gamma}_w$ is surjective. 

Fix $f \in \calC^{0,\gamma}_w$, and let $u \in \calD^{0,\gamma}_w(L_\phi)$ solve $L_\phi u = f$. By local elliptic regularity (equivalently, 
boundedness of $G$ on edge H\"older spaces), we also know that $u \in \calC^{2,\gamma}_e$. We must show that $u \in 
\calD^{0,\gamma}_w(L_0)$.  

Choose a sequence $\phi_j$ of polyhomogeneous functions which converge to $\phi$ in $\calD^{0,\gamma}_w(L)$.  
For each $j$, there is a unique $u_j \in \calD^{0,\gamma}_w(L)$ with $L_{\phi_j} u_j = f$.  Applying \eqref{apriori} 
(with $L_{\phi_j}$) gives 
\[
\sum_{Q \in \mathcal Q^*} ||Qu_j||_{w; 0,\gamma} \leq C. 
\]

There is a subsequence of the $u_j$ such that each $Q u_j$ converges in some weaker norm $\calC^{0,\gamma'}_w$, and furthermore,
the $\calC^{0,\gamma}_w$ norms of these $Qu_j$ are uniformly bounded.  
There is a limiting function $u \in \calC^{0,\gamma}_w$ (even though the limit takes place in a weaker topology), and 
moreover each $Qu$ lies in $\calC^{0,\gamma}_w$, so in fact $u \in \calD^{0,\gamma}_w(L)$ and $L_\phi u = f$, as desired. 

This proves that $L_\phi$ restricted to $\calD^{0,\gamma}_w(L)$ is surjective, and hence finally that $\calD^{0,\gamma}_w(L_0) = 
\calD^{0,\gamma}_w(L_\phi)$.
\end{proof}
\begin{remark}
\rm
This result is equivalent to the assertion that $L_\phi: \calD^{0,\gamma}_w(L_0) \to \calC^{0,\gamma}_w$ is surjective. 
The latter statement is clearly an open condition for $\phi \in \calD^{0,\gamma}_w(L_0)$, which gives
the stronger conclusion that the result actually holds not just for $\phi$ lying in the closed subspace in $\calD^{0,\gamma}_w(L)$ 
consisting of limits of polyhomogeneous functions, but for all $\phi$ in some open neighbourhood of this subspace. 
This suggests, of course, that the result might be true for all $\phi \in \calD^{0,\gamma}_w(L) \cap \PSH(M,\o)$. 
We do not have a proof of this, but in any case this extension is not needed here. 
\end{remark}

\subsection{Pseudodifferential edge operators and their boundedness}
\label{mainlinearpropproofsubsection}
We describe the proof of Proposition~\ref{mainlinearprop} in this subsection. The main point is to describe 
the structure of the Green operator $G$, or more specifically, the precise pointwise structure of its Schwartz 
kernel $G(z,z')$.  This structure is then used to bound the integrals 
\begin{equation}
Q_i u(z) = \int_X Q_i G(z,z') f(z')\, dV_g(z'), \quad f=Lu \in \calC^{0,\gamma}_w. 
\label{integral}
\end{equation}
The fact which makes this work is that the operators $Q_i \circ G$ are pseudodifferential edge operators; most of these compositions
are of weakly positive type, cf.\ Definition~\ref{nonnegindset}, and \cite[Proposition 3.27]{Ma1} gives their boundedness on the edge
H\"older space $\calC^{0,\gamma}_e$.  We provide an extension of that argument to prove that each $Q_i \circ G$ is
bounded on $\calC^{0,\gamma}_w$ as well.   While this replicates the results of \cite{D}, the refined structure of these
operators proved here is an important ingredient in the higher regularity theory. 

More broadly, we describe why the Schwartz kernel $G$ has a polyhomogeneous structure, and show how one can
deduce from this that most of the Riesz potentials $Q_i \circ G$ are in the edge calculus and of weakly positive type. 
The boundedness of such weakly positive edge operators on edge and wedge H\"older spaces is a basic feature 
of the edge calculus. Donaldson derives the polyhomogeneous structure of the Green function for the flat model 
problem $G_\beta$ by explicit calculation, and then proves the H\"older estimates on the wedge spaces by hand. 
The edge calculus is a systematization of the perturbation arguments which allow one to pass from this flat model 
to the actual curved problem, but one which yields in particular the polyhomogeneous structure of the Green function 
for the curved problem, which plays a significant role for the higher regularity theory. 

The edge calculus $\Psi^*_e(X)$ is a space of pseudodifferential operators on $X$, elements of 
which have degeneracies at $\del X$ similar to the ones exhibited by differential edge operators 
as in \eqref{exeo}. We use $X$ systematically now rather than $M$ since it is more natural
for the descriptions below to work on a manifold with boundary. This space of operators is 
large enough to contain not only all differential edge operators $A$, but also parametrices and generalized 
inverses for the elliptic operators in this category, as well as for incomplete elliptic 
edge operators like $L = r^{-2}A$. The term `calculus' (rather than algebra) is used to indicate 
that $\Psi^*_e(X)$ is almost closed under composition, with the caveat that not every pair of elements may be 
composed due to growth properties of Schwartz kernels in the incoming and outgoing variables 
which prevent the corresponding integrals from converging.

An element $B \in \Psi^*_e(X)$ is characterized by specific regularity properties of its Schwartz 
kernel $B(z,z')$ as a distribution on $X \times X = X^2$; the superscript $*$ is a placeholder for a 
set of indices which indicate the singularity structure of this distribution in various 
geometric regimes in $X^2$. By definition, any such $B(z,z')$ is the pushforward of a distribution $K_B$ defined 
on a space $X^2_e$, called the edge double space, which is a resolution of $X^2$ obtained by 
performing a (real) blow-up of the fiber diagonal (defined below) of $(\del X)^2$. This distribution $K_B$ has a 
standard pseudodifferential singularity along the lifted diagonal (by which we mean a 
polyhomogeneous expansion in powers of the distance to this submanifold), as well as 
polyhomogeneous expansions at all boundary hypersurfaces of $X^2_e$, and product-type 
expansions at the higher codimension corners. We have defined polyhomogeneity on manifolds with 
boundary earlier, and will extend this to manifolds with corners below.  The detailed notation 
$B \in \Psi^{m, k, E_{\rf}, E_{\lf}}_e(X)$ records the pseudodifferential order $m$ along the 
diagonal and the exponent sets in the expansions at the various boundary faces.  We explain 
this in more detail now, but all of this is described fully in \cite[Sections 2-3]{Ma1}. 

We first construct the blowup $X^2_e$. The product $X^2 = X \times X$ is a manifold with corners up to codimension two.
The corner $(\del X)^2$ has a distinguished submanifold, denoted $\mathrm{fdiag}_{\del X}$, which is the fiber diagonal. 
This consists of the set of points $(p, p')$ such that $\pi(p) = \pi(p')$.  This is blown up normally, resulting in 
a space $[X^2, \mathrm{fdiag}_{\del X}]$ which by definition is the edge double space $X^2_e$. Using local coordinates 
$(r,\theta,y)$ on the first factor of $X$ and an identical copy $(r',\theta', y')$ on the second copy, the corner is the 
submanifold $\{r = r' = 0\}$, and $\mathrm{fdiag}_{\del X} = \{r = r' = 0, y = y' \}$. The blowup may be thought of as 
introducing polar coordinates around this submanifold: 
\begin{multline*}
R = | (r, r', y - y')| \geq 0, \ \omega = R^{-1} (r, r', y-y') \in S^{2n-1}_+ = \{\o=(\omega_1, \omega_2, \tilde{\omega}): |\omega| = 1, 
 \omega_1, \omega_2 \geq 0\}, 
\end{multline*}
supplemented by $y', \theta, \theta'$ to make a full coordinate system. Thus $X^2_e$ has a new boundary hypersurface,
$\{R = 0\}$, called the `front face' $\ff$, and the lifts of the two original boundary hypersurfaces, 
$\{\omega_1 = 0\}$ and $\{\omega_2 = 0\}$, called the right and left faces, $\rf$ and $\lf$, respectively.
We write defining functions for these faces as $\rho_\ff$, $\rho_\rf$ and $\rho_\lf$.  The diagonal of $X^2$ lifts to 
the submanifold $\mathrm{diag}_e = \{\omega = (1/\sqrt{2}, 1/\sqrt{2}, 0), \theta = \theta', R \geq 0 \}$.

Here are some motivations for this construction. First, Schwartz kernels of pseudodifferential operators 
are singular along the diagonal in $X^2$, but the fact that this diagonal intersects the corner nontransversely makes 
these singularities hard to describe near this intersection. By contrast, the lifted diagonal $\mathrm{diag}_e$ intersects 
the boundary of $X^2$ only in the interior of $\ff$, and this intersection is transversal; this turns out to allow for a 
simpler description of the singularity of the Schwartz kernel there. Another point is that $X^2_e$ 
captures the homogeneity under dilations inherent in this problem.  The flat model operator
\[
L_\beta = \del_r^2 + r^{-1}\del_r + (\beta r)^{-2}\del_\theta^2 + \Delta_y,
\]
on the product space $[0,\infty)_r \times S^1_\theta \times \RR^{2n-2}_y$ is homogeneous of order $-2$ with respect to 
the dilations $(r,\theta, y) \mapsto (\lambda r, \theta, \lambda y)$, and is also translation invariant in $y$. It follows 
that the Schwartz kernel $G_\beta(z,z')$ of the inverse for the Friedrichs extension of $L_\beta$ commutes with translations 
in $y$, thus depends only on the difference $y-y'$ rather than $y$ and $y'$ individually, and is homogeneous of order 
$-2n + 2$ in the sense that 
\begin{equation}
G_\beta( \lambda r, \lambda r', \lambda (y-y'), \theta, \theta') = \lambda^{-2n+2} G_\beta (r, r', y-y', \theta, \theta').
\label{homogG}
\end{equation}
In the polar coordinate system above, this simply says that
\begin{equation}
G_\beta(r, r', y-y', \theta, \theta') = \mathcal G_\beta( \omega, \theta, \theta') R^{-2n+2},
\label{Gpolar}
\end{equation}
or equivalently, that $G_\beta$ lifts to the double space $(\RR^+ \times S^1 \times \RR^{2n-2})^2_e$ and
decomposes as a product of the simple factor $R^{-2n+2}$ and the `angular part' $\mathcal G_\beta$.  
A further analysis shows that $\mathcal G_\beta$ has a singularity at $\omega = (1/\sqrt{2}, 1/\sqrt{2}, 0)$ 
and polyhomogeneous expansions along the side faces $\{\omega_1 = 0\}$ and $\{\omega_2 = 0\}$. 

We now recall the general definition of polyhomogeneity on manifolds with corners. We do this only on the model orthant 
$\calO = (\RR^+)^k \times \RR^\ell$, with linear coordinates $(x_1, \ldots, x_k, y_1, \ldots, y_\ell)$, 
but this definition is coordinate-invariant, hence translates immediately to arbitrary manifolds with corners. 
First, let
\[
\calV_b(\calO):= \mathrm{span}_{\calC^\infty} \{ x_1 \del_{x_1}, \ldots, x_k \del_{x_k}, 
\del_{y_1}, \ldots, \del_{y_\ell}\}
\]
be the space of all smooth vector fields tangent to all boundaries of this space. We may as well assume that all 
distributions are supported in a ball  $\{|x|^2 + |y|^2 \leq 1\}$. If $\nu = (\nu_1, \ldots, \nu_k) \in \RR^k$, 
then $u$ is conormal of order $\nu$, $u \in \mathcal A^\nu(\calO)$, if
\[
V_1 \ldots V_j u \in x^\nu L^\infty(\calO)\ \forall\, j \geq 0\ \mbox{and for all}\ V_i \in \calV_b(\calO).
\]
Next, $u$ is polyhomogeneous if near the origin in $\calO$, $u$ has an expansion of the form
\[
u \sim \sum_{\ell} \sum_{|p| \leq N_j} a_{\ell, p}(y) x^{\gamma^{(\ell)}} (\log x)^p,
\]
where $\{\gamma^{(\ell)} = (\gamma_1^{(\ell)}, \ldots, \gamma_k^{(\ell)})\}\}$ is a sequence of $k$-tuples
in $\CC^k$ with $\mbox{Re} \gamma_j^{(\ell)} \to \infty$ as $\ell \to \infty$, $x^{\gamma} = x_1^{\gamma_1}
\ldots x_k^{\gamma_k}$ and $(\log x)^p:= (\log x_1)^{p_1} \ldots (\log x_k)^{p_k}$, with each $p \in \NN_o^k$. 
The coefficients $a_{\ell, p}(y)$ are smooth.  As with polyhomogeneous expansions for functions near a codimension
one boundary, these sums are not usually convergent, but may still be differentiated term-by-term, etc.
If $u$ is polyhomogeneous in this sense, then each coefficient of its expansion at any one of
the boundary hypersurfaces or corners is polyhomogeneous on that face. We 
associate to such an expansion an index family $E = \{E^{(\ell)}\}$, $\ell = 1, \ldots, k$, 
consisting of all pairs of multi-indices $\{(\gamma, p)\}$ of exponents which occur in this 
expansion, and denote by $\calA_{\phg}^E$ the space of all such distributions. As in the 
codimension one case, we say that $u \in \calA_{\phg}^0$ if $u$ is polyhomogeneous and if each 
index set $E^{(\ell)}$ is greater than or equal to $0$ in the sense of Definition \ref{co-phg}. We also 
write the simple index set $\{(\gamma + \ell, 0): \ell \in \mathbb N_0\}$ simply as $(\gamma)$; 
thus $\calA_{\phg}^{(\gamma)} = x^\gamma \calC^\infty$, i.e., $u = x^\gamma v$ where $v$ is 
$\calC^\infty$ up to that face. (Even more specifically, a function which is smooth in the 
traditional sense up to the boundary and corners has index set $(0)$.)

We now define the space of pseudodifferential edge operators on $X$.
\begin{definition}
We say that $B \in \Psi^{m, \eta, E_{\rf}, E_{\lf}}_e(X)$ if the Schwartz kernel of $B$ is the pushforward from $X^2_e$ to $X^2$ of a distribution 
$K_B$ on $X^2_e$ which has the following properties.  $K_B$ decomposes as a sum $K_B^{(1)} + K_B^{(2)}$ where 
$K_B^{(1)} = \rho_{\ff}^{-2n+\eta}K_B'$ is supported in a neighbourhood of $\mbox{diag}_e$ which does not intersect the side 
faces, where $K_B'$ has a classical pseudodifferential singularity of order $m$ along this lifted diagonal which is smoothly 
extendible across $\ff$. (This simply says that $K_B'$ is smooth up to $\ff$ away from $\mbox{diag}_e$, and
that the singularity along this lifted diagonal extends across $\ff$ so that it remains conormal on the `continuation'
of the diagonal across this face.) The term $K_B^{(2)}$ is required to be polyhomogeneous on $X^2_e$ with
index sets $(\eta-2n)$ at $\ff$, $E_{\rf}$ at $\rf$ and $E_{\lf}$ at $\lf$, and vanishes in a neighbourhood of the lifted diagonal. 
\end{definition}  
This decomposition of $K_B$ into two terms isolates that part of $K_B$ which contains the diagonal singularity, and emphasizes
the key fact that this singularity is uniform up to $\ff$. The shift of the order at the front face by $-2n$ is an artifact of
a normalization: indeed, the volume form $dV_g$ is uniformly equivalent to $r dr d\theta dy$, so the 
the Schwartz kernel of the identity operator relative to this measure is a smooth nonvanishing multiple of
$r^{-1} \delta(r-r') \delta(\theta - \theta') \delta(y-y')$. Since$\delta(r-r')$ and $\delta(y-y')$ are homogeneous of 
degrees $-1$ and $2-2n$, respectively, this Schwartz kernel is homogeneous of order $-2n$, and we simply
want this to match with the fact that $\mbox{Id}$ is an operator of order $0$. 

Finally, we may state the basic structure theorem for the Green function of $L$.
\begin{prop}
\label
{BasicGreenProp}
Let $g$ be a polyhomogeneous edge metric with angle $\beta$ along $D$ and $L = -\Delta_g + V$ where $V$ is 
polyhomogeneous and bounded on $X$, and suppose that $G$ is the generalized inverse to the Friedrichs extension of $L$. 
Then $G \in \Psi_e^{-2, 2, E, E}(X)$, where the index set $E$ is determined by the indicial 
roots of $L_\beta$ and
by the index sets of $g$ and $V$. In particular, if $g$ and $V$ are smooth (i.e., both have index set $(0)$ at $\del X$), then
\begin{equation}
E \subset \{ (j/\beta + k, \ell): j, k, \ell \in \mathbb N_0\ \mbox{and}\ \ell = 0\ \mbox{for}\ j+k \leq 1,\ (j,k,\ell) \neq (0,1,0) \}.
\label{indset}
\end{equation}
Moreover, if $g$ and $V$ are polyhomogeneous with index set contained in
the index set \eqref{indset}, then the index set $E$ for $G$ is also contained in the 
index set \eqref{indset}. 
\label{strG}
\end{prop}
\begin{remark}
{\rm
The fact that $G$ has the same index set $E$ at the left and right faces is natural since $G$ is symmetric.
The index set $E$ may be slightly more complicated when $\beta \in \mathbb Q$ since in that case $j/\beta$ 
can equal a positive integer for certain $j$, and this creates extra logarithmic factors in the expansion (i.e., elements of 
$E$ of the form $(k,1)$), but these all occur sufficiently high in the expansion -- in $\mbox{Re}\, \zeta \ge 1$ --
and hence do not affect the considerations below.  These log terms are absent if $g$ is an orbifold metric.

Despite the seemingly elaborate language needed to state this result, this structure theorem for $G$ includes the one given by 
Donaldson \cite{D} for the model Green function, but the key advantage is that we have this same refined structure 
for the curved operator $G$ too.
}
\end{remark}
 
This result is one of the main conclusions of \cite{Ma1}: it is simply the elliptic parametrix construction in the edge calculus, 
modified slightly to accomodate the minor differences for Laplacians of incomplete rather than complete edge metrics. 
As with any parametrix construction, the first main step is to obtain detailed information about the solution operator for the 
model problem $\Delta_{g_\beta}$, or in other words about the model Green function $G_\beta$. This is the
technical core, and the rest of the argument uses pseudodifferential calculus to write the Green function for $L$ 
as a perturbation of $G_\beta$. The specific information we need to obtain, then, is that the Schwartz kernel of 
$G_\beta$ has the same polyhomogeneous structure as in the statement of Proposition \ref{BasicGreenProp}. 
This may be approached in a few ways. The first, appearing in \cite[\S\S 4, 5]{Ma1}, is to take the Fourier transform 
in $y$, thus reducing $\Delta_\beta$ to the family of operators 
\[
\widehat{\Delta}_\beta = \del_r^2 + r^{-1}\del_r + (\beta r)^{-2} \del_\theta^2 - |\eta|^2
\]
on $\RR^+ \times S^1$ where $\eta$ is the variable dual to $y$. 
To keep track of the dependence on $\eta$, set 
$s = r |\eta|$ to convert this to
\[
|\eta|^{2} \left(\del_s^2 + \frac{1}{s} \del_s + \frac{1}{\beta^s s^2} \del_\theta^2 - 1\right). 
\]
This can be analyzed explicitly by separation of variables. Chasing back through these transformations yields a 
tractable expression for $G_\beta$. The equivalent approach in \cite{D} is to write $G_\beta$ as an integral over 
$0 < t < \infty$ of the heat kernel $\exp (t \Delta_\beta)$.  This heat kernel is the product of the heat kernel on 
the model two dimensional cone with cone angle $2\pi \beta$ and the Euclidean heat kernel on $\RR^{2n-2}$. 
The former of these is known classically, albeit as an infinite sum involving Bessel functions, see \cite{D} and \cite{Mooers}, 
while the latter is the standard Gaussian.  Either method requires about the same amount of work. 

A minor point in the statement of Proposition~\ref{strG} which turns out to be important below is the fact that 
the index set $E$ does not contain the element $(1,0)$; in other words, the monomials $r$ and $r'$ do not appear 
in the expansion of $G$ at the left and right faces, respectively. This can be explained as follows. As a distribution, 
$G(r,\theta, y, r', \theta', y')$ satisfies
\[
L G = r^{-1} \delta(r-r') \delta(\theta - \theta') \delta(y-y').
\]
Restricting to the interior of $\rf$, away from the front face, we see that $L G = 0$ there.  Since we know at this
point that $G$ is polyhomogeneous, we can calculate formally, i.e.\ letting $L$ act on the series expansion and collecting
terms with the same powers. It is then easy to see that $L$ cannot annihilate the term $a(\theta, y) r$; indeed, referring
to \eqref{decompLap}, the only possible problematic term in 
$
L (a(\theta, y) r)$ is 
$
(
\del_r^2 + r^{-1} \del_r+ \beta^{-2} r^{-2} \del_\theta^2)
(a r)
= 
r^{-1} (\beta^{-2}\del_\theta^2 +1) a 
$. 
However, since $\beta < 1$, there is no other term
in the expansion which could cancel this, and it is 
impossible for $\beta^{-2}a_{\theta \theta} + a$ 
to vanish unless $a \equiv 0$.  This proves the claim. 

\begin{definition}
An index set $E$ is called nonnegative if, for any $(\gamma, p) \in E$, $\Re \gamma \, \geq 0$ and if
$\Re \, \gamma = 0$, then $(\gamma,p) = (0,0)$. 

An operator $B \in \Psi^{m,\eta,E,E'}_e(X)$ is said to be of nonnegative type if $m \leq 0$, $\eta \geq 0$, and
both $E$ and $E'$ are nonnegative. It is called weakly positive if, in addition, either $\eta > 0$ or $E > 0$ (thus the
excluded case is when $\eta=0$ and $E$ contains $(0,0)$.  We shall always assume that if the leading
exponent at $\lf$ is $0$, then the corresponding coefficient does not depend on $\theta$. 
\label{nonnegindset}
\end{definition}

We now state the basic boundedness theorem needed in the proof of Proposition \ref{mainlinearprop}.
\begin{prop} Let $B \in \Psi^{m,r,E, E'}_e(X)$. If $B$ is of weakly positive type, so in particular $m \leq 0$, and 
if the first nonzero element of the index set for $B$ at $\lf$ is greater than $1$, then 
\[
B: \calC^{\ell,\gamma}_w(X) \longrightarrow \calC^{\ell,\gamma}_w(X)
\]
is bounded for any $\ell \in \NN_0$. 
\label{bddpsido}
\end{prop}
\begin{proof} The order of the singularity along the diagonal is not the key issue here, so we assume that $m=0$.
We proceed in a series of increasingly general steps. For the first, suppose that the index set of $B$ at $\lf$ is strictly
positive and $\mu \geq 0$ is strictly smaller than all elements in this index set, and we consider boundedness on
weighted edge H\"older spaces. The result here is \cite[Proposition 3.27]{Ma1}, which asserts that
\begin{equation}
B: r^\mu \calC^{\ell,\gamma}_e(X) \longrightarrow r^\mu \calC^{\ell,\gamma}_e(X)
\label{bddwpos}
\end{equation}
is bounded for every $\ell \geq 0$. 

Let us recall how \eqref{bddwpos} is proved. Decompose $B$ into a sum of two operators, $B_1 + B_2$, where the Schwartz 
kernel of $B_1$ is supported near the lifted diagonal $\diag_e \subset X^2_e$, and carries the full pseudodifferential singularity, 
while that of $B_2$ is polyhomogeneous on $X^2_e$ and has the same index sets as $B$ at $\lf$, $\ff$ and $\rf$. 
The boundedness of $B_1: r^\mu \calC^{\ell,\gamma}_e \to r^{\mu +\eta}\calC^{\ell,\gamma}_e\hookrightarrow r^\mu \calC^{\ell,\gamma}_e$ 
is a consequence of the approximate dilation 
invariance of both the Schwartz kernel of $B_1$ and of the edge H\"older norm, as well as the standard local boundedness 
on H\"older spaces for (ordinary, nondegenerate) pseudodifferential operators. Rigorously, this is done using a Whitney cube 
decomposition and scaling arguments. This uses only the conditions $m \leq 0$ and $\eta \geq 0$.  

As for the other term, noting that $(r\del_r)^i (r\del_y)^j \del_\theta^\ell B_2$ has the same structural properties and index 
sets as $B_2$ itself, we see that it suffices to prove that if $\tilde{f} \in r^\mu \calC^{\ell,\gamma}_e$, then $|B_2 \tilde{f}| \leq Cr^\mu$, 
since then every $|(r\del_r)^i (r\del_y)^j \del_\theta^k B_2 \tilde{f}| \leq Cr^\mu$ as well.   

For convenience, and since slight variants of this calculation are used several places below, here is a precise statement of
a special case of the pushforward theorem, \cite[Proposition A.18]{Ma1}. 
Suppose that $H \in \Psi^{-\infty, \eta, E, E'}_e$ has Schwartz kernel which is pointwise nonnegative, 
and where $E' + \mu > -2$. Suppose the smallest element in the index set $E$ is $(\lambda_0, 0)$. Then 
\[
H (r^\mu) = \int H(r,\tilde{r},y,\tilde{y},\theta, \tilde{\theta}) \tilde{r}^\mu \, \tilde{r} d\tilde{r}d\tilde{y} d\tilde{\theta}
\]
is polyhomogeneous with leading term $a r^{\hat{\lambda}}$, where $\hat{\lambda} = \min\{\lambda_0, \mu + \eta\}$, provided
$\mu + \eta \neq \lambda_0$, or with leading term $a r^{\lambda_0} \log r$ if $\lambda_0 = \mu + \eta$.   More generally, if
$|f| \leq C r^\mu$, then $|Hf| \leq C r^{\hat{\lambda}}$ when $\lambda_0 \neq \mu + \eta$ and $|Hf| \leq C r^{\lambda_0}\log r$
when $\lambda_0 = \mu + \eta$. 

Returning to the problem above, applying this result shows that because of the assumption about the index set of $B$ 
and the fact that $\mu \geq 0$, we obtain that $|B_2 (r^\mu)| \leq C r^\mu$ as desired. 

Next, when the leading exponent in the expansion of $B$ at $\lf$ is $0$, then the sharp statement is that 
\[
B: r^\mu \calC^{\ell,\gamma}_e \longrightarrow \calC^{0,\gamma}_w \cap \calC^{\ell,\gamma}_e. 
\]
It suffices to consider just the contribution from $B_2$ and to show that $B_2 \tilde{f} \in \calC^{0,\gamma}_w$.  For this 
particular argument it is actually enough to assume that $B$ is of nonnegative type, so let us assume that the 
leading exponent in the expansion of $B_2$ at $\ff$ is $0$.  As a first step note that since $|\tilde{f}| \leq C r^\gamma$, 
then by the pushforward theorem, $|B_2 \tilde{f}| \leq C$.  To improve this, note that $r\del_r$ annihilates the leading term
$a_0 r^0$ of $B_2$ at $\lf$, so $r\del_r B_2$ is weakly positive with leading exponent at $\lf$ strictly greater than $\gamma$.
Applying the pushforward theorem for this operator gives $|r\del_r B_2 \tilde{f}| \leq C r^\gamma$, or equivalently, 
$|\del_r B_2 \tilde{f}| \leq C r^{\gamma-1}$. We next observe that $r\del_y B_2$ has index set greater than or equal to $1$ at 
$\lf$, so $|r\del_y B_2 \tilde{f}| \leq C r^\gamma$, i.e., $\del_y B_2 \tilde{f}| 
\leq r^{\gamma-1}$.  Finally, by hypothesis, $r^{-1} \del_\theta B_2$ has positive index set at $\lf$ and has order $\eta = -1$ at $\ff$, 
hence $|r^{-1}\del_\theta B_2 \tilde{f} \leq C r^{\gamma-1}$. These estimates on the derivatives yield, by a standard argument, that 
$B_2 \tilde{f} \in \calC^{0,\gamma}_w$. 

Finally, let us turn to our actual goal, when $f \in \calC^{\ell,\gamma}_w$.  First suppose that $\ell = 0$ and fix $f \in \calC^{0,\gamma}_w$. 
We wish to prove that $Bf \in \calC^{0,\gamma}_w$ as well.  Since $B$ is of weakly positive type, we have that either the
index set $E$ at $\lf$ is greater than $\gamma$ or else the order $\eta$ at $\ff$ is positive. In our application, $\eta \geq 1$,
but in fact any $\eta > \gamma$ is sufficient for this estimate. The key to this argument is the decomposition $f = \hat{f} + 
\tilde{f}$ from \S 2.6.3, where $\hat{f}$ is the harmonic extension of $f_0 = f|_D$ and $\tilde{f} \in r^\gamma \calC^{0,\gamma}_e$. 
We have just proved that $B \tilde{f} \in \calC^{0,\gamma}_w$, so the remaining issue is to study the behaviour of
$\hat{u} = B \hat{f}$, assuming just that $B$ is of weakly positive type.  Recalling now the bounds \eqref{extu0}, we first estimate that
$\del_y (B \hat{f}) = B \del_y \hat{f} + [\del_y, B] \hat{f}$.  (Because $\hat{f}$ is smooth in the interior, we no longer need
to isolate the diagonal contribution of $B$.) By \cite[Proposition 3.30]{Ma1}, the commutator $[\del_y, B]$ is an operator
of the same type and order as $B$, so in particular $|[\del_y , B] \hat{f}| \leq C |\log r|$.  On the other hand, using \eqref{extu0},
$| B \del_y \hat{f}| \leq |B (\tilde{r}^{\gamma-1})| \leq C r^{\gamma-1}$.  For the $r$ derivative we reintroduce the decomposition 
$B = B_1 + B_2$.  We have $r\del_r B_1 \hat{f} = B_1 r\del_r \hat{f} + [r\del_r, B_1] \hat{f}$.  Now $B_1$ is of strictly positive type, i.e., 
it is of order $0$, vanishes (to all orders) at $\lf$ and to order $\eta > \gamma$ at $\ff$; by \cite[Proposition 3.30]{Ma1} again, 
the commutator $[r\del_r, B_1]$ has the same properties. Therefore $| r\del_r B_1 \hat{f} | \leq C r^\gamma$.  Finally,
$ r\del_r B_2$ vanishes to order greater than $\gamma$ at $\lf$ and $\ff$, hence $|r\del_r B_2 \hat{f}| \leq C r^\gamma$ too.
Altogether, $|\del_r B \hat{f}| \leq C r^{\gamma-1}$, as required. The analogous estimate for $r^{-1} \del_\theta B \hat{f}$ is similar
but simpler. These estimates together imply that $B \hat{f} \in \calC^{0,\gamma}_w$.

Suppose at last that $\ell > 0$.  We use a simple commutator argument again, noting that $\del_y B f = B \del_y f + [\del_y, B] f$, 
so by induction we obtain that $f \in \calC^{\ell,\gamma}_w \Rightarrow Bf \in \calC^{\ell,\gamma}_w$.
\end{proof}

Implicit in this argument is that when $\eta = 0$ and $E \geq 0$, $B f$ may have logarithmic growth as $r \to 0$ when 
$f \in \calC^{0,\gamma}_e$. To see this, observe first that $B_1 f$ is well-behaved as before. In addition, $r\del_r B_2$ is 
weakly positive with positive index set at $\lf$, so $r\del_r  B_2 f \in \calC^{\ell,\gamma}_e$ for all $\ell$. Using
only that $|r\del_r Bf| \leq C$ and integrating from $1$ to $r$ gives $|Bf| \leq C(1 + |\log r|)$.  It is for this reason
that we have to treat certain of the operators in $\calQ^*$ separately. 

\begin{proof}[Proof of Proposition~3.3]
Proposition~\ref{bddpsido} provides the main step. The key point is that $Q_i \circ G$ is of weakly positive type for all 
$Q_i \in {\calQ}$; we then reduce to this case for the remaining operators $Q_i \in \calQ^*  \setminus {\calQ}$.

Suppose then that $Q_i \in {\calQ}$; we show that $Q_i G$ is of weakly positive type. Observe that each vector field $r\del_r$, 
$\del_\theta$ and $r\del_{y_j}$ on $X$ lift smoothly to $X^2_e$ via the blowdown map $\pi: X^2_e \to X$; indeed, each of these 
lifts is a vector field on this blown up space which is tangent to all boundary faces. Thus $\rho_{\ff} \pi^* \del_{y_\ell}$ is tangent to 
all faces, as is $\rho_{\ff} \pi^* r^{-1}\del_\theta$, while $\rho_{\ff} \pi^* \del_r$ differentiates transversely to the left face, but 
is tangent to all other faces.  Therefore, each operator $Q_i \in \mathcal Q^*$ lifts to an operator on $X^2_e$ 
of the form $\rho_{\ff}^{-k} \tilde{Q}_i$ where $\tilde{Q}_i$ acts tangentially along $\ff$ and where $k = 1$ or $2$. 
When $k = 1$, the composition $Q_i \circ G$ has order $1$ at $\ff$, so is of weakly positive type. When $k = 2$,
then it is necessary that $Q_i$ annihilate the leading coefficient of $G$ at $\lf$ so that $Q_i \circ G$ has 
positive index set there, and this is precisely what determines the subcollection ${\calQ} \subset \calQ$. 
Note that we are using here that the coefficient $a_0$ of $r^0$ at $\lf$ is independent of $\theta$, which is the case 
since this term is annihilated by the indicial operator.  
We have now proved that $Q_i \circ G: \calC^{0,\gamma}_w \to \calC^{0,\gamma}_w$ for all $Q_i \in {\calQ}$. 

Next consider the operator $Q = \del^2_{y_i y_j}$. From the considerations in the previous paragraph, $Q \circ G$ is nonnegative,
but not strictly positive, so if $f \in \calC^{0,\gamma}_w$, then it takes an extra step to show that $Q \circ G (f) \in \calC^{0,\gamma}_w$. 
Decompose $G = G_1 + G_2$ with $G_1$ supported near the diagonal and vanishing near the left and right faces, 
and so that $G_2$ has no diagonal singularity.  Then $Q \circ G_1$ has order $0$ but has empty index set near $\lf$, hence
is weakly positive, so Proposition~\ref{bddpsido} shows that $Q \circ G_1 f \in \calC^{0,\gamma}_w$.   

For the other term, recall the decomposition $f = \hat{f} + \tilde{f}$, as in the proof of Proposition~\ref{bddpsido}, where 
$\tilde{f} \in r^\gamma \calC^{0,\gamma}_e$ and $\hat{f}$ satisfies \eqref{extu0}.  The previous proof shows that $Q \circ G_2 (\tilde{f}) 
\in r^\gamma \calC^{0,\gamma}_e \subset \calC^{0,\gamma}_w$, cf.\ \eqref{eqw0e}.  Next, write 
\[
Q\circ G_2 (\hat{f}) = (\del_{y_i} \circ G_2)  (\del_{y_j} \hat{f}) + \del_{y_i} [\del_{y_j}, G_2] \hat{f}.
\] 
Observe that $\del_{y_i} \circ G_2 \in \Psi^{-\infty, 1, E, E}_e$ (here $E$ is the same index set as in the characterization of $G$).
The other operator is of the same type; indeed, by \cite[Proposition 3.30]{Ma1}, the commutator $[\del_{y_j}, G_2] \in 
\Psi^{-\infty,2,E,E}_e$ so $\del_{y_i} [\del_{y_j}, G] \in \Psi^{-\infty,1,E,E}_e$ as well.  Now using the pushforward theorem together
with the estimates $|\hat{f}| \leq C$, $|\del_{y_j} \hat{f}| \leq C r^{-1 + \gamma}$, we see that $|Q \circ G_2 (\hat{f})| \leq C$.
To show that these terms are actually in $\calC^{0,\gamma}_w$, note that $\del_r \circ (\del_{y_i} \circ G_2)$, 
$\del_{y_\ell} \circ (\del_{y_i} \circ G_2)$ and $r^{-1} \del_\theta \circ (\del_{y_i} \circ G_2)$ all lie in $\Psi^{-\infty, 0, E', E}_e$
where $E'$ is a nonnegative index set, so by the pushforward again, applying these to a function bounded by 
$r^{-1 + \gamma}$ produces a function which is again bounded by a multiple of $r^{-1 + \gamma}$.  (For the second
term we even have a much stronger estimate, but this is unimportant here.)  We have proved that
$|\nabla \circ Q \circ G_2 (\hat{f})| \leq C r^{-1 + \gamma}$, and hence that $Q \circ G_2 (\hat{f}) \in \calC^{0,\gamma}_w$. 

The final operator to consider is $P_{1\bar{1}}$. We use the same trick as \cite{D}, noting that $P_{1 \bar{1}} = \Delta_g - \sum a_i Q_i$
where the $Q_i$ are all operators of the type considered above, with coefficients in $\calC^{0,\gamma}_w$, and from this
the corresponding bound is clear. 
\end{proof}

We remark that the proof gives slightly more. Namely, in case $a_2$, the coefficient of $r^2$ in the expansion
of $G$ at $\lf$ is independent of $\theta$, which is the case for solutions of the \MA equation (see 
Proposition~\ref{PreciseAsympExpansionProp}), and $\be < 1/2$, then $r^{-2}\del_\theta^2 \circ G$ is
of weakly positive type, hence both $ r^{-2} \del_\theta^2 \circ G$ and $(\del_r^2 + r^{-1}\del_r) \circ G$ 
are bounded on $\calC^{0,\gamma}_w$. 

\subsection{A comparison of methods}
The previous subsections provide a review of the terminology and basic results about edge operators.
The point of including this is to show that Proposition~\ref{mainlinearprop} follows directly from this 
general existing theory. Since Donaldson's approach \cite{D} has the appearance of being more 
elementary, it is worth saying a bit more about the similarities and differences between the approaches, 
as well as the advantages of the one here. 

The two slightly different methods for constructing the model kernel $G_\beta$ are equivalent, and there is 
little to recommend one method over the other. The other two steps of the argument in \cite{D} are inverted 
relative to the development here. The edge parametrix construction provides a systematic way to pass from 
the polyhomogeneous structure of the model inverse $G_\beta$ to the corresponding structure for the actual 
inverse $G$. Once that is known, the H\"older boundedness for $G$ and $G_\beta$ are then deduced 
from the general result about boundedness of edge pseudodifferential operators acting on edge H\"older spaces,
the proof of which reduces by scaling to little more than the boundedness of standard pseudodifferential 
operators on ordinary H\"older spaces. Donaldson, by contrast, first establishes the H\"older estimates for 
the model operator $G_\beta$ using related scaling arguments and then observes that these estimates can be 
patched together to obtain the H\"older boundedness for the differentiated kernels $P_{i\jbar} G$. In other words,
the patching (or transition from the model to the actual inverse) is done at the level of the parametrix in our 
approach, but at the level of a priori estimates in particular function space in Donaldson's. The disadvantage
of this latter approach is that one is too closely tied to the function space on which the model a priori
estimates were obtained. This makes that method harder to apply when proving the higher regularity
estimates, for example, and this higher regularity turns out to be key in the existence theory. 
Thus, while the edge theory requires a certain amount of technical overhead, it provides a number of
substantial benefits.  These become even more apparent in the generalization of this theory to the case
of divisors with simple normal crossings. 

\section{Higher regularity for solutions of the \MA equation}
\label{HigherRegSection}
We now use the machinery of the last section to prove one of our main results, that under reasonable initial hypotheses, 
solutions of the complex Monge-Ampere equation are polyhomogeneous (Theorem \ref{PhgMainThm}). This type of proof 
has appeared in many places by now. One origin is the proof of polyhomogeneity for {\it complete} Bergman and 
K\"ahler--Einstein metrics on strictly pseudoconvex domains by Lee and Melrose \cite{LM}; that argument was clarified
and recast into something near the present form in \cite{Ma2}, where polyhomogeneity of solutions of the singular 
Yamabe problem (or obstructions to such polyhomogeneity) was determined. This regularity result was announced in \cite{Ma0}. 

\medskip
We turn to the proof of Theorem \ref{PhgMainThm}.  
There are three main steps. The first is to show that $u \in \calC^{k,\gamma}_e$ 
for every $k \in \NN$; the second is to improve this to full conormality, i.e., to show that $u \in \calA^0$;
in the last, we improve this conormality to the existence of a polyhomogeneous expansion.   The first step
is equivalent to standard higher elliptic regularity for Monge-Ampere equations; this uses the dilation
invariance properties of the edge H\"older spaces in a crucial way.  The second step then breaks this
dilation invariance by showing that we may also differentiate arbitrarily many times along $\D$. 
The final step uses an iteration to show that $u$ has a longer and longer partial polyhomogeneous expansion. 

We begin, then, by quoting a consequence of the Evans--Krylov--Safonov
theory concerning solutions of \MA equations \cite{KrylovSafonov1,Evans1,Krylov},
or rather its extension to the complex \MA equation \cite{Siu}.
\begin{thm}
\label{EvansKrylovThm}
Let $\o$ be a smooth \K metric in a ball $B\subset \CC^n$ and $F\in \calC^\infty(B\times\RR)$. Suppose that $u\in \calC^2(B)$ 
is a solution of $\o_u^n/\o^n = F(z,u)$ on $B$. Then for any $k \geq 2$, there is a constant $C$ depending on $F$,
$k$, $\o$, $\sup_B |u|$ and $\sup_B |\Delta_\o u|$ such that if $B'$ is a ball with the same center as $B$ but with 
half the radius, then
\[
||u||_{\calC^{k,\gamma}(B')} \leq C.
\]
The constant $C$ depends uniformly on the $\calC^{k+3}(B)$ norm of the coefficients of $\o$. 
\end{thm}
To be precise, the Evans--Krylov theorem gives the $\calC^{2,\gamma}(B)$ estimate. The higher regularity is
obtained by a straightforward bootstrap, since differentiating the equation with respect to any coordinate
vector field $W$ gives a linear equation for $Wu$ with coefficients depending on at most the second derivatives 
of $u$, to which we can apply ordinary Schauder estimates since using the $\calC^{2,\gamma}(B)$ estimate,
the coefficients in the resulting equation are \Holder continuous.

To adapt this to our setting, we first observe that the \MA equation is invariant with respect to the scaling 
$S_\lambda(r,\theta,y)  = (\lambda r, \theta, \lambda y)$ (which in the original complex coordinates
takes the form $(z_1, \ldots, z_n) \to (\lambda^{1/\beta} z_1, \lambda z_2, \ldots, \lambda z_n)$).  
To see this, let $\tilde \o$ be any polyhomogeneous \K edge current, i.e.\ an element of $\calH_{\o_0}$
(this was denoted by $\o$ in the statement of Theorem \ref{PhgMainThm}, but we use the tilde here 
to avoid confusion with the reference metric $\o$ \eqref{oDefEq}), and let $\tilde g$ denotes its 
associated \K metric. We see from \eqref{pertgb} (and polyhomogeneity) that as $\lambda\ra \infty$, 
\[
\lambda^{2}S_{1/\lambda}^* \tilde g \longrightarrow c g_\beta, 
\]
for some constant $c > 0$, where $g_\beta$ is the flat model edge metric on $\CC^n$. 

Now let $B$ be the ball of radius $r_0/2$ centered at some point $(r_0,y_0)$ in the coordinates $(r,y)$, where $r_0$ 
is small, and let $B'$ be the ball of half this radius, and consider the sets $B \times S^1$ and $B' \times S^1$. 
Choosing coordinates so that $y_0 = 0$, we obtain the family of metrics 
\[
g_{r_0}:= r_0^{-2} S_{r_0}^* \left(\left. \tilde g \right|_{B\times S^1}\right), 
\]
which we regard as defined on $\tilde{B} \times S^1$, 
where $\tilde{B}$ is a ball of radius $1/2$ centered at $(1,0)$. Let $\tilde{B}'$ be the ball of radius $1/4$ 
centered at this same point. Finally, consider the family of functions $u_{r_0}(r,\theta,y) = S_{r_0}^* u(r,\theta, y) = 
u( r_0 r, \theta, r_0 y)$, also defined on $\tilde{B} \times S^1$. 

By pulling back the original \MA equation \eqref{genma} from the ball $B$ to $\tilde{B}$, we see that for each $r_0 < 1$,
$u_{r_0}$ satisfies the \MA equation with respect to the metric $g_{r_0}$. Applying the Evans-Krylov estimate and
bootstrapping in this standard ball then gives that
\[
||u_{r_0}||_{k,\gamma; \tilde{B}' \times S^1}  \leq C,
\]
where $C$ depends on $g_{r_0}$, $\sup |u_{r_0}|$ and $\sup |\Delta_{g_{r_0}} u_{r_0}|$.  Since $g_{r_0}$ converges
smoothly in this region, $\sup |u_{r_0}|$ is uniformly controlled, and using that 
$\Delta_{g_{r_0}} u_{r_0} = r_0^2 (\Delta_{\tilde{g}} u)_{r_0}$, this last term is also uniformly bounded as $r_0 \searrow 0$,
we conclude that $u_{r_0}$ is uniformly bounded in any $\calC^{k,\gamma}$ norm in $\tilde{B'} \times S^1$.  

The last step is to recall that the edge H\"older norms are invariant under these rescalings. In other words,
\[
||\left. u\right|_{B'}||_{e; k, \gamma} = ||\left. u_{r_0}\right|_{\tilde{B}'}||_{e; k, \gamma}.
\]
The global $\calC^{k,\gamma}_e$ norm of $u$ is the supremum of these norms over all such balls $B'$, and
hence this too is finite for all $k \geq 0$. We have proved that $u \in \calC^{k,\gamma}_e$ for any $k$. 

We have proved that $(r\del_r)^j (r\del_y)^k \del_\theta^\ell  u$ is bounded for any $j, k, \ell \geq 0$, thus a priori
we only know that $\del_y^\alpha u$ may blow up like $r^{-|\alpha|}$. We now address this
and show that these tangential derivatives are bounded too. Write the \MA equation as 
\[
\log \det (g_{i\jbar} + u_{i \jbar}) = \log \det (g_{i\jbar})  + \log F(z,u).
\]
(As explained just after the statement of Theorem~\ref{PhgMainThm}, the result holds when the usual exponential on the right hand
side is replaced by a function $F(z,u)$ satisfying a few properties.) 
Applying $\del_y$ to both sides and using the standard formula for the derivative of a logarithmic determinant, we find that
\begin{equation}
(\Delta_{\tilde{g}} - V)\del_y u = f,
\label{difff}
\end{equation}
where $\tilde{g}_{i\jbar} = g_{i\jbar} +  u_{i\jbar}$, $V = F_u(z,u)/F(z,u)$ and $f= \del_y 
\log \det (g_{i\jbar})-\Delta_{\tilde g}\del_y\phi +  F_z(z,u)/F(z,u)$, where $\phi$ is a local \K potential 
for the reference metric $g$, i.e., such that $\phi_{i\b j}=g_{i\b j}$.

Recall that even if the initial assumption is that $u \in \calD^{0,\gamma}_e$, we immediately know from Theorem~\ref{TianThmB} 
that $u \in \Dw$, 
and this implies that both $V$ and $f$ lie in $\calC^{0,\gamma}_w$. Since $\del_y \in Q$, Proposition~\ref{mainlinearprop} 
implies that $\del_y u $ is bounded.  We now wish to apply 
Corollary~\ref{equivdomains} to improve this regularity, but to do so, we must show that $u$---and hence
the K\"ahler potential for the metric $\tilde g$--- is the limit in $\Dw$ of polyhomogeneous 
functions. Granting this for the moment, this Corollary implies that $\del_y u \in \Dw$. 

This same argument goes on to show that $\del^k_y u\in \Dom^{0,\gamma}_e$ for any $k$. Indeed, suppose
inductively that for some $k \geq 2$, $\del_y^j u \in \Dw$ for all $j \leq k-1$, and 
\[
(\Delta_{\tilde{g}} - V)\del_y^{k-1} u = 
f_0^{(k-1)} + H^{(k-1)}(z, u, \del_y u, \ldots, \del_y^{k-1} u, Q u, \ldots, Q \del_y^{k-2} u) \in \calC^{0,\gamma}_e, 
\]
where $f_0^{(k-1)} \in \calA^0_{\phg}$ and $H^{(k-1)}$ is a smooth function of its arguments $\del_y^j u$ and $Q_i \del_y^j u$, 
$Q_i \in \calQ$. (In our example, $H^{(k-1)}$ is an algebraic function of these arguments.)  Differentiating yields
\begin{multline*}
(\Delta_{\tilde{g}} - V) \del_y^k u = \del_y f^{(k-1)}_0 + \\
\del_y H^{(k-1)}(z, u, \del_y u, \ldots, \del_y^{k-1} u, Q u, \ldots, Q\del_y^{k-2} u) - 
[ \del_y, \Delta_{\tilde{g}} - V] \del_y^{k-1} u.
\end{multline*}
Since $\del_y^{k-1} u \in \Dw$, we conclude first that $\del_y^k u \in \calC^{0,\gamma}_w$ and in addition, by
a straightforward calculation, that the right side lies in $\calC^{0,\gamma}_w$. (Note that we may as well assume 
that $[\del_y, Q] = 0$.) Hence applyng Corollary~\ref{equivdomains} with precisely the same operator
as before shows that $\del_y^k u$ lies in $\Dw$ and satisfies an equation with correct structure.   This completes
the inductive step.   Recalling that we already proved that $(r\del_r)^i (r\del_y)^j \del_\theta^\ell u \in
\calC^{0,\gamma}_e$ for every $i, j, \ell \geq 0$, an almost identical induction proves that $(r\del_r)^i \del_y^j 
\del_\theta^\ell u \in \Dw$ for every $i, j, \ell \geq 0$. This proves, altogether, that $u \in \calA^0$. 

We now address the claim that the \K potential for $\tilde{g}$ is a limit of polyhomogeneous functions. Prior to 
this inductive argument, we only know that $u \in \calD^{0,\gamma}_e$ (or more precisely, that $u \in \calD^{k,\gamma}_e$
for every $k \geq 0$).  Theorem~\ref{DwCor}, which rests on Tian's Theorem~\ref{TianThmB} as stated and proved 
in Appendix B, asserts that if $u$ is a solution to this \MA equation such that $u$ and $\Delta_{\tilde g} u$ are simply bounded, 
then necessarily $u \in \Dw$.  The claim is implied by the fact that if the H\"older exponent $\gamma$ is replaced by a slightly 
smaller value $\gamma' \in (0,\gamma)$, then $u$ can be approximated by polyhomogeneous functions in the topology of 
$\calD^{0,\gamma'}_w$. In the interior, away from the edge, this is the familiar fact that the closure of $\calC^\infty$ in 
$\calC^{0,\gamma'}$ contains $\calC^{0,\gamma}$ for any $0 < \gamma' < \gamma < 1$, which can be proved by mollification. 
Near the edge, it is possible to use a similar mollification argument in a fixed local coordinate system, but let us explain 
a more systematic approach using the heat kernel.

\begin{lem} If $0 < \gamma' < \gamma < 1$, then $\calA:=\calA^0_{\phg}\cap \Dw$ is dense in $\calD^{0,\gamma'}_w$
\label{density}
\end{lem} 
\begin{proof}
Consider the heat kernel $e^{t\Delta}$ associated to the $L^2$ Friedrichs extension of $\Delta_g$, where $g$ is any fixed (smooth or
polyhomogeneous) edge metric.  The Schwartz kernel of $e^{t\Delta}$ is constructed in \cite{MV}, and it is proved there that if $t > 0$, then
$f_t := e^{t\Delta} f \in \calA^0_{\phg}$ for any $f \in L^2$, in particular for $f \in \Dw$. In addition, $\del_t f_t$ 
is polyhomogeneous with nonnegative index set for any $t > 0$, and $\del_t f_t =\Delta f_t$, so $f_t \in\Dw$ too, hence
$f_t \in \calA$.  Next, since $\Delta$ commutes with $e^{t\Delta}$, it follows that $f_t \to f$ and $e^{t\Delta} \Delta f
= \Delta e^{t\Delta} f = \Delta f_t \to \Delta f$ in $L^2$. This already implies that $\calA$ is dense in $\calD_{\mathrm{Fr}}(\Delta)$ in 
the $L^2$ graph topology; we shall need this fact later in \S 6. 

To prove the corresponding H\"older space result, note that using the same commutation, it suffices
to prove that $f_t \to f$ in $\calC^{0,\gamma'}_w$, since the same argument also gives $\Delta f_t \to \Delta f$ 
in $\calC^{0,\gamma'}_w$. This convergence is proved by noting that by standard heat kernel arguments, $f_t \to f$ in $\calC^0$,
and moreover, $||f_t||_{\calC^{0,\gamma}_w} \leq C$ uniformly in $t$. (This last fact can be proved using very similar arguments
to the ones in the proof of Proposition~\ref{bddpsido}.) It is then a simple exercise in real analysis to conclude
that $f_t \to f$ in the slightly weaker norm $|| \cdot ||_{\calC^{0,\gamma'}_w}$. 
\end{proof}

We come to the final step, that $u$ is polyhomogeneous. This requires two more boundedness properties of edge pseudodifferential
operators, namely that this class of operators preserves the spaces of conormal and of polyhomogeneous functions. In particular,
if $B$ is any weakly positive pseudodifferential edge operator, then 
\begin{equation}
B: \calA^0(X) \longrightarrow \calA^0(X) \quad \mbox{and}  \qquad B: \calA^0_{\phg}(X) \longrightarrow  \calA^0_{\phg}(X)
\label{bddcon}
\end{equation}
are both bounded. The pseudodifferential order $m$ is irrelevant at this point since since we are applying $B$ to functions which 
are infinitely differentiable (with respect to the edge vector fields) anyway. The improvement in the argument below relies on a 
refinement of \eqref{bddcon}.  For the following argument, introduce the notation $\calA^{\nu-}(X) = \cap_{\e > 0} \calA^{\nu-\e}$. 
\begin{lem}
Let $B \in \Psi^{m,2,E,E'}_e(X)$, where $E$ and $E'$ are nonnegative. Then
\[
B: \calA^0(X) \longrightarrow \calA^0_{\phg}(X) + \calA^{2-}(X),
\]
and more generally, if $\nu \geq 0$, 
\[
B: \calA^0_{\phg}(X) + \calA^{\nu-}(X) \longrightarrow \calA^0_{\phg}(X) + \calA^{(\nu+2)-}(X).
\]
More concretely, 
\[
u \sim \sum_{0 \leq \mbox{Re}\, \gamma < \nu} a_{\gamma, p} r^\gamma (\log r)^p + \calO(r^{\nu-}) \Longrightarrow 
Bu \sim \sum_{0 \leq \mbox{Re}\, \gamma < \nu+2} b_{\gamma, p} r^\gamma (\log r)^p + \calO(r^{(\nu+2)-}),
\]
where the errors on each side are conormal, $\calO(r^{\nu-})$ denotes an error which decays like $r^{\nu-\e}$ for all
$\e > 0$, and $a_{\gamma,p} = b_{\gamma,p} = 0$ if $\mbox{Re}\, \gamma =0$ and $p \geq 1$. 
\end{lem}
\begin{proof}
The second assertion is an easy consequence of the first. To prove this first assertion, 
if $B$ has index set with all exponents greater than or equal to $2$ at the left ($r \to 0$) face, then 
since $B$ vanishes to order $2$ at the front face, we can write $B = r^2 \tilde{B}$ where $\tilde{B}$
is nonnegative. Hence in that case, $B: \calA^0 \to \calA^{2-}$. 

Now suppose that the exponents in the expansion of $B$ at the left face of $X^2_e$ which lie in the range $[0,2)$ are 
$\gamma_1, \ldots, \gamma_N$, and assume that there are no log terms in these expansions for simplicity. Then
\[
B^{(N)} := (r\del_r - \gamma_1)(r\del_r - \gamma_2) \ldots (r\del_r - \gamma_N) B \in \Psi^{m+N, 2, E(2), E'}_e(X),
\]
where $E(2)$ is some new index set derived from $E$ which has all elements greater than or equal to $2$. 
Thus we can apply the previous observation to see that $B^{(N)}: \calA^0 \to \calA^{2-}$, or said slightly
differently, if $f$ is bounded and conormal, so $f \in \calA^0$, then $B^{(N)}f = u^{(N)}$ is of the
form $r^{2-\e} v_\e$ for any $\e > 0$ where $v_\e$ is bounded and conormal.  Now we can integrate the ODE $(r\del_r - \gamma_1)
\ldots (r\del_r - \gamma_N)$ to see that $u = Bf$ has a partial polyhomogeneous expansion with
all terms of the form $r^{\gamma_j}$, $j = 1, \ldots, N$, since each of these terms are killed by $r\del_r - \gamma_j$. 
\end{proof}

We wish to apply this lemma when $G$ is the Green function for $\Delta_g + V$, where $g$ and $V$ are polyhomogeneous.
It is straightforward to extend this result slightly to show that it remains valid for some fixed $\nu$ provided both $g$ and 
$V$ only lie in $\calA_{\phg}^0 + \calA^{\nu}$.  We leave details of this extension to the reader. 

Finally, let us apply this to the equation $L \del_y u = f$, cf.\ \eqref{difff}.  We know initially that $f \in \calA^0$, hence
at the first step, $\del_y u \in \calA^0_{\phg} + \calA^{2-}$. But this now gives that $f$ and the coefficients
of $L$ lie in $\calA^0_{\phg} + \calA^{2-}$, hence $\del_y u \in \calA^0_{\phg} + \calA^{4-}$. Continuing on in
this manner gives a complete expansion for $\del_y u$, and from this we deduce also that $u$ is
polyhomogeneous. This concludes the proof of Theorem \ref{PhgMainThm}.

Let us remark what is really going on in this proof. Once we have established that $u$ is conormal, i.e., that it is infinitely 
differentiable with respect to $r\del_r$, $\del_\theta$ and $\del_y$, then we can treat the \MA equation satisfied by $u$ 
as an ODE in the $r$ direction; all dependence in the other directions can be treated parametrically, and in particular, 
$y$ and $\theta$ directions are harmless. Thus the important step is going from $u \in \cap \,\calC^{k,\gamma}_e$ to 
$\del_y^\ell u \in \cap\, \calC^{k,\gamma}_e$ for all $\ell \geq 0$. 

While this sort of iteration method was already mentioned in \cite{D}, it is less awkward to use edge spaces here.
The reason is that the different scales in this problem make it necessary to work with functions involving integer 
powers of both $r$ and $r^{1/\be}$, and these are only finitely differentiable in the wedge spaces, but 
infinitely differentiable in the edge spaces.

\subsubsection*{Determination of leading terms}
\label{AsympExpansionSubSec}
For various applications below, in particular the determination of the asymptotics of the metric and curvature,
we must determine the first few terms of the expansion of a solution of the \MA equation.

\begin{prop}
\label{PreciseAsympExpansionProp}
Let $\vp$ be a solution of the \MA equation \eqref{RCMEq}. Suppose that $\vp \in 
\Dom^{0,\gamma}_w$, and hence by Theorem \ref{PhgMainThm}, $\vp \in \calA_{\phg}^0$. Then the asymptotic 
expansion of $\vp$ takes the 
form
\begin{equation}
\label{GeneralExpansionEq}
\vp(r,\th,y) \sim  \sum_{j,k, \ell \geq 0} a_{jk\ell}(\th,y)r^{j+\frac k\be}(\log r)^{\ell}
\end{equation}
as $r \searrow 0$.  Certain coefficients are always absent; for example, $a_{00\ell} = 0$ for $\ell > 0$
and $a_{1 0 \ell} \equiv 0$ for all $\ell$. 
If $a_{jk\ell} = 0$ for some $j, k$ for all $\ell > 0$, then we write this coefficient simply as $a_{jk}$.  
When $0 < \be < 1/2$, 
\begin{equation}
\label{ExpansionBetaLessHalfEq}
\vp(r,\th,y) \sim  a_{00}(y) + a_{20}(y)r^2 + (a_{01}(y)\sin\th+b_{01}(y)\cos\th)r^{\frac1\be} 
+ a_{40}(y)r^4 + O(r^{4+\eps})
\end{equation}
for some $\eps=\eps(\be) > 0$; when $\be = 1/2$, the asymptotic sum on the 
right includes an extra term $(a_{02}(y)\sin2\th+b_{02}(y)\cos2\th)r^4$; finally, 
if $1/2 < \be < 1$, then
\begin{equation}
\label
{ExpansionAllBetaEq}
\begin{aligned}
\vp(r,\th,y) =& a_{00}(y) + (a_{01}(y)\sin\th+b_{01}(y)\cos\th)r^{\frac1\be} + a_{20}(y)r^2  + O(r^{2+\eps}) 
\end{aligned}
\end{equation}
for some $\eps=\eps(\be) > 0$. 
\end{prop}

We begin with a lemma.
\begin{lem}
\label{foAsympExpansionLemma}
The twisted Ricci potential $f_\o$ can be expressed as
\[
f_\o 
= 
\sum_{k=-1}^{n-1}c_{0k} r^{2k + \frac{2}{\be}} + \sum_{k=0}^{n-1} (c_{1k} 
+ c_{2k} r\cos \th + c_{3k} r\sin\th) r^{2k},
\]
where each $c_{jk}$ is a smooth function of $r^{\frac1\be},\th,$ and $y$.
\end{lem}
\begin{remark}
We may, of course, Taylor expand the coefficients $c_{ik}$ to obtain an asymptotic sum 
involving the terms $r^{2k + (2+\ell)/\beta}$ and $r^{2k+\ell/\beta}$, respectively, with coefficients depending
only on $y$ and $\theta$. 
\end{remark}
\begin{proof}
By (\ref{ddbarPhioneFirstEq}), 
\begin{multline}
\label{DetOmeganEq}
\o^n/(n!(\i)^{n^2} d{\bf z}\w d{\bf \b z}) 
= \det\Big[\frac{\del^2(\psi_0+\phi_0)}{\del z^i\del\overline{ z^j}}\Big] \\ 
= \sum_{k=0}^n f_{0k} |z_1|^{2k\be}  + \sum_{k=1}^n (f_{1k} + f_{2k} z_1 + f_{3k} \overline{z_1}) |z_1|^{2 k \be - 2},
\end{multline}
where all $f_{jk}$ are smooth functions of $(z_1, \ldots, z_n)$,
and $d{\bf z}:=dz_1\w\cdots\w dz_n$.
It follows that 
\begin{equation}
\label{ExpansionVolRatioEq}
\frac{\o^n}{|s|_h^{2\be-2}\o_0^n}
=
\frac{(\o_0+\i\ddbar\phi_0)^n}{|s|_h^{2\be-2}\o_0^n}
= \;  \sum_{k=-1}^{n-1} \tilde{f}_{00} r^{2k + \frac{2}{\be}} + \sum_{k=0}^{n-1} ( \tilde{f}_{1k} + \tilde{f}_{2k} r\cos \th + \tilde{f}_{3k} r\sin\th)) r^{2k}.
\end{equation}
where each $\tilde{f}_{jk}$ is a smooth function of the arguments $r^{\frac1\be}\cos \th, r^{\frac1\be}\sin\th$ and $y$. 
In addition, we have already noted that $\phi_0 = r^2 \Phi_0$ where $\Phi_0$ is also smooth as a function of 
$r^{\frac1\be}\cos \th, r^{\frac1\be}\sin\th$ and $y$. The result now follows directly from the equation
\begin{equation}
\label
{fomegaSecondEq}
e^{-f_\o}=\frac{(\o_0+\i\ddbar\phi_0)^n}{|s|_h^{2\be-2}\o_0^n}  e^{\mu\phi_0-F_{\o_0}},
\end{equation}
where $F_{\o_0}$ is defined by 
$\i\ddbar F_{\o_0} = \Ric\o_0-\mu\o_0+(1-\beta)\i\ddbar\log a$
(where $a$ is defined in \eqref{afunctionEq}),
and the equation itself, together with \eqref{normalization},
fixes a normalization for $F_{\o_0}$, and again $F_{\o_0}$ is smooth in these same arguments.
\end{proof}

\begin{proof}[Proof of Proposition~\ref{PreciseAsympExpansionProp}]
The idea is quite simple. Since we now know that $\vp$ has an asymptotic expansion, we
simply substitute a `general' expansion into the equation 
\begin{equation}
\ovpn/\on= F(z,\vp) 
\label{DetMAEq}
\end{equation} 
and determine the unknown exponents and coefficients.   Since our main case of interest is
when $F(z,\vp) = e^{f_\o-s\vp}$, we shall explain the argument for this special function,
but it should be clear that the same type of argument works in general. 

Using the precise form of the expansion for $f_\o$ determined above, 
the index set for $\vp$ must be contained in 
\[
\Gamma := \{  (j + {k}/{\be}, \ell): j, k, \ell \in \NN_0\},
\]
or in other words, the only terms which appear are of the form 
$a_{j k \ell}(\th,y) r^{j + \frac{k}{\be}}(\log r)^\ell$. 
This is done inductively. Supposing that we know that this is true for all
$j, k$ such that $j + k/\be \leq A$, then we only need consider the
action of $P_{1\b1}$ on the next term in the series $a_{\gamma \ell} r^\gamma (\log r)^\ell$.
This must either be annihilated by $P_{1\b1}$, i.e., $\gamma$ is an integer multiple
of $1/\be$, or else it must match a previous term in the expansion, i.e.,
$\gamma  - 2 = j' + k'/\be$. In either case, the form of the expansion propagates
one step further. 

Since the solution $\vp$ is bounded, there are no terms $a_{00\ell} (\log r)^\ell$ with $\ell > 0$,
so using the convention in the statement of the theorem, the leading term is simply $a_{00} r^0$.
Note further that $a_{00}$ depends only on $y$ but not on $\th$. This can be seen by
substituting in the equation. If $a_{00}$ were to depend nontrivially on $\th$, then the term $P_{1\b1} \vp$
would contain $r^{-2} \del_\th^2 a_{00}$, and this is not cancelled by any other term in the equation.
Hence $a_{00} = a_{00}(y)$. 

Similar reasoning can be applied to the next few terms in the expansion.  We use discreteness of the set of exponents 
to progressively isolate the most singular terms after we substitute the putative expansion for
$\vp$ into the equation.  Since $a_{00}$ is independent of $\th$ and $r$, $P_{1\b1} a_{00}$, 
and $P_{1\jbar} a_{00}$ and $P_{i\b1} a_{00}$ 
are all bounded (in fact, zero). Hence if the next term in the expansion is 
$a_{\gamma \ell} r^\gamma (\log r)^\ell$ with $\gamma \leq 2$, 
then applying $P_{1\b1}$ to it produces as its most singular term $r^{\gamma-2}(\log r)^\ell (\gamma^2 + \del_\th^2) a_{\gamma \ell}$. 
This shows immediately that either $\gamma$ must be an indicial root, 
i.e., $\gamma = 1/\be$ if $\be > 1/2$
with $a_{\gamma \ell}$ a linear combination of $\cos\th$ and $\sin\th$, or else $\gamma = 2$.
Note that this also shows that $a_{10\ell} \equiv 0$ for all $\ell \geq 0$.  

Assuming $\gamma < 2$ and $\ell > 0$, then using the leading order cancellation, the next
most singular term in $P_{1\b1} a_{01\ell} r^{\frac1{\be}}(\log r)^\ell$ is $\gamma r^{\gamma-2} (\log r)^{\ell-1}a_{01\ell}$ 
with no other term to cancel it. This is impossible, so we have ruled out all such terms with $\ell > 0$.
If $\gamma = 2$ and $\ell > 0$, there is no longer a leading order cancellation, but we are left with 
the singular term $a_{20\ell} (\log r)^\ell$, so $a_{20\ell} = 0$ when $\ell > 0$.  

Now consider what happens to the term $a_{20} r^2$. It interacts with the leading order terms $a_{00}$
in $\vp$ and $c_{00}$ in $f_\o$ only. Neither of these depend on $\th$, so we find that $a_{20}$
is a function of $y$ alone. 

We can continue this same reasoning further. Applying $P_{1\b1}$ to the next term in the expansion 
$a_{\gamma \ell} r^{\gamma} (\log r)^\ell$ beyond $a_{20}r^2$ produces
a leading order term which is a nonzero multiple of $a_{\gamma \ell} r^{\gamma-2} (\log r)^\ell$
if $\ell > 0$. Even though this term is bounded now, there are no other log terms at
the level $r^{\gamma-2}$ in \eqref{DetMAEq}.  On the other hand, if $\ell = 0$, then we 
end up with a term $r^{\gamma-2}(\gamma^2 + \del_\th^2)a_{\gamma 0}$, and there are no
terms in \eqref{DetMAEq} to cancel it either. Hence $\gamma$ must be one of the two indicial 
roots $k/\be$, $k = 1$ or $2$, and the coefficient must be a linear combination of $\cos k\th$ and  $\sin k\th$.

We comment further on the cases $\be=1/2$ or $\be=1/4$. 
In the former, one might suspect that one would need a term $r^2 \log r a_{021}$ because
applying $P_{1\b1}$ to this should match the $r^0$ term coming from the leading
coefficients of  $\vp$ and $f_\o$. However, those coefficients do not depend on $\th$,
whereas $a_{021}$ would be a combination of $\cos2\th$ and $\sin 2\th$, as above,
so there is no interaction, hence no log terms at this location.  This is also true for $\be = 1/4$. 
\end{proof}

\begin{remark} 
{\rm It is worth noting explicitly that while both the reference and solution metrics have expansions, the solution metric 
may have more terms in its expansion than the reference metric. One consequence of this is that the computations in 
the Appendix do {\it not} apply to the solutions $\o_{\vp(s)},\, s>-\infty$; in particular one cannot conclude that the bisectional 
curvatures of the solution metrics are bounded when $\be > 1/2$, and indeed, they are not!}
\end{remark}

Using Theorem \ref{PhgMainThm} and Proposition \ref{PreciseAsympExpansionProp}, we obtain the following regularity statement.
\begin{cor}
\label{DsCor}
Let $\vp$ be a solution of the \MA equation \eqref{RCMEq}, with $\vp \in \Dom^{0,\gamma}_e$. Then $\vp$ is polyhomogeneous,
and there exists some $\epsilon>0$ depending only on $\be$ such that $\vp \in \Dom^{0,\gamma'}_w$ for every 
$\gamma'\in[0,\epsilon(\be)]$. 
\end{cor}

\begin{remark} 
{\rm  As noted in the Introduction, Proposition \ref{PreciseAsympExpansionProp} sheds light on the
distinction between the easier ``orbifold regime" $\be\in(0,\frac12]$ and the case $\be \in (\frac12,1)$. 
In particular, we see that one should not expect uniform estimates even on the third derivatives $\vp_{i\b j k}$ 
when $\be > \frac12$. This is one reason why we study the \Holder norms of second derivatives in 
\S \ref{HolderSection} rather than considering the third order estimates as in the classical approach of Aubin and Yau.}
\end{remark}

\section{Maximum principle and the uniform estimate}
\label{CzeroNonpositiveSection}
We now recall the formulation of the maximum principle in this singular setting.
The main issue is to find barrier functions which allow one to reduce to the classical
maximum principle on $\MsmD$. These barrier functions were used already in \cite{J2}.

\begin{lem}
\label{BarrierFnLemma}
Let $f$ be continuous on $M$ and satisfy $|f(r,\th,y) - a(y)| \leq C r^\gamma$ for some $a \in \calC^0(\D)$
and $0 < \gamma < 1$. Then for $\eps$ sufficiently small,
\begin{enumerate}
\item[(i)] if $C>0$, then $f+C|s|_h^{\eps}$ achieves its maximum in $\MsmD$; \
\item[(ii)] if $c >0$ is small enough, then $c |s|_h^{\eps}\in \PSH(M,\o)$.
\end{enumerate}
\end{lem}
\begin{proof}
(i) The function $|s|_h^\eps$ is comparable to $r^{\eps/\beta}$, so for $C > 0$, $r \mapsto f(r,\theta,y) + 
C |s|_h^\eps$ strictly increases, hence cannot reach its maximum at $r = 0$. 
\hfill\break
(ii) Let $h$ be a smooth Hermitian metric on $L_D$ with global holomorphic section $s$
so that $\D = s^{-1}(0)$. For any $b\ge0$, we have $\i\ddbar b\ge b\i\ddbar\log b$.
Setting $b:=|s|^\eps_h$ gives
\begin{equation}
\label{LaplacianOfBarrierFnEq}
\i\ddbar b \geq \i\, \eps|s|^\eps_h \ddbar\log|s|_h =-\frac12\eps|s|^{\eps}_h R(h)>-C\o,
\end{equation}
where $C$ depends only on the choice of $\o,h,s,\eps$. Thus $C^{-1} b\in \PSH(M,\o)$.
\end{proof}

The assumption on $f$ above holds in particular for $f\in \calC^{0,\gamma}_w$, and for $f$ and $\Delta_\o f$ when 
$f\in \calD^{0,\gamma}_w$. 

This lemma is used as follows. Replacing $|s|_h^\eps$ by $c |s|_h^\eps$ and letting $c$ tend to 0, 
we obtain estimates which are the same as those one would expect from the maximum principle on $\MsmD$.
See the proofs of Lemmas~\ref{AYUniquenessLemma} and
\ref{TYLaplacianEstimateLemma} below for more on this.
The uniqueness and a priori $\calC^0$ estimate when $\mu\le 0$ are 
now immediate consequences.

\begin{lem} 
\label{AYUniquenessLemma}
Solutions to the \MA equation \eqref{TwoParamCMEq} with $s \le 0$ are unique (when $s=0$, only unique
up to a constant) in $\Dom^{0,\gamma}_w \cap \PSH(M, \o)$ and satisfy
\begin{equation}
\label{AYTYCzeroEq}
||\vp(s,t)||_{\calC^0(M)} \le C=C(||f_\o||_{\calC^0(M)},M, \o). 
\end{equation}
\end{lem}
\begin{proof}
Uniqueness when $s < 0$ is proved in \cite{J2}; that argument carries over directly to this \MA
equation and either of the types of function spaces we are using here, 
because of Lemma~\ref{BarrierFnLemma}.
Finally, when $s=0$ the result of \Blocki~\cite{Bl2003} gives uniqueness in $L^\infty(M)$
up to a constant, and that constant can be chosen by requiring that $\sup\vp(0,t)=\lim_{s\ra0^-}\sup\vp(s,t)$.

The same argument also shows that  
$||\vp(s,t)||_{\calC^0(M)} \le -2s^{-1}||f_\o||_{\calC^0(M)}$, for each $s <0$.
One can then obtain a uniform estimate for all $s\le0$ as follows.
First, by the above, we may assume that $s>S$, for some $S<0$.
With respect to the fixed smooth \K form $\o_0$, \eqref{TwoParamCMEq} can be rewritten as 
\[
\o_{\vp}^n=\o_0^n F|s|^{2\be-2}_h e^{tf_\o+c_t-s \vp},
\]
where $F\in \calC^0(M)$. By the previous estimate, 
$||e^{tf_\o+c_t-s \vp}||_{\calC^0(M)}\le C$
uniformly in $s$. It follows that 
$||F|s|^{2\be-2}_h e^{tf_\o+c_t-s\vp}||_{L^p(M,\o_0^n)}\le C_p$, for all $p\in(1,1/(1-\be))$, with 
$C_p$ independent of $s\le0$. Assuming this, by \Kolodziej's estimate \cite{K1} 
$\h{\rm osc}\,\vp(s,t)\le C$, with $C>0$ independent of $s,t$, and since 
by \eqref{TwoParamCMEq} $\vp(s,t)$ changes
sign then also $|\vp(s,t)|\le C$.
\end{proof}

\section{The uniform estimate in the positive case}
\label{UniformPositiveSection}
In contrast to the nonpositive curvature cases, when $\mu>0$, there  are well-known obstructions to the existence 
of an a priori $\calC^0$ estimate along the continuity path. In this section 
we review the standard theory due to Tian and others \cite{T97,Tbook} along 
with the necessary modifications to adapt it to our setting.
For an alternative variational approach that can be applied to more general classes of
plurisubharmonic functions we refer to \cite{Berm}.

\subsection{Poincar\'e and Sobolev inequalities}
\label{PoincareSobolevSubSec}
In this subsection we show that along the continuity path (\ref{RCMEq}) one has uniform 
Poincar\'e and Sobolev inequalities.

We first prove that a uniform Poincar\'e inequality holds as soon as $s>\eps>0$. The following argument 
is the analogue of \cite[Lemma 6.12]{Tbook} in this edge setting, and also generalizes \cite[Lemma 3]{LT} 
to higher dimensions. The second part is the same assertion as \cite[Proposition 8]{D}. 
The proof here takes advantage of the fine regularity results for solutions available to us. 
\begin{lem}
\label{LambdaoneLemma} 
Denote by $\Delta_{\o_{\vp(s)}}$ the Friedrichs extension of the
Laplacian associated to $\o_{\vp(s)}$. 
\begin{itemize}
\item[i)] For any $s\in(0,\mu)$, $\lambda_1(-\Delta_{\o_{\vp(s)}})>s$. 
\item[ii)] For $s=\mu$, $\lambda_1(-\Delta_{\o_{\vp(\mu)}})\ge \mu$.  If
$(\Delta_{\o_{\vp(\mu)}}+\mu)\psi=0$ then $\nabla^{1,0}_{g_{\vp(\mu)}}\psi$
is a holomorphic vector field tangent to $D$.
\end{itemize}
\end{lem}
\begin{proof}
(i) Let $\psi$ be an eigenfunction of $\Delta_{\o_{\vp(s)}}$ with eigenvalue $-\lambda_1$.  Since
$\vp(s)$ is polyhomogeneous, then the eigenfunctions of $\Delta_{\o_{\vp(s)}}$ are also polyhomogeneous.
This is a special case of the main regularity theorem for linear elliptic differential edge operators
from \cite{Ma1}. The proof uses the same pseudodifferential machinery described in \S 3 (although
for this particular result it is possible to give a more elementary proof). 
The key fact is that $\psi \sim a_0 r^0 + a_1 r^{\frac{1}{\be}} + a_2 r^2 + O(r^{2+\eta})$ for some $\eta > 0$, and 
in particular there is no $\log r$ in this expansion, since we are using the Friedrichs extension. 

The Bochner--Weitzenb\"ock formula states that on $M\sm D$,
\[
\frac12\Delta_g|\nabla_g f|_g^2 = \Ric(\nabla_g f,\nabla_g f)+|\nabla^2 f|^2_g+\nabla f \cdot \nabla (\Delta_g f).
\]
Since $\Delta_g=2\Delta_\o$ and $|\nabla^2 f|^2_g=2|\nabla^{1,0}\nabla^{1,0}f|^2+2(\Delta_\o f)^2$, this becomes
\begin{equation}
\label{WeitzCxEq}
\Delta_\o|\nabla^{1,0} \psi|_g^2 = 2\Ric(\nabla^{1,0} \psi,\nabla^{0,1} \psi) + 2|\nabla^{1,0}\nabla^{1,0}\psi|^2
+ 2\lambda_1^2\psi^2 - 4\lambda_1|\nabla^{1,0} \psi|_\o^2.
\end{equation}
We now claim that
\begin{equation}
\label{LaplacianIntegrateZeroEq}
\int_M \Delta_\ovp|\nabla^{1,0} \psi|_{\ovp}^2\o_{\vp}^n=0.
\end{equation}
This follows directly from the expansion of $\psi$, since the worst term in the expansion of $\nabla^{1,0}\psi$
is $r^{\frac{1}{\be} - 1}$. Hence if we integrate over $r \geq \eps$ then the boundary term is of order 
$\eps^{\frac{2}{\be} - 2}$ (taking into account the measure $r d\th dy$ on this boundary), and this tends to $0$ with $\eps$. 
This proves the claim. Thus integrating (\ref{WeitzCxEq}) and using that $\Ric\o(s)>s\o(s)$
when $s<\mu$ we see that $\lambda_1>s$.

\medskip
\noindent
(ii) When $s=\mu$ this same argument yields $\lambda_1\ge \mu$.  Moreover, equality holds precisely when 
$\nabla^{1,0}\nabla^{1,0}\psi=0$ on $M\setminus D$, i.e., $\nabla^{1,0}\psi$ is a holomorphic vector field on 
$\MsmD$. Using the asymptotic expansion,  $\nabla^{1,0}\psi$ is continuous up to $D$, and hence extends holomorphically 
to  $M$. Now, the coefficient of $\frac{\del}{\del\zeta}$ equals $g^{1\jbar}\psi_{\jbar}$.
By \eqref{gupperijEq} $g^{1\jbar}=O(r^{\eta'})$, hence vanishes on $D$ for $j\ne1$, (and $\psi$ is infinitely differentiable 
in the $j\ne 1$ directions), while although $g^{1\b1}$ is uniformly positive,  $\psi_{\b 1}=O(r^{\onebe-1})$, so this 
terms also vanishes on $D$. In conclusion the $\frac{\del}{\del\zeta}$ component of $\nabla^{1,0}\psi$ vanishes at $D$,
so this vector field is tangent to $D$. 
\end{proof}

We now estimate the Sobolev constant.  First observe that the Sobolev inequality holds 
for the model edge metric $g_\beta$, i.e., since $\dim M = 2n$,  
\begin{equation}
\label{regularSobolevEq}
||f||_{L^{\frac{2n}{n-1}}(M,g_\beta)}\le C_S ||f||_{W^{1,2}(M,g_\beta)},
\end{equation}
and hence also for any metric uniformly equivalent to it. One way to prove this is to note that it suffices
to prove this inequality locally, in the neighbourhood of any point; away from $D$ this is just
the standard Sobolev inequality, while in a neighbourhood of any point $p \in D$ we can
use the $(\zeta,Z)$ coordinate system to reduce to the standard Euclidean case. An alternate proof
relies on the well-known equivalence of the Sobolev inequality with the fact that the heat kernel
for the scalar Laplacian blows up like $t^{-n}$ as $t \searrow 0$ (since the overall dimension is $2n$). 
Since $g_\beta$ is a product of a cone with a Euclidean space, this, in turn, reduces to the fact that the 
heat kernel on a two-dimensional cone blows up like $t^{-1}$, which can be verified by direct computation, see, e.g., \cite{D}. 

As an aside, observe that using \eqref{regularSobolevEq}, 
the standard Moser iteration proof of the 
$\calC^0$ estimate for $s=0$ \cite{Tbook} goes through exactly as in the smooth case, and hence can be 
used instead of \Kolodziej's estimate to prove Lemma \ref{AYUniquenessLemma}.

Next, we derive a uniform Sobolev inequality when $s > \eps$.  Our approach follows Bakry \cite{Bakry} closely,
and relies on the general theory of diffusive semigroups. The following result is essentially a special case of 
\cite[Theorem 6.10]{Bakry}.
\begin{prop}
\label{SobolevProp}
Let $\eps\in(0,1)$ and $s\in(\eps,1]$. There exist a uniform constant $C>0$ depending only
on $(M,\o), n$ and $\eps$ so that for any $f\in W^{1,2}(M,\o_{\vp(s)})$,  
\[
||f||_{L^{\frac{2n}{n-1}}(M,\o_{\vp(s)})} \le C||f||_{W^{1,2}(M,\o_{\vp(s)})}.
\]
\end{prop}
\begin{proof}
Let $L=\Delta_{\ovp}$. Proposition 2.1 in \cite{Bakry} (which holds for substantially more general operators $L$) asserts that if 
$\calA \subset \Dom_{\mathrm{Fr}}(L)$ (recall \eqref{FriedrichsEq}) is a subspace preserved by $L$ and $e^{tL}$ and dense in $L^2$,
then it is also dense in $\Dom_{\mathrm{Fr}}(L)$ with respect to the graph norm $||f||_{L^2}+||Lf||_{L^2}$. We can also verify this directly
in our setting, and in fact have already done so in the proof of Lemma~\ref{density} above, but cf.\ also the discussion 
in \cite[p. 35]{Bakry}.

Now, for any two functions $f,g\in\calA$ define the quantities
\beq\label{GEq}
2\G(f,g):=L(fg)-fLg-gLf,
\eeq
and
\beq\label{G2Eq}
2\G_2(f,g):=L\G(f,g)-\G(f,Lg)-\G(g,Lf).
\eeq
Note that on the smooth part $M\sm D$, 
\beq\label{GForEq}
\G(f,f)=|\nabla f|^2
\eeq
(the gradient and norm are taken with respect to $\ovp$), and
\beq\label{G2ForEq}
\G_2(f,f)=\frac12L|\nabla f|^2-\nabla f.\nabla Lf= \Ric\ovp(\nabla f,\nabla f)+|\nabla^2 f|^2.
\eeq
Since $f\in\calA$, \eqref{GForEq} holds on all of $M$ as a $W^{1,2}$ distribution. Furthermore, by virtue of 
\eqref{RicAlongCMEq}, \eqref{G2ForEq} implies that 
\beq\label{G2For2ndEq}
\G_2(f,f) \ge C\eps\, \G(f,f)+\frac1{2n}(L f)^2
\eeq
in the sense of distributions on all of $M$, where $C=C(n)$ is a universal constant.  Both of these assertions can
be checked easily using that $f$ is \phgs\ with an expansion $f=a_0(y) +r^{\frac1\be}(a_1(y)\cos\th+a_2(y)\sin\th)+
a_2(y)r^2+O(r^{2+c})$, for some $c>0$. 

Following the definition and notation of \cite[p.93]{Bakry}, we have proved that $L$ satisfies the ``uniform curvature-dimension 
condition'' ${\h{CD}(C\eps,2n)}$. We can then follow the general argument in \cite{Bakry} to obtain uniform Sobolev 
bounds \cite[Theorem 6.10]{Bakry}, cf.\ also \cite[Theorem 1]{BakryLedoux}.  This procedure also leads to a uniform 
Poincar\'e estimate; however, an examination of the proof of \cite[Proposition 6.3]{Bakry} shows that this is 
essentially equivalent to the one given above in Lemma \ref{LambdaoneLemma} (i).   In any case, we now sketch 
Bakry's reasoning and explain in detail why it applies here. 

The first point is that it suffices to prove the uniform Sobolev inequality only for functions in $\calA$. Indeed, using
the density of $\calA$ in the graph norm, we must show that both sides in the Sobolev inequality 
are continuous in this topology.  For the right hand side, this is obvious.  Since there is {\it some} (not necessarily uniform)
Sobolev inequality, cf.\ the paragraph containing \eqref{regularSobolevEq}, $W^{1,2}$
is contained in $L^{\frac{2n}{n-1}}$ and so a sequence converging in $W^{1,2}$ converges weakly in $L^{\frac{2n}{n-1}}$, and so
the left hand side is also continuous in the appropriate sense.  

We now show uniform control of the Sobolev constant. Fix $2<p<\frac{2n}{n-1}$  and $\delta>0$ and let $f^{(p)}_k \in \calA$ 
be a sequence which converges towards the supremum of the ratio 
\begin{equation}
\frac{||F||_{p}^2-(1+\delta)||F||_{2}^2}{\G(F,F)}
\label{sr}
\end{equation}
over $F\in\Dom_{\mathrm{Fr}}(L)$. Denote this supremum by $\gamma_p$. As usual, we can assume that $f^{(p)}_k\ge0$ 
and $||f^{(p)}_k||_{2}=1$. Using the compactness of $W^{1,2}$ in $L^{p}$ (that is a consequence of the existence of a
Sobolev inequality; for a general semigroup, this compactness is not automatic and is proved in \cite[Theorem 4.11]{Bakry}), 
we can extract a subsequence converging weakly in $\Dom_{\mathrm{Fr}}(L)$ and strongly in $L^p$ to a nontrivial
limit function $f^{(p)} \geq 0$, which we call $f$ for simplicity. This satisfies
$ ||f||_{p}^2=(1+\delta)+\gamma_p \G(f,f). $ 
Assuming that we have normalized the measure associated to $\o^n_{\vp(s)}$ to have unit volume, then $f$ must
be nonconstant since $\delta>0$. Since $f$ maximizes \eqref{sr}, the usual argument in the calculus of variations gives 
$
||f||^{2-p}_p\langle f^{p-1},g\rangle =(1+\delta)\langle f,g\rangle+\gamma_p \G(f,g) =\langle f,(1+\delta)g-\gamma_p Lg\rangle,
$
for any $g\in\Dom_{\mathrm{Fr}}(L)$, or equivalently 
$||f||^{2-p}_p\langle f^{p-1},R_\lambda(h)\rangle =\langle f,\gamma_ph\rangle$.
Here $\lambda:=(1+\delta)/\gamma_p$ and $g=R_\lambda(h)$ where $R_\lambda=\int_0^\infty e^{-\lambda t}e^{tL}dt=(\lambda I-L)^{-1}$ 
is the resolvent of $L$. This shows that $f=R_\lambda(||f||^{2-p}_pf^{p-1})/\gamma_p$, or equivalently
\beq
\label{fResolventEq} 
||f||^{2-p}_pf^{p-1}=((1+\delta)-\gamma_p L)f. 
\eeq
Following \cite[\S 3.2]{ACM}, the solution to this subcritical Yamabe-type equation must be polyhomogeneous.
Then, by a determination of the leading terms in the expansion of $f$,
it readily follows that $f\in\Dw$. Both these asserations are simpler analogues of Theorem \ref{PhgMainThm} and
Proposition \ref{PreciseAsympExpansionProp}, and their proof follows similar, but simpler, arguments  since this is
a quasilinear equation, and not a fully nonlinear one. Thus, $f\in\calA$.
\begin{lem}
The constant $\gamma_p$ associated to the embedding $W^{1,2}\subset L^p$ satisfies 
\beq\label{SobolevEstimateEq}
\gamma_p\le \frac{(2n-1)(p-2)(1+\delta)}{2nC\eps}.
\eeq
\end{lem}
\begin{proof}
Fix $a\in\RR$. We let $g$ be such that $g^a=f\equiv f^{(p)}$. We then divide \eqref{fResolventEq} by $f$ and then substitute $f=g^a$ to get 
\beq\label{Eq1}
\baeq
||f||^{2-p}_pg^{a(p-2)}
&=1+\delta-\gamma_p g^{-a}[ag^{a-1}Lg+a(a-1)g^{a-2}\G(g,g)],
\cr
&=1+\delta-a\gamma_p [g^{-1}Lg+(a-1)g^{-2}\G(g,g)],
\eaeq
\eeq
Now, following Bakry, we multiply this by $-gLg$ and integrate:
\begin{equation*}
\baeq
-||f||^{2-p}_p\langle g^{1+a(p-2)},Lg\rangle
&=(1+\delta)\G(g,g)+a\gamma_p||Lg||^2+a\gamma_p(a-1)\langle \frac{Lg}g,\G(g,g)\rangle.
\eaeq
\end{equation*}
The left hand side can be rewritten as
$$
||f||^{2-p}_p\G(g^{1+a(p-2)},g)=C_p(1+a(p-2))\langle g^{a(p-2)},\G(g,g)\rangle.
$$
This can be rewritten using \eqref{Eq1} as
$$
(1+a(p-2))\langle 1+\delta-a\gamma_p [g^{-1}Lg+(a-1)g^{-2}\G(g,g)],\G(g,g)\rangle.
$$
Altogether, we have
\begin{equation*}
\baeq
&(1+a(p-2))\Big\langle 1+\delta-a\gamma_p [g^{-1}Lg+(a-1)g^{-2}\G(g,g)],\G(g,g)\Big\rangle
\cr
&=(1+\delta)\G(g,g)+a\gamma_p||Lg||^2+a\gamma_p(a-1)\langle \frac{Lg}g,\G(g,g)\rangle.
\eaeq
\end{equation*}
So the constant $||f||^{2-p}_p$ disappears; from this point on we follow Bakry, and as in \cite[(6.37)]{Bakry}, we obtain
$$
\frac{1+\delta}{\gamma_p}(p-2)||\G(g,g)||_1
=
||Lg||_2^2+a(p-1)\langle Lg/g,\G(g,g)\rangle+(a-1)(1+a(p-2))||\G(g,g)/g||_2^2.
$$
We now invoke a consequence of \eqref{G2For2ndEq}, which holds by the chain rule \cite[(6.38)]{Bakry}: for any $b \in \RR$,  
$$
\G_2(g,g)+b\G(g,\G(g,g))/g+b^2(\G(g,g)/g)^2\ge C\eps\G(g,g)+\frac1{2n}(Lg+b\G(g,g)/g)^2.
$$
Integrating gives 
$$
\Big(\frac{1+\delta}{\gamma_p}(p-2)-\frac{2n}{2n-1}C\eps\Big)\G(g,g)\ge
\Big((a-1)(1+a(p-2))-b(b+2n/(2n+1))\Big)
||\G(g,g)/g||_2^2.
$$
Choosing $a,b$ appropriately as in \cite[p. 110]{Bakry}, we see that the right hand side is nonnegative, 
which implies a uniform bound on $\gamma_p$ since $\G(g,g)\ge0$. 
\end{proof}
Letting $p \nearrow 2n/(n-1)$, and using the fact that there is a Sobolev inequality at the critical exponent,
we see that this Sobolev constant has the upper bound \eqref{SobolevEstimateEq}
with $p=2n/(n-1)$. This concludes the proof of Proposition \ref{SobolevProp}. 
\end{proof}
\begin{remark}
{\rm
In fact, \cite[Theorem 3]{BakryLedoux} shows that we can find a uniform bound for the diameter of $(M,\o_{\vp(s)})$
from Proposition \ref{SobolevProp}. Indeed, define
$$
D(\G):=\sup\{|f(x)-f(y)|\,:\, x,y\in M, f\in A, ||\G(f,f)||_{L^\infty(M)}\le1\},
$$
and apply the Sobolev inequality to the functions $(1+\lambda f)^{1-\frac n2}$ for any $f\in\cal A$. (One must check that such 
functions are once again in $\calA$.) It then follows that $D(\G)\le C\eps^{-1/2}.$}
\end{remark}

\begin{remark}
{\rm
There are other possible approaches to the estimation of the Sobolev constants. One approach, suggested in a remark in 
the first version of this article, is to approximate $\o_{\vp(s)}$ by smooth \K metrics with a uniform positive
lower bound on the Ricci curvature. This has been carried out in detail in \cite{T12,CDS}. Another approach is 
to show that as a metric-measure space, the completion of $(M\sm D,\ovp,\ovp^n)$ satisfies a uniform 
(generalized) doubling property. The arguments of \cite{HK,HindePetersen} show that the Poincar\'e inequality 
implies a Sobolev inequality. This was described in detail in an earlier version of this paper, but for brevity 
we have replaced this by the semigroup approach above. 
}
\end{remark}

\subsection{Energy functionals}
Unlike in the previous cases, there are well-known obstructions to obtaining a $\calC^0$ estimate in the positive case. 
The existence of such an estimate is then described in terms of the behavior of certain energy functionals.
For more background we refer  to \cite{Aubin1984,BM,T87,Tbook}. 

The energy functionals $I,J$, introduced by Aubin \cite{Aubin1984}, are defined by \vglue-9pt
$$
\begin{aligned}
I(\o,\ovp) & =\frac1V\int_M\i\del\vp\w\dbar\vp\w\sum_{l=0}^{n-1}\o^{n-1-l}\w\ovp^{l}
=\frac1V\int_M\vp(\on-\ovpn),
\cr J(\o,\ovp) & =\frac{V^{-1}}{n+1}\int_M\i\del\vp\w\dbar\vp\w\sum_{l=0}^{n-1}(n-l)\o^{n-l-1}\w\ovp^{l}.
\end{aligned}
$$
This definition certainly makes sense for pairs of smooth \K forms, and by the continuity of the mixed 
\MA operators on $\PSH(M,\o_0)\cap \calC^0(M)$ \cite[Proposition 2.3]{BT}, 
these functionals can be uniquely extended to pairs $(\o_0,\ovp)$,
with $\o_0$ smooth and $\ovp\in\calH_{\o_0}$, and hence 
also to $\calH_{\o}\times\calH_\o$, where now by $\o$ we mean 
the reference metric given by \eqref{oDefEq}.
These functionals are nonnegative and equivalent, 
\begin{equation}
\label{IIminusJComparisonEq}
\frac1n J\le I-J\le \frac{n}{n+1} I\le nJ.
\end{equation}
One use of these functionals is in deriving a conditional $\calC^0$ estimate.

\begin{lem}
\label
{PositiveCaseUniquenessCzeroLemma}
Let $s\in(0,\mu)$. Any $\calC^0(M)\cap \PSH(M,\o)$ solution 
$\vp(s)$ to (\ref{RCMEq}) is unique.
Moreover, if $\vp(s)\in \calD^{0,\gamma}_s$ then 
$||\vp(s)||_{\calC^0(M)}\le C(1+I(\o,\o_{\vp(s)}))$, for all $s\in(\eps,\mu)$.
\end{lem}

\begin{proof}
The uniqueness is due to Berndtsson \cite{Bern}. 

We now prove the estimate. Using the uniform estimates on the Poincar\'e 
and Sobolev constants, the arguments proceed much as in the 
smooth case \cite[Lemma 6.19]{Tbook}. 

First, let $G_\o$ be the Green function of $-\Delta_\o$, i.e., $-\Delta_\o G_\o = -G_\o \Delta_\o = \mbox{Id} - \Pi$,
where $\Pi$ is the orthogonal projector onto the constants.  (Note that this is contrary to our previous 
sign convention for $G$, but conforms with the usual convention for this estimate.) 
Necessarily, $\int_M G_{\o}(\cdot,\tilde{z})\o^n(\tilde{z})=0$.  We claim that $A_\o:=-\inf_{M\times M} G_{\o}<\infty$.
Assuming this for the moment, we can write 
\[
\vp_{s}(z) = V^{-1}\int_M \vp_s\o^n -\int_M G_\o(x,y)\Delta_\o\vp_s(y)\o^n(y).
\]
Hence, since $-n<\Delta_\o\vp_s$,
\begin{equation}
\label{SupEstimateEq}
\sup\vp(s) \le \frac 1V\int_M \vp(s)\o^n+nV\!A_\o,
\end{equation}
To prove this claim about the Green function, recall that
\[
G(z,\tilde{z}) =  \int_0^\infty ( H(t, z, \tilde{z}) - \Pi(z,\tilde{z}))\, dt,
\]
where $H$ is the heat kernel associated to this (Friedrichs) Laplacian, 
and $\Pi(z,\tilde{z})$ is the Schwartz kernel of this rank one projector. 
This integral converges absolutely for any $z \neq \tilde{z}$. We rewrite this as
\begin{equation}
\label{GreenKernelEq}
G(z, \tilde{z}) = \int_0^1 H(t,z,\tilde{z})\, dt - \Pi(z, \tilde{z}) + \int_1^\infty (H(t, z, \tilde{z}) - \Pi(z, \tilde{z}))\, dt.
\end{equation}
It follows easily from standard estimates that the integral from $1$ to $\infty$ converges to a bounded function.
On the other hand, by the maximum principle, $H > 0$, 
so the first term on the right is nonpositive. 
Finally, $\Pi(z,\tilde{z})=V^{-1}$ is just a constant, so $G$ is bounded below. 

To conclude the proof, it suffices to prove $-\inf\vp(s) \le -\frac CV\int_M \vp(s)\o_{\vp(s)}^n$
(indeed, $\vp(s)$ changes sign by the normalization (\ref{foDefEq})
of $f_\o$, so $||\vp(s)||_{\calC^0(M)}\le \h{osc}\,\vp(s)$).
This can be shown in one of two ways.
The first is by noting that Bando--Mabuchi's Green's
function lower bound \cite{BM} extends to our present setting,
and thus $A_{\o(s)}<C$ uniformly in $s$ and 
$-\inf\vp(s)\le -\frac1V\int_M\vp(s)\o_{\vp(s)}^n+nVC$.
Indeed, the proof of their bound relies on an estimate
of Cheng--Li \cite{ChengLi} of the heat kernel 
$H_{\o(s)}(t,z,\tilde z)-V^{-1}\le Ct^{-n}$ with
$C$ depending only on terms of the Poincar\'e and Sobolev constants,
and hence independent of $s>\eps$. Thus, by \eqref{GreenKernelEq}
$A_{\o(s)}<C$, as desired. The second uses Moser iteration, as in \cite{T87}.
\end{proof}

\subsection{Mabuchi's K-energy and Tian's invariants}
\label{TianMabuchiSubSec}

Define the twisted Mabuchi K-energy functional  by
integration over paths $\{\o_{\vp_t}\}\subset \calH_{\o_0}$ smooth in $t$,
\begin{equation}
\begin{aligned}
\label{KenergyVariationEq}
E_0^\beta(\o,\o_\vp) &:=-\frac1V\int_{M\times[0,1]}\dot\vp_t\Delta_{\vp_t} f_{\vp_t}\o^n_{\vp_t}\w dt.
\end{aligned}
\end{equation}
Its critical points are \KE edge metrics.
The following is an extension of a formula of Tian
\cite[p. 254]{T1994},\cite[(5.12)]{T}
(cf. \cite{Berm,Li})
to the twisted setting. 
In particular it shows that 
$E_0^\beta$ is well-defined on $\calH_{\o_0}\times\calH_{\o_0}$. 
The proof, as others in this subsection, are straightforward extensions of their
counterparts from the smooth setting, and are included for the
reader's convenience.

\begin{lem}
One has,
\begin{equation}
\label{EbetaFormula}
E^\be_0(\o,\o_{\vp})=
\frac1V\int_M\log\frac{\o_{\vp}^n}{\on}\o_{\vp}^n
-\mu(I-J)(\o,\o_{\vp})
+\frac1V\int_M f_\o(\on-\o_{\vp}^n).
\end{equation}
\end{lem}

\begin{proof}
For any smooth (in $t$) path $\{\o_{\vp_t}\}\subset\calH_\o$
connecting $\o$ and $\ovp$ \cite[p. 70]{Tbook},
\begin{equation}
\begin{aligned}
\label{IJVariationalEq}
(I-J)(\o,\o_{\vp_1}) & 
=
-\frac1V
\int_{M\times [0,1]}\vp_t\Delta_{\vp_t}\dot\vp_t\o_{\vp_t}\w dt.
\end{aligned}
\end{equation}
Hence the variation of the right hand side of (\ref{EbetaFormula}) equals
$$
\int_M\Delta_\vp\dot\vp\Big(
\log\frac{\ovpn}{\on}+1
+\mu\vp-f_\o
\Big)\ovpn,
$$
and this coincides with $dE^\beta_0(\dot\vp)$ 
since $f_{\ovp}=f_\o-\mu\vp-\log\frac{\ovpn}{\on}+c_\vp$ with $c_\vp$ a constant.
The formula then follows since both sides
vanish when $\ovp=\o$.
\end{proof}

As noted in the Introduction, a key property of the continuity path (\ref{RCMEq}) is
the monotonicity of $E^\be_0$. Monotonicity of similar twisted K-energy functionals was 
noted, e.g., in \cite{R}, and the following is the analogue of \cite[Lemma 9.3]{R}.

\begin{lem}
\label
{KEnergyMonotonicityLemma}
$E_0^\be$ is monotonically decreasing along the continuity path (\ref{RCMEq}).
\end{lem}

\begin{proof}
By (\ref{RicAlongCMEq}), $\i\ddbar f_{\ovp}=-(\mu-s)\i\ddbar\vp$,
and from (\ref{RCMEq}) we have $(\Delta_\vp+s)\dot\vp=-\vp$. It
follows that
$$
\frac{d}{ds}E^\be_0(\o,\o_{\vp(s)})
=
-\frac{\mu-s}V\int_M\dot\vp\Delta_\vp(\Delta_\vp+s)\dot\vp\ovpn,
$$
and this is nonpositive by the positivity of $\Delta^2_\vp+s\Delta_\vp$,
which is immediate for $s\le 0$, and follows from Lemma \ref{LambdaoneLemma}, when
$s\in(0,\mu)$.
\end{proof}

Following Tian \cite{T97}, we say that $E_0^\be$ is proper if $\lim_{j\ra\infty}(I-J)(\o,\o_j)=\infty$ implies that necessarily,
$\lim_{j\ra\infty}  E_0^\be(\o,\o_j)=\infty$. From Lemmas \ref{PositiveCaseUniquenessCzeroLemma} and \ref{KEnergyMonotonicityLemma} we have:
\begin{cor}
\label
{PropernessCzeroCor}
Let $\vp(s)\in \Ds\cap \PSH(M,\o)$. If $E^\be_0$ is proper then $||\vp(s)||_{\calC^0(M)}\le C$, 
independently of $s\in(\eps,\mu)$. 
\end{cor}

We also note that as observed by Berman \cite{Berm}, 
an alternative proof of Corollary \ref{PropernessCzeroCor}
follows by combining \Kolodziej's estimate \cite{K1} and 
the following result contained in \cite[Lemma 6.4]{BBGZ}
and \cite{Berm} (note that $\vp(s)$ change sign).

\begin{lem} 
{\rm \cite{BBGZ,Berm}}
\label{BermanLemma}
Suppose $J(\o,\ovp)\le C$.
Then for each $t>0$ there exists $C'=C'(C,M,\o,t)$
such that $\int_M e^{-t(\vp-\sup\vp)}\on\le C'$.
\end{lem}

We next recall the definition of Tian's invariants \cite{T87,T92} 
$$
\begin{aligned}
\quad\alpha_{\O,\chi}
&:=\sup\Big\{\; a \,:\, \sup_{\vp\in\PSH\cap \calC^\infty(M,\o_0)}\int_M 
e^{-a(\vp-\sup\vp)}\chi^n<\infty 
\Big\},
\qquad
\a(M):=\a_{c_1(M),\o_0},
\cr
\be_{\O,\o} &:=\sup\, \{\; b \,:\, \Ric\chi\ge b\chi, \; \h{for some\ }\chi\in\calH_{\o}  \}, 
\qquad
\cr
\beta(M) &:=\sup\{\; b \,:\, \Ric\lambda\ge b\lambda, 
\;\h{for some\ } \lambda\in\calH^\infty_{c_1}  \}, \!\!\!
\end{aligned}
$$
where the measure $\chi^n$ is assumed to have density in $L^p(M,\o_0^n)$, for some $p>1$,
and where, for emphasis, $\calH_{\o}$ is given by \eqref{KEdgePotentialsSpace} 
and, when $M$ is Fano, $\calH^\infty_{c_1}$ denotes the space of 
smooth \K forms representing $c_1(M)$
(and finally, as always 
$\O=\frac1\mu c_1(M)-\frac{1-\be}\mu c_1(L_D)$ with 
$\O=[\o_0]=[\o]$, $\o_0$ a smooth \K form,
and $\o=\o(\be)$ the reference \K edge current).
These invariants are always positive as shown by Tian when $\chi^n$ is smooth, and hence
by the \Holder inequality also in general. 
For some relations between $\a_{\O,\o}$ and $\a_{\O,\o_0}$ 
we refer to \cite{Berm} where such invariants for singular measures
were studied in depth.

\begin{lem}
\label{alphainvariantepsLemma}
Suppose that $\a_{\O,\o}-\frac{n\mu}{n+1}>\eps$. Then $E_0^\be\ge \eps I-C$,
for some $C\ge0$.
\end{lem}

\begin{proof}
Again we follow the classical argument
\cite[p. 164]{T},\cite[p. 95]{Tbook}.
Note that for any $a\in(0,\a_\o)$ there exists a constant $C_a$
such that
$
\frac1V\int_M\log\frac{\ovpn}{\on}\ovpn\ge aI(\o,\ovp)-C_a.
$
Indeed, by (\ref{SupEstimateEq})
and Jensen's inequality,
$$
\begin{aligned}
e^{C_a} &\ge \frac1V\int_M e^{-a(\vp-\frac1V\intM\vp\on)}\on
\cr &= \frac1V\int_M e^{-\log\frac{\ovpn}{\on}-a(\vp-\frac1V\intM\vp\on)}\ovpn
\ge e^{-\frac1V\intM(\log\frac{\ovpn}{\on}+a(\vp-\frac1V\intM\vp\on))\ovpn}.
\end{aligned}
$$
By (\ref{IIminusJComparisonEq})
and (\ref{EbetaFormula}) it then follows that
$E_0^\beta\ge (a-\frac{n\mu}{n+1})I-C$.
\end{proof}

\begin{cor}
\label{alphainvariantCzeroCor}
For all $s\in (-\infty,\frac{n+1}n\alpha_{\O,\o})\cap (-\infty,\mu]$,
we have $||\vp(s)||_{\calC^0(M)}\le C$, with $C$ independent of $s$.
\end{cor}

\begin{proof}
When $\a_{\O,\o}>\frac{n\mu}{n+1}$ the result follows from
Lemma \ref{alphainvariantepsLemma} and Corollary \ref{PropernessCzeroCor}
(note, as explained in \S\ref{ProofKESection}, that there is no difficulty in treating
the interval $s\in(0,\eps)$ for some $\eps>0$).
In general the classical derivation \cite{Tbook} carries over.
\end{proof}

This conditional $\calC^0$ estimate implies of course, given
the other ingredients of the proof of Theorem \ref{ConicKEMainThm}, 
that $\be_{\O,\o}\ge\min\{\mu,\frac{n+1}n\a_{\O,\o}\}$, just as in the smooth setting.
We remark that Donaldson \cite{Don2009} conjectured that 
when $D\subset M$ is a smooth anticanonical divisor of a Fano manifold, then
$
\be(M)=\sup\,\{\,\be\,:\, \eqref{KEMAEq} \h{\ admits a solution with \ } \mu=\be\, \}.
$
Note that the Calabi--Yau theorem shows that the left hand side is positive, while 
Corollary \ref{FanoCor} shows that the right hand side is positive.
Our results have further direct bearings on this problem that we will discuss
elsewhere.

\section{The Laplacian estimate}
\label{LaplacianSection}
Let $f:M\ra N$ be a holomorphic map between two complex manifolds.
The Chern--Lu inequality was originally used by Lu to bound $|\del f|^2$ 
when the target manifold has negative bisectional curvature \cite{Lu} under
some technical assumptions. This inequality was later used by Yau \cite{Y1} 
together with his maximum principle to greatly generalize the result to the case 
where $(M,\o)$ is a complete \K manifold with a lower bound $C_1$ on the Ricci curvature,
and $(N,\eta)$ is a Hermitian manifold whose bisectional curvature is bounded above 
by a negative constant $-C_3$. These results lead to Yau's Schwarz lemma, which says
that the map $f$ decreases distances in a manner depending only on $C_3>0$ and $C_1$.

In a related direction, the use of the Chern--Lu inequality to prove a Laplacian estimate for 
complex \MA equations seems to go back in print at least to Bando--Kobayashi \cite{BK},
who considered the case $\Ric\o\ge -C_2 \eta$ and $C_3$ arbitrary (not necessarily positive)
in the context of constructing a Ricci flat metric on the complement of a divisor. Next, the case 
$\Ric\o\ge -C_1\o$ (and again $C_3$ arbitrary) was used in proving the a priori Laplacian estimate 
for the Ricci iteration \cite{R}. 

The point of Proposition \ref{ChernLuApplicationProp} below is to state the Chern--Lu inequality in a unified 
manner that applies to a wide range of \MA equations that appear naturally in \K geometry. It makes 
the Laplacian estimate in these settings slightly simpler, and the explicit dependence on the geometry 
more transparent. 
In addition, the Chern--Lu inequality applies in some situations where the standard derivation \cite{Au,Y2,Siu} 
of the Aubin--Yau Laplacian estimate may fail (as in the case of the Ricci iteration) or give an estimate with 
different dependence on the geometry (which will be crucial in our setting).
While Proposition \ref{ChernLuApplicationProp} below should be folklore 
among experts, it seems that it is less well-known than it deserves to be.
In particular, we are not aware of a treatment of the Aubin--Yau or Calabi--Yau Theorems that uses it.

\subsection{The Chern--Lu inequality}
Let $(M,\o)$, $(N,\eta)$ be compact \K manifolds and let $f:M\ra N$ be a holomorphic map with 
$\del f\neq0$. The Chern--Lu inequality \cite{Chern,Lu} is
\begin{equation}
\label{ChernLuEq}
\Delta_{\o} \log |\del f|^2\ge \frac{\Ric\o\otimes\eta(\del f,\dbar f)}{|\del f|^2}
-\frac{\o\otimes R^N(\del f,\dbar f,\del f,\dbar f)}{|\del f|^2},
\quad \h{on\ $M$}.
\end{equation}
Since the original statement \cite[(7.13)]{LuThesis},\cite[(4.13)]{Lu} contains a misprint, 
we include a direct and slightly simplified derivation (since we restrict to the \K setting) for completeness. 
We note also that \eqref{ChernLuEq} can be obtained as a special case of a formula of Eells--Sampson \cite[(16)]{EellsSampson}
on the Laplacian of the energy density of a harmonic map.

Write $\del f:T^{1,0}M\ra T^{1,0}N$. Then $\del f$ is a section of $T^{1,0\,\star}M\otimes T^{1,0}N$ given in local holomorphic
coordinates by $\del f=\frac{\del f^i}{\del z^j}dz^j\otimes\frac{\del}{\del w^i}$.
With respect to the metric induced by $\o$ and $\eta$ on the product bundle above,
\begin{equation}
\label{utraceEq}
u:=|\del f|^2=g^{i\b l}h_{j\b k}\frac{\del f^j}{\del z^i}\overline{\frac{\del f^k}{\del z^l}}.
\end{equation}
Compute in normal coordinates at a point 
$$
\begin{aligned}
\Delta_\o u=
\sum_pu_{p\b p}
& =
-\sum_{i,l,p}g_{l\b i, p\b p}h_{j\b k} f^j_{,i}f^{\b k}_{,\b l}
+
\sum_{i,p}h_{j\b k,d\b m}f^d_{,p}f^{\b m}_{,\b p}f^j_if^{\b k}_{,\b i}
+
\sum_{i,j,p}f^j_{,ip}f^{\jbar}_{,\b i\b p}
\cr
& =
\Ric\o\otimes\eta(\del f,\dbar f)-\o\otimes R^N(\del f,\dbar f,\del f,\dbar f)
+\sum_{i,j,p}f^j_{,ip}f^{\jbar}_{,\b i\b p}.
\end{aligned}
$$
By the Cauchy--Schwarz inequality,
\[
u\sum_{i,j,p}f^j_{,ip}f^{\jbar}_{,\b i\b p}\ge \sum_ku_ku_{\b k},
\]
and since $\Delta_\o \log u={\Delta_\o u}/u-{\sum_ku_ku_{\b k}}/{u^2}$, the desired inequality follows.

One particularly useful form of the Chern--Lu inequality is when $f$ is the identity map.

\begin{prop}
\label{ChernLuApplicationProp}
In the above, let $f=\h{\rm id}:(M,\o)\ra (M,\eta)$ be
the identity map, and assume that $\Ric\o\ge -C_1\o-C_2\eta$
and that $\h{\rm Bisec}_\eta\le C_3$, for some $C_1,C_2,C_3\in\RR$. 
Then,
\begin{equation}
\label{FirstTraceChernLuIneq}
\Delta_{\o} \log |\del f|^2\ge -C_1-(C_2+2C_3)|\del f|^2.
\end{equation}
In particular, if $\o=\eta+\i\ddbar\vp$ then
\begin{equation}
\label{ChernLuIneqPotentialEq}
\Delta_{\o} \big(
\log \tr_\o\eta
-(C_2+2C_3+1)\vp
\big)
 \ge -C_1-(C_2+2C_3+1)n + \tr_\o\eta.
\end{equation}
Hence, $\o\ge C\eta$ for some $C>0$ depending only on
$C_1,C_2,C_3,n$ and $||\vp||_{\calC^0(M)}$.
\end{prop}

\begin{proof}
By (\ref{utraceEq}), $u=\tr_\o\eta$. The assumption on $\Ric\o$ implies that
$$
\begin{aligned}
\Ric\o\otimes\eta(\del f,\dbar f)
& =
g^{i\bar l}g^{k\jbar}R_{i\jbar}h_{k\b l}
\ge
-C_1 g^{i\bar l}g^{k\jbar}g_{i\jbar}h_{k\b l}
-C_2 g^{i\bar l}g^{k\jbar}h_{i\jbar}h_{k\b l}
\cr
& =-C_1\tr_\o\eta-C_2(\eta,\eta)_\o
\ge -C_1\tr_\o\eta-C_2(\tr_\o\eta)^2,
\end{aligned}
$$
where the last inequality follows from the 
identity \cite[Lemma 2.77]{Besse}
$$
(\eta,\eta)_\o=(\tr_\o\eta)^2-n(n-1)\frac{\eta\w\eta\w \o^{n-2}}{\o^n}.
$$

Similarly, we also have
$$
\begin{aligned}
-\o\otimes R^N(\del f,\dbar f,\del f,\dbar f)
& = -g^{i\jbar}g^{k\b l}R^N_{i\jbar k\b l}
\cr
& \ge -C_3 g^{i\jbar}g^{k\b l} (h_{i\jbar}h_{k\b l}+h_{i\b l}h_{k\jbar})
\ge -2C_3(\tr_\o\eta)^2.
\end{aligned}
$$
Thus, (\ref{FirstTraceChernLuIneq}) follows from (\ref{ChernLuEq}).
Since $\tr_\o\eta=n-\Delta_\o\vp$, equation (\ref{ChernLuIneqPotentialEq})
follows from (\ref{FirstTraceChernLuIneq}).
\end{proof}

\subsection{The Laplacian estimate in the singular H\"older spaces}
We now apply the Chern--Lu inequality to obtain an a priori Laplacian estimate for the continuity 
method (\ref{TwoParamCMEq}). For solutions of \eqref{RCMEq} it gives a bound depending,
in addition to the uniform norm, on an upper bound on the bisectional curvature of the reference
metric;  in contrast the well-known Aubin--Yau bound depends on a lower bound 
for the bisectional curvature \cite{Au,Y2,Siu}.
\begin{lem}
\label{TYLaplacianEstimateLemma}
Suppose that there exists a reference metric $\o\in\calH_{\o_0}$ 
with $\Ric\o\ge -C_2\o$ and $\h{\rm Bisec}_\o\le C_3$, on $\MsmD$,
where $C_2\in\RR\cup\{\infty\},\, C_3\in\RR$.
Let $s>S$. Solutions to \eqref{TwoParamCMEq} in $\Ds\cap \calC^4(M\sm D)\cap \PSH(M,\o)$, satisfy
\begin{equation}
\label{AYTYCtwoEq}
\begin{aligned}
||\Delta_\o\vp(s,t)||_{\calC^0(M)} &\le C=C(||f_\o||_{\calC^0(M)},
||\vp(s,t)||_{\calC^0(M)},S,(1-t)C_2,C_3),
\end{aligned}
\end{equation}
where $(1-t)C_2$ is understood to be 0 when $t=1$.
Moreover, $\frac1C \o\le \o_{\vp(s)}\le C\o$.
\end{lem}
\begin{proof}
Along the continuity path (\ref{TwoParamCMEq}), 
$$
\Ric\o_{\vp}=(1-t)\Ric\o+s\o_{\vp}+(\mu t-s)\o+2\pi(1-\be)[D]
\ge S\o_{\vp_{s,t}}-(1-t)C_2\o,
$$
Hence, the assumptions of Proposition \ref{ChernLuApplicationProp} are satisfied
(we take $\h{id}:(M,\o_\vp)\rightarrow (M,\o)$),
and the desired estimates follow directly from (\ref{ChernLuIneqPotentialEq})
if the maximum of (the bounded function) $\log\tr_{\o_{\vp}}\o-A\vp$ takes place in $\MsmD$.

Next, suppose the maximum is attained on $D$. We claim that 
$ \log\tr_{\o_{\vp}}\o\in \calC^{0,\tilde\gamma}_s,$
for any $\tilde\gamma \le\min\{\frac{1}{\beta} - 1,\gamma\}$.
Indeed, in the local coordinates $z_1,\ldots,z_n$, 
$g_{\vp}^{i\bar j}=\frac1{\det g_{\vp}} A_{i\bar j}$,
where $A$ is the cofactor matrix of $[g_{\vp}]$.  
Since $A_{i\jbar}$ is a polynomial in the components
$g_{\vp\,k\b \ell}$, it too lies in $\calC^{0,\gamma}_s$.
In addition, $1/\det g_{\vp}=e^{-f_\o-c_t+s\vp}/\det g_\o
=|z_1|^{2-2\be}F$ for some $F\in \calC^{0,\gamma}_s$, hence this lies 
in $\calC^{0,\tilde\gamma}_s$ for
$\tilde\gamma\le\frac1\be-1$. Hence 
$\tr_{\o_{\vp}}\o=g^{i\bar j}_{\vp}g_{i\bar j}\in \calC^{0,\tilde\gamma}_s$, 
proving the claim. 

Now by Lemma \ref{BarrierFnLemma} applied to $f:=\log\tr_{\o_{\vp}}\o-A\vp$, we have that
$f+|s|^\eps_h$ achieves its maximum away from $D$ for $\eps < \beta\gamma $
(when $s=e$ we use that $\vp\in\calA_{\phg}$ by Theorem \ref{PhgMainThm}).
By (\ref{ChernLuIneqPotentialEq}) and Lemma \ref{BarrierFnLemma} (ii) 
(and in particular \eqref{LaplacianOfBarrierFnEq}) we have for all sufficiently large $N>1$ 
$$
\Delta_{\vp} (f+ N^{-1}|s|^\eps_h) \ge -C_1-(2C_3+1)n + (1-C/N)\tr_{\o_{\vp}}\o.
$$
The maximum principle thus implies  $\tr_{\o_{\vp}}\o\le C=C(C_1,C_3,||\vp||_{\calC^0(M)},\o)$,
and so $\o_{\vp}\ge C\o$. Going back to (\ref{RCMEq}) 
we have 
$\ovpn\le C\on$
(with $C$ depending on $||f_\o||_{\calC^0(M)}$ and $||\vp(s,t)||_{\calC^0(M)}$),
and so also $\o_{\vp}\le C\o$. 
\end{proof}

\section{\Holder estimates for second derivatives}
\label{HolderSection}
In the interior of $\MsmD$ the Evans--Krylov regularity theory for \MA equations (Theorem~\ref{EvansKrylovThm}) may be
applied to obtain the a priori interior $\calC^{2,\gamma}$ estimate for a solution $\vp$ on any ball $B'$ depending on 
$\calC^0$ estimates for $\vp$ and $\Delta_\o \vp$ on a slightly larger ball. This depends heavily, of course, on the 
uniform ellipticity of the Laplacian, and hence does not apply directly for balls arbitrarily close to $D$.

We now explain how to obtain a priori estimates in $\Dw$ using the a priori Laplacian estimate from the last section. 
The proof uses an old argument due to Tian. 

\begin{thm} 
\label{DwCor}
Let $\vp(s)\in \calD_w^{0,0}\cap \calC^4(M\sm D)\cap \PSH(M,\o)$ be a solution to \eqref{RCMEq} with $s>S$ 
and $0 < \beta \leq 1$. Then 
$$
||\vp(s)||_{\Dw}\le C,
$$
where $C=C(S,\beta,\o,n,||\Delta_\o\vp||_{L^\infty(M)},||\vp||_{L^\infty(M)}).
$
\end{thm}
\begin{proof}
Let $\calU$ be a neighborhood in $M$. According to Theorem \ref{TianThmB}, if $\ovp$ is locally
represented by $u_{i\b j}dz^i\w \overline{dz^j}$ on $\calU\sm D$, then for some fixed $\gamma,r_0>0$,
every $a\in(0,r_0)$, and all $x$ such that $B_a(x)\subset U$, we have 
$||\nabla u_{i\bar j}||^2_{B_a(x)}\le\, C\, a^{2n-2+2\gamma}.$ The constants $\gamma,r_0,C$ are all uniformly controlled. 
The Poincar\'e inequality gives $||u_{i\bar j}-C_{x,a}||^2_{B_a(x)}\le\, C\, a^{2n+2\gamma},$
where $C_{x,a}=\int_{B_a(x)}u_{i\bar j}\on/\int_{B_a(x)}\on$. Using the integral characterization of H\"older spaces,
see \cite[Theorem 3.1]{HanLin} for example, patching up the estimates over a finite cover,
and using that we already have uniform bounds on  $\Delta_\o\vp$ itself, it follows that $||\Delta_\o\vp||_{C^{0,\gamma}_w}\le C$.
\end{proof}

The proof of Theorem \ref{TianThmB} requires the following lemma that is perhaps of independent interest. 
Let $\psi$ be a fixed \K potential for $\o$ valid in a neighborhood of $y_0\in D$. 
For each pair of parameters $(s,t)$, consider the function $h=h(s,t)$ defined by the equality
\begin{equation}
\label{DefinehEq}
\log h:= \log F+\log\det[\psi_{i\bar j}]=tf_\o+c_t-s\vp+\log\det[\psi_{i\bar j}].
\end{equation}
By Lemma \ref{ReferenceMetricExistsLemma}
(see \eqref{ddbarPhioneSecondEq}), $\psi\in \calD_w^{0,\gamma}$, for any $\gamma\in(0,\frac1\be-1]$.

We now state a collection of estimates for $h$, but note that only the Lipschitz bound (iv) is used later in the proof 
of Theorem \ref{TianThmB}.
\begin{lem}
\label{hLipschitzLemma}
Define $h=h(s,t)$ by \eqref{DefinehEq}, with $s>S$. Then 
the following estimates hold with constants independent of $t,s$:\hfill\break
(i) For $\be\le 1$, $||h(s,t)||_{C^0(M)}\le C(S,M,\o,\be,||\vp(s,t)||_{C^0(M)})$.
\hfill\break
(ii) For $\be\leq 1/2$, $||h(s,t)||_{\calD_w^{0,0}}\le 
C=C(S,M,\o,\be,||\Delta_\o\vp(s,t)||_{\calC^0})$. 
\hfill\break
(iii) For $\be \leq 2/3$, $||h(s,t)||_{w; 0,1} \le C=C(S,M,\o,\be,||\vp(s,t)||_{\calC_w^{0,1}})$.
\hfill\break
(iv) For $\be \le 1$,
\begin{equation}
\label{LipschitzHEstimate}
||h(s,1)||_{w; 1,\frac1\beta-1} \le C=C(S,M,\o,\be,||\vp(s,1)||_{\calC_w^{1,\frac1\beta-1}}).
\end{equation}
Moreover, $||h(s,1)_{k\b l}||_{C^0(M)}\le C(S,M,\o,\be,||\Delta_\o\vp(s,1)||_{\calC^0})$.
\end{lem}

\begin{proof}
By \eqref{DetOmeganEq} and \eqref{zzeta}, near $D$,
$$
\begin{aligned}
\det[\psi_{i\jbar}]
& =
\be^2|\zeta|^{\frac2\be-2}
\det\Big[\frac{\del^2(\psi_0+\phi_0)}{\del z^i\del\overline{ z^j}}\Big] 
\cr
& = \be^2\sum_{k=0}^n  f_{0k} |\zeta|^{2k-2+\frac2\be}  
  + \be^2\sum_{k=0}^{n-1} ( f_{1k} +  f_{2k} \zeta^{\frac1\be} 
  +  f_{3k} \overline{\zeta^{\frac1\be}}) |\zeta|^{2 k},
\end{aligned}
$$
with $f_{jk}$ smooth functions of $(z_1,\ldots,z_n)$.
Thus, if $\be\in(0,2/3]$,  $\log\det[\psi_{i\bar j}]$ is in $\calC^{0,1}_w$. Moreover, if $\be\in(0,1/2]$ then
$\log\det[\psi_{i\bar j}]$ is in $\calC^{1,1}_w$: for that it suffices to remark that 
$$
\del_r f_{jk} 
=
\frac{\del f_{jk}}{\del \rho}\frac{\del \rho}{\del r}
= 
\frac{\del f_{jk}}{\del \rho}(\be r)^{\frac1\be-1}\in C^{0,1}_w,
$$
if $\be\in(0,1/2]$.
Next, by Lemma \ref{foAsympExpansionLemma}, and the same reasoning
as above, it follows that when $\be\in(0,1/2]$, $f_\o\in C^{1,1}_w$
and that when $\be\in(0,2/3]$, $f_\o\in C^{0,1}_w$. Therefore, 
(ii) and (iii) follow. Note also that the above computations
show that $f_\o,\det[\psi_{i\bar j}]\in L^\infty(M)$ for all $\be\in(0,1]$,
proving (i).

Now, assume $t=1$. Denote, as before, $d{\bf z}=dz_1\w\cdots\w dz_n$.
By \eqref{fomegaSecondEq},
$$
\begin{aligned}
f_\o+\log\det[\psi_{i\bar j}]
& =
f_\o+
\log 
\frac
{(\o_0+\i\ddbar\psi_0)^n}
{(\i)^{n^2}d\zeta\w d\bar\zeta\w dz_2\w d\bar{z_2} \w\ldots\w dz_n\w d\bar{z_n}}
\cr
& =
f_\o+
\log 
\frac
{(\o_0+\i\ddbar\psi_0)^n}
{|z_1|^{2\be-2}(\i)^{n^2} d{\bf z}\w d{\bf \b z}}
\cr
& =
\log 
\frac
{|s|^{2\be-2}_h\o_0^n e^{F_{\o_0}-\mu \phi_0}}
{|z_1|^{2\be-2}(\i)^{n^2} d{\bf z}\w d{\bf \b z}}
\cr
& =
(2\be-2)\log a+F_{\omega_0}-\mu\phi_0+\log\det\Big[\frac{\del^2\psi_0}{\del z^i\del\bar{z^j}}\Big],
\end{aligned}
$$
where $a$ is defined in \eqref{afunctionEq}.
Thus $f_\o+\log\det[\psi_{i\bar j}]$ can be written as a sum $\Phi_1-\mu\phi_0=\Phi_1-\mu r^2\Phi_0$
with $\Phi_i,\, i=0,1$, smooth functions of
$(z_1,\overline{z_1},\ldots,z_n,\overline{z_n})$. Hence, by the reasoning above $f_\o+\log\det[\psi_{i\bar j}]\in C^{1,\frac1\be-1}_w$,
and therefore $h(s,1)$ belongs to $C^{1,\frac1\be-1}_w$ as soon as $\vp(s,1)$ does. The statement about the $(1,1)$-part of
the Hessian of $h(s,1)$ follows in the same way.
This concludes the proof of (iv).
\end{proof}

\section{Existence of \KE edge metrics}
\label{ProofKESection}
We now conclude the proof of Theorem \ref{ConicKEMainThm} on the existence of \KE edge metrics,
as well as of the convergence of the twisted Ricci iteration (Theorem \ref{RIThm}). 
We then describe the additional regularity properties as stated in Theorem \ref{ConicKEMainThm}.

\medskip
\noindent {\bf Starting the continuity path. }
Intuitively, the Ricci continuity path \eqref{RCMEq} has the trivial solution $\o(-\infty)=\o$ at $s=-\infty$.
Even if one could make rigorous sense of this, one could not apply the implicit function theorem 
directly to obtain solutions for large negative finite values of $s$. Indeed, reparametrizing \eqref{RCMEq} by 
setting $\sigma = -1/s$, then the linearization of the Monge-Amp\`ere equation at $\sigma$ equals
$\sigma \Delta_{\vp(-1/\sigma)} - 1$, and this degenerates at $\sigma = 0$. It is therefore necessary to find
a different way to produce a solution of \eqref{RCMEq} for sufficiently negative, but finite, values of $s$. 
Once this has been accomplished, we can then proceed with the rest of the continuity method.

When $\be\in(0,1/2]$, this difficulty can be circumvented by using the two-parameter family, see 
Remark \ref{OneTwoParamRemark}. Indeed, as described in \S\ref{CMSubSec}, the original continuity 
path \eqref{RCMEq} embeds into the two-parameter family \eqref{TwoParamCMEq}, and 
it is trivial that solutions exist for the finite parameter values $(s,0)$.  Unfortunately, the a priori estimates
needed to carry out the rest of the continuity argument for the two-parameter family hold only when $\be \leq 1/2$.  
Thus, to handle the general case, we must use another method to obtain a solution
of \eqref{RCMEq} for some large negative value of $s$.  Wu \cite[Proposition 7.3]{Wu} used a 
Newton iteration argument to obtain such a solution in a different setting.
However, his argument requires a lower Ricci curvature bound on the reference metric1
(see \cite[p. 431]{Wu}), which we lack. In other words, no small multiple of $f_\o$ belongs to 
$\calH_\o$. What follows is an adaptation of Wu's argument that requires
no curvature control on the reference metric. 

Reformulate the original complex \MA equation in terms of the operator
\[
N_{\sigma}:\Dw \ra \calC^{0,\gamma}_w, \ \  \  N_\s(\Phi):=\log (\o_{\sigma\Phi}^n/ e^{f_\o}\o^n)-\Phi.
\] 
As we remark at the end of this argument, the following argument works equally well in the
edge spaces, and leads to the same conclusion. 
Observe that $DN_\s|_\Phi=\sigma\Delta_{\sigma\Phi}-\h{Id}$. Now, suppose that $\s\Phi\in\calH_\o \cap \calA_{\phg}$. By  
Proposition~\ref{char1}, $DN_\s|_\Phi:\Dw\ra \calC^{0,\gamma}_w$ is Fredholm of index $0$, 
and by the maximum principle and Lemma~\ref{BarrierFnLemma}, its nullspace $K$ is trivial when $s < 0$. 
Hence, this operator is an isomorphism from $\Dw$ to $\calC^{0,\gamma}_w$, with
\begin{equation}
\label{est1}
||u||_{\Dw}\le C||DN_{\s}u||_{\calC^{0,\gamma}_w},
\end{equation}
Denote by $DN_\s|_{\Phi}^{-1}$ the inverse of this map on $\calC^{0,\gamma}_w$. 

We now set up the iteration method that will converge to a solution (Newton iteration for $N_\sigma$). Define a sequence 
of elements $\Phi_k \in \De$ by setting $\Phi_0=0$ and then 
\[
\Phi_k=(\h{Id}-DN_\s|_{\Phi_{k-1}}^{-1}\circ N_\s)(\Phi_{k-1}), \quad k\in\NN,
\] 
or equivalently, 
\begin{equation}
\Phi_k - \Phi_{k-1} - DN_\s|_{\Phi_{k-1}}^{-1}N_\s(\Phi_{k-1})  = 0.
\label{fff}
\end{equation}
When $\Phi \in \calA^0_{\phg}$, $DN_\s|_\Phi^{-1}$ preserves polyhomogeneity, so each of the successive 
$\Phi_k$ are polyhomogeneous.  Since $N_\sigma(-f_\o) = 0$ when $\sigma = 0$, it might seem more 
natural to set $\Phi_0 = -f_\o$. However, this would cause a problem at the very next step since, as already 
observed two paragraphs above, $\i\ddbar f_\o$ blows up at $r=0$ when $\be > 1/2$.   

Next, observe that
\begin{equation}
 \label{NewtonHessianEq}
\begin{aligned}
N_\s(\Phi_k) = N_\s(\Phi_k)& -N_\s(\Phi_{k-1})-DN_\s|_{\Phi_{k-1}}(\Phi_{k}-\Phi_{k-1}) \\ 
& =\int_0^1(1-c)D^2N_\s|_{c\Phi_k+(1-c)\Phi_{k-1}}(\Phi_k-\Phi_{k-1},\Phi_k-\Phi_{k-1})\, dc.
\end{aligned}
\end{equation}
This will be estimated using the equality $D^2N_\s|_{\Phi}(a,b)=\s^2(\ddbar a,\ddbar b)_{\o_{\s\Phi}}$, 
which holds provided $\s\Phi\in\calH_\o$.
We now deduce inductively the sequences of estimates
\vglue-9pt
\begin{equation}
||\Phi_j - \Phi_{j-1}||_{\Dw} \leq C_1 ||N_\s(\Phi_{j-1})||_{\calC^{0,\gamma}_w}
\label{firstest}
\end{equation}
and
\vglue-9pt
\begin{equation}\label{NsigmasqEq}
||N_\s(\Phi_{j})||_{\calC^{0, \gamma}_w} \leq C_2 \sigma^2 ||\Phi_{j} - \Phi_{j-1}||_{\Dw}^2
\end{equation}
for every $j \geq 1$ with constants $C_1$ and $C_2$ independent of $j$.   Suppose then that these hold
for every $j \leq k$. We shall prove that they hold also for $j = k+1$ with the same constants $C_i$. 

First note that 
$
\Phi_k = \Phi_k - \Phi_0 = \sum_{j=1}^k (\Phi_j - \Phi_{j-1}).
$
Using \eqref{firstest} and \eqref{NsigmasqEq} iteratively gives
\vglue-6pt
\[
||\Phi_{j} - \Phi_{j-1}||_{\Dw} \leq C_1 C_2 \s^2 ||\Phi_{j-1}-\Phi_{j-2}||^2_{\Dw} 
\leq (C_1 C_2 \s^2)^{2^{j-1}-1} ||\Phi_1||^{2^{j-1}}_{\Dw},
\]
and then
$
||N_\s(\Phi_j)||_{\calC^{0, \gamma}_w} \leq C_2 \s^2 (C_1 C_2 \s^2)^{2^{j-1}-1} ||\Phi_1||^{2^{j-1}}_{\Dw}. 
$
We conclude that
\vglue-9pt
\[
||\Phi_k||_{\Dw} \leq \sum_{j=1}^k ||\Phi_j - \Phi_{j-1}||_{\Dw} \leq 2 ||\Phi_1||_{\Dw},
\]
provided $C_1 C_2 \s^2 ||\Phi_1||_{\Dw} \leq 1/2$.  Hence if $\s$ is sufficiently small, then
$||\s \Phi_k||_{\Dw} \leq \eta$ for some fixed $\eta > 0$. 
So, $\s\Phi_k \in \calH_\o$, 
and if we let $C_1$ denote the supremum of the norm of $DN_\s|_\Phi^{-1}$ among all $\Phi$ 
with $||\Phi||_{\Dw} \leq \eta$, then \eqref{firstest} holds with $k$ replaced by $k+1$ and the same $C_1$. 

To obtain the final estimate, note that $\sigma(c\Phi_k+(1-c)\Phi_{k-1})$ lies in the same ball of 
radius $\eta$ for $0 \leq c \leq 1$, hence also in $\calH_\o$, which means that \eqref{NewtonHessianEq} 
with $k$ replaced by $k+1$ can be estimated as before; $C_2$ is the constant needed to estimate
this integral, which is uniform so long as $c\Phi_k+(1-c)\Phi_{k-1}$ remain in the fixed $\eta$ ball. 
This proves \eqref{NsigmasqEq} with $k$ replaced by $k+1$. 

We may now conclude that
$
\Phi_\infty := \lim_{k\to \infty} \Phi_k = \sum_{k=0}^\infty (\Phi_{k+1}-\Phi_k)
$
exists and lies in the same $\eta$ ball in $\Dw$, so $\s \Phi_\infty \in \calH_\o$.

Finally, by Theorem~\ref{genma}, $\Phi_\infty$ is polyhomogeneous.

\medskip
\noindent {\bf Openness.} Define $M_{s,t}:\Dw\ra C^{0,\gamma}_w$ by 
$$
M_{s,t}(\vp):= \log \frac{\o_{\vp}^n}{\on}-tf_\o+s\vp,
\quad (s,t)\in A=(-\infty,0]\times[0,1]\;\cup [0,\mu]\times\{1\}.
$$
Note that $M_{s,0}(0)=0$. If $\vp(s,t)\in\Dw\cap \PSH(M,\o)$ is a solution of \eqref{TwoParamCMEq}, 
we claim that its linearization 
\begin{equation}
\left. DM_{s,t}\right|_{\vp(s,t)} =\Delta_{\vp(s,t)}+s :\Dw\ra \calC^{0,\gamma}_w,
\quad (s,t)\in A,
\label{oxf}
\end{equation}
is an isomorphism when $s \neq 0$. If $s = 0$, this map is an isomorphism if we restrict on each side 
to the codimension one subspace of functions with integral equal to $0$.  Furthermore, we also 
claim that $\Dw \times A \ni (\vp,s,t) \mapsto M_{s,t}(\vp) \in 
\calC^{0,\gamma}_w$ is a $\calC^1$ mapping. Given these claims, the Implicit Function Theorem then 
guarantees the existence of a solution $\vp(\tilde s,\tilde t) \in \Dw$ 
for all $(\tilde s,\tilde t)\in A$ sufficiently close to $(s,t)$. 

Proposition \ref{char1} asserts that \eqref{oxf} is Fredholm of index $0$ for any $(s,t) \in A$, and by
Proposition \ref{mainlinearprop}, 
\begin{equation}
\label{DsLinearizationNormEq}
||u||_{\Dw}\le C(||DM_{s,t}u||_{\calC^{0,\gamma}_w} + ||u||_{\calC^0}),
\end{equation}
Its nullspace $K$ is clearly trivial when $s < 0$, and also by Lemma~\ref{LambdaoneLemma} for $(s,1)$ with $s \in (0,\mu)$; 
finally, when $s=0$ it consists of constants. Thus $DM_{s,t}$ is an isomorphism when $s \neq 0$, 
and is an isomorphism on the $L^2$ orthogonal complement to the constants when $s=0$. 
This proves the first claim.

The second claim follows from \eqref{oxf} and Corollary~\ref{equivdomains}, which shows that the domains of these 
linearizations at different $\vp$ are all the same. The smooth dependence on $(s,t)$ is obvious. 

Note finally that using \eqref{DsLinearizationNormEq},  nearby solutions remain in $\PSH(M,\o)$.   

We have written this out explicitly for the wedge spaces, but note that all of these arguments go through verbatim
for the edge spaces. Observe, however, that using the results of Section 4, the nearby solutions are necessarily
polyhomogeneous.

\medskip

\noindent {\bf Closedness.} 
Fix some $S<0$ and denote $A_S:=\{(s,t)\in A\,:\, s\in (S,0]\}$. Let $\{ (s_j,t_j)\}$ be a sequence in $ \h{\rm int} \, A_S$ 
converging to $(s,t)\in \overline{A_S}$, and let $\vp(s_j,t_j)\in\Dw\cap \PSH(M,\o)$ be solutions to \eqref{TwoParamCMEq}. 
Under the assumptions of Theorem \ref{ConicKEMainThm}, the results of \S\ref{CzeroNonpositiveSection}, 
\ref{LaplacianSection}, and \ref{HolderSection} imply that $||\vp(s_j,t_j)||_{\Dw}\le C$, where $C$ depends on $S$, 
a  lower bound on the Ricci curvature of $\o$ times $(1-\min_{j} t_j)$, and an upper bound on 
its bisectional curvature, both over $\MsmD$; alternatively, the Aubin--Yau Laplacian estimate \cite{Au,Y2,Siu} 
gives a bound depending on $S$ and a lower bound on the bisectional curvature of $\o$ over $\MsmD$.
Thus, when $\be\in(0,\frac12]\cup\{1\}$, Lemma \ref{CurvRefMetricLemma} implies that either type of bounds give 
a uniform estimate $||\vp(s,t)||_{\Dw}\le C$, for all $(s,t)\in A_S$. In general, restrict to the path \eqref{RCMEq}  
(i.e., let $t_j=1$ for all $j$) and then $||\vp(s,1)||_{\Dw}\le C$, for all $s\in(S,0]$ by Proposition~\ref{AppendixProp} 
and the results of \S\ref{CzeroNonpositiveSection}, \ref{LaplacianSection}, and \ref{HolderSection}. 
Thus, for any $\gamma' \in (0,\gamma)$,  there is a subsequence which converges in $\calD_w^{0,\gamma'}$
such that the limit function $\vp(s,1)$ lies in $\calD_w^{0,\gamma}$.  
Observe that $M_{s,1}(\vp(s,1)) = 0$, and $\vp(s,1) \in \calC^\infty(M\sm D)$. 

Letting $S\ra-\infty$, we obtain a solution for all $(s,t)\in A_\infty$
in the case $\be\in(0,\frac12]\cup\{1\}$, and for all $(s,1)\in(-\infty,0]\times\{1\}$
in the general case.
Now by openness in $A$ about the solution at $(0,1)$ (cf.\ \cite{Aubin1984,BM}), there exist
solutions also for $[0,\eps)\times\{1\}\subset A$. Then by Corollary 
\ref{PropernessCzeroCor} and the previous arguments we obtain solutions
for all $(s,t)\in A$ when $\be\in(0,\frac12]\cup\{1\}$, and for all $(s,1)\in(-\infty,\mu]\times\{1\}$ in the general case. 
By Theorem \ref{PhgMainThm}, these solutions are polyhomogeneous. 
Finally, $\vp(s,t) \in \PSH(M,\o)$ (this follows from a continuity
argument, observing that the right hand side of the \MA equation is positive). 

\medskip 
\noindent {\bf Regularity.}
Using the steps above, we obtain a solution $\vp:=\vp(\mu,1)\in \calA^0_{\phg} \cap \PSH(M,\o)$
to \eqref{RCMEq}. Denote by $g_\mu$ the associated \KE edge metric. Using
Proposition~\ref{PreciseAsympExpansionProp} and \eqref{pertgb}, $g_\mu$ is asymptotically 
equivalent to the reference metric $g$, and moreover, by the explicit form of the expansion and
the fact that $P_{1\b1}$ annihilates the $r^0$ and $r^{\frac{1}{\be}}$ terms, we obtain that
$\vp \in \calA^0 \cap \calD_w^{0,\eps(\be)}$, where $\eps(\be)$ is determined by Proposition \ref{PreciseAsympExpansionProp}
(see Corollary \ref{DsCor}). This completes the proof of Theorem \ref{ConicKEMainThm}.

\medskip
\noindent
{\bf Convergence of the Ricci iteration.}
We use the notation of \S\ref{RISubSec}. As noted there, $\mu-\frac1\tau$ plays the role of $s$.
Consider first the case $\mu\le0$. By the earlier analysis of \eqref{RCMEq}, for any $\tau>0$ the 
iteration 
exists uniquely and $\{\psi_{k\tau}\}_{k\in\NN}\subset \Dw$. By Lemma~\ref{BarrierFnLemma} the inductive 
maximum principle argument of \cite{R} yields $|\psi_{k\tau}|\le C$. Along the iteration, just as for the path 
\eqref{RCMEq}, the Ricci curvature is bounded from below by $\mu-\frac1\tau$, hence 
Proposition~\ref{ChernLuApplicationProp} and Lemma~\ref{BarrierFnLemma} show that 
$|\Delta_{\o_{k\tau}}\psi_{k\tau}|\le C$ (we consider the maps $\id:(M,\o_{k\tau})\ra (M,\o)$). 
Going back to the equation \eqref{RIEq} and using the 
$\calC^0$ estimate then shows that $|\Delta_\o{\psi_{k\tau}}|\le C$, hence by Theorem~\ref{DwCor},
$|\psi_{k\tau}|_{\Dw}\le C$.  Thus a subsequence converges  (as explained above for the continuity method)
to an element $\psi_\infty$ of $\Dw\cap \calC^\infty(\MsmD)$. 
Since each step in the iteration follows a continuity path of the form \eqref{RCMEq} 
with $\o$ replaced  by $\o_{k\tau}$, Lemma~\ref{KEnergyMonotonicityLemma} implies that $E_0^\be(\o_{(k-1)\tau},\o_{k\tau})<0$
(unless $\o$ was already \KEno). Since $E_0^\be$ is an exact energy functional, i.e., satisfies a cocyle condition 
\cite{Mabuchi1986}, then 
$
E^\be_0(\o,\o_{k\tau})
=
\sum_{j=1}^{k}E_0^\be(\o_{(j-1)\tau},\o_{j\tau})<0.
$
Therefore, $\psi_\infty$ is a fixed point of $E^\be_0$, hence a \KE edge metric. By Lemma~\ref{AYUniquenessLemma} 
such \KE metrics are unique; we conclude that the original iteration converges to $\psi_\infty$ both in 
$\calA_0$ and in $\calD^{0,\gamma'}_w$ for each $\gamma'\in(0,\gamma)$. 

Next, consider the case $\mu>0$, and take $\mu=1$ for simplicity. 
By the properness assumption, 
Corollary~\ref{PropernessCzeroCor} implies the iteration exists 
(uniquely by Lemma \ref{PositiveCaseUniquenessCzeroLemma})
for each $\tau\in(0,\infty)$ and then the monotonicity of $E_0^\be$ implies 
that $J(\o,\o_{k\tau})\le C$. 
To obtain a uniform estimate on $\hbox{\rm osc}\,\psi_{k\tau}$ 
we will employ the argument of \cite{BBEGZ} as explained to us by Berman.
By Lemma~\ref{BermanLemma}, have $\int_M e^{-p(\psi_{k\tau}-\sup\psi_{k\tau})}\on\le C$, 
where $p/3=\max\{1-\frac1\tau,\frac1\tau\}$. Now rewrite \eqref{RIEq} as 
\begin{equation}
\label{RISecondEq}
\o_{\psi_{k\tau}}^n=\on 
e^{f_\o-(1-\frac1\tau)\psi_{k\tau}-\frac1\tau\psi_{(k-1)\tau}}.
\end{equation}
Using \Kolodziej's estimate and the \Holder inequality this yields 
the uniform estimate $\h{\rm osc}\, \psi_{k\tau}\le C$.
Unlike for solutions of \eqref{RCMEq}, the functions $\psi_{k\tau}$ need not be changing signs.
Therefore we let $\tilde\psi_{k\tau}:=\psi_{k\tau}-\frac1V\int_M\psi_{k\tau}\on$. As in the previous paragraph 
we obtain a uniform estimate $\tr_{\o_{k\tau}}\o\le C$. However, to conclude that 
$\tr_\o\o_{k\tau}\le C$ from \eqref{RISecondEq}
we must show that $|(1-\frac1\tau)\psi_{k\tau}-\frac1\tau\psi_{(k-1)\tau}|\le C$. 
This is shown in \cite[p. 1543]{R}. Thus, as before, we conclude that 
$\{{\tilde\psi_{k\tau}}\}$ subconverges to the potential of a \KE
edge metric. Whenever it is unique, the iteration itself necessarily converges. 
Berndtsson's generalized Bando--Mabuchi Theorem \cite{BM, Bern} shows uniqueness 
of \KE edge metrics up to an automorphism (which must preserve 
$D$ by \eqref{EdgeKEEq} or Lemma \ref{LambdaoneLemma}), 
This concludes the proof of Theorem \ref{RIThm}.

\appendix
\section{Upper bound on the bisectional curvature of the reference metric}

{\bf Chi Li and Yanir A. Rubinstein}

\bigskip

\begin{prop}
\label{AppendixProp}
Let $\be\in(0,1]$, and let $\o=\o_0+\i\ddbar|s|^{2\be}_h$ be given by \eqref{oDefEq}. 
The bisectional curvature of $\o$ is bounded from above on $\MsmD$.
\end{prop}

\def\ij{{i\jbar}}
\def\kl{{k\bar\ell}}
\def\bell{{\bar\ell}}
\def\pq{{p\b q}}
\def\st{{s\b t}}

We denote throughout by $\hat g,g$ the \K metrics associated to $\o_0,\o$,
respectively.
As in \cite{TY1990}, to simplify the calculation and estimates
we need a lemma to choose an appropriate local 
holomorphic frame and coordinate system, whose
elementary proof we include for the reader's convenience.
We thank Gang Tian for pointing out to us
the calculations in \cite{TY1990} which were helpful
in writing this Appendix.

\begin{lem}\label{choose1}
{\rm \cite[p. 599]{TY1990} }
There exists $\epsilon_0>0$ such that
if $0<\h{\rm dist}_{\hat{g}}(p, D)\le \epsilon_0$,   
then we can choose a local holomorphic frame $e$ of $L_D$ 
and local holomorphic coordinates $\{z_i\}_{i=1}^n$ valid in a neighborhood
of $p$, such that (i) $s=z_1e$, 
and $a:=|e|_h^2$ satisfies $a( p) =1$, $d a( p)=0$, 
$\frac{\partial^2 a}{\partial z_i\partial z_j}a( p)=0$,
and (ii) $\hat g_{i\jbar,k}(p)
=\frac\del{\del z^k}\o_0(\frac\del{\del z^i},\frac\del{\overline{\del z^{j}}})|_p=0$, whenever $j\ne 1$.
\end{lem}

\begin{proof}
(i) Fix any point $q\in D$, and choose a local holomorphic 
frame $e'$ and holomorphic coordinates $\{w_i\}_{i=1}^n$ in $B_{\hat{g}}(q, \epsilon(q))$ 
for $0<\epsilon(q)\ll 1$.  Let $s=f'e'$ with $f'$ a holomorphic function 
and $|e'|_h^2=c$. Let $e=Fe'$ for some nonvanishing holomorphic function $F$ to be specified later. 
Then $a=|Fe'|_h^2=|F|^2 c$. Now fix any point $p\in B_{\hat{g}}(q, \epsilon(q))\sm\{q\}$. 
In order for $a$ to satisfy the vanishing properties with respect to 
the variables $\{w_i\}_{i=1}^n$ at a point $p$, we can just choose $F$ such 
that $F( p)=c( p)^{-1/2}$, 
and 
\begin{eqnarray*}
{\partial}_{w^i} F( p)
&=&
-c^{-1}F{\partial}_{w^i} c( p)
=
-c^{-3/2}{\partial}_{w^i} c( p)
\cr
\partial_{w^i}\partial_{w^j} F( p)
&=&
-c ^{-1}(
F\partial_{w^i}\partial_{w^j} c 
+\partial_{w^j} c\partial_{w^i} F
+\partial_{w^i} c\partial_{w^j} F)(p)\cr
&=&
-c^{-3/2}\partial_{w^i}\partial_{w^j} c(p)+2c^{-5/2}\partial_{w^i} c\partial_{w^j} c(p).
\end{eqnarray*}
Since $c=|e'|^2_h$ is never zero, when $\epsilon(q)$ is small, which implies $|w-w(p)|$ is small, we can assume $F\ne0$ in $B_{\hat{g}}(q, \epsilon(q))$.
Now  $s=f e=f'e'$ with $f=f'F^{-1}$ a holomorphic function. 
Since $D=\{s=0\}$ is a smooth divisor, we can 
assume $\partial_{w^1} f(q)\neq 0$,
and choosing $\epsilon(q)$ sufficiently small, we can 
assume that $\partial_{w^1} f\neq 0$ in $B_{\hat{g}}(q, \epsilon(q))$. 
Thus by the inverse function theorem, $z_1=f(w_1, \ldots, w_n), z_2=w_2, \ldots, z_{n}=w_{n}$ 
are holomorphic coordinates in $B_{\hat{g}}(q, \epsilon(q)/2)$ and now $s=f(w) e=z_1 e$.  
By the chain rule, it then follows that $a$ 
satisfies $a( p)=1$, $\partial_{z^i}a( p)=\partial_{z^i}\partial_{z^j}a( p)=0$. 

Now cover $D$ by $\cup_{q\in D} B_{\hat{g}}(q, \epsilon(q)/2)$. By compactness of $D$  
the conclusion follows.

\medskip

\noindent
(ii) Denote by $\{w^i\}_{i=1}^n$ the coordinates obtained in (i).
Following
\cite[p. 108]{GH}, let 
$
\tilde z^k:= w^k-w^k(p)+\frac12 b^k_{st} (w^s-w^s(p))(w^t-w^t(p)),
$
with $b^k_{st}=b^k_{ts}$, define a new coordinate system.
Then, $
\o_0(\frac{\del}{\del w^i},\frac{\del}{\overline{\del w^j}})
=
\o_0(\frac{\del}{\del  \tilde z^i},\frac{\del}{\overline{\del \tilde z^j}})
+
\hat g_{t\jbar}b^t_{ip} w^p+\hat g_{i\bar t}\overline{b^t_{jp} w^p}+O(\sum_{i=1}^n|w^i-w^i(p)|^2),
$
and 
$$
d_{\ij k}:=
\frac{\del}{\del w^k}
\o_0(\frac{\del}{\del  w^i},\frac{\del}{\overline{\del  w^j}})|_p
=
\frac{\del}{\del \tilde z^k}
\o_0(\frac{\del}{\del \tilde z^i},\frac{\del}{\overline{\del  \tilde z^j}})|_p
+\hat g_{t\jbar}(p) b^t_{ik}
=: e_{\ij k}+\hat g_{t\jbar}(p) b^t_{ik}.
$$
Let $\hat g'_{r\bar s}:=\hat g_{r\b s},$ for each $r,s>1$, and
denote the inverse of the $(n-1)\times(n-1)$ matrix $[\hat g'_{r\b s}]$
by $[\hat g'^{r\bar s}]$.
Let $b^1_{ik}=0$. Then, for each $j>1$, the equations can be rewritten as
$
d_{\ij k}-\sum_{t>1}\hat g'_{t\jbar}(p) b^t_{ik}
=
e_{\ij k}.
$
Hence, $\sum_{j>1}\hat g'^{s\jbar}e_{\ij k}=\sum_{j>1}\hat g'^{s\jbar}d_{\ij k}-b^s_{ik}, \,s>1$.
For each $s>1$, define $b^s_{ik}$ so that the right hand side vanishes. 
Multiplying the equations by $[\hat g'_{s\bar t}]$, we
obtain $e_{i\b t k}=0$ for each $t>1$.
Finally, set $z^i:=\tilde z^i+w^i(p),\; i=1,\ldots,n$.
Since $b^1_{ik}=0$, we have 
$z^1=w^1$, and therefore these coordinates satisfy both properties (i) and (ii) of the statement, as desired.
\end{proof}

Let $H:=a^{\be}$, then $|s|_h^{2\be}=|z_1e|_h^{2\be}=H|z_1|^{2\beta}$. 
Note that both $a$ and $H$ are locally defined smooth positive functions.
Let $\o=\frac\i2 g_{\ij}dz^i\w \overline{dz^j},\;
\o_0=\frac\i2 \hat g_{\ij}dz^i\w \overline{dz^j}$,
and write $z\equiv z_1$ and $\rho:=|z|$.
Using the symmetry for subindices, we can calculate in a straightforward manner:
\[
g_{i{\jbar}}=\hat{g}_{i{\jbar}}+H_{i{\jbar}}|z|^{2\beta}+\beta H_i\delta_{1{\jbar}}|z|^{2\beta-2}z+\beta H_{{\jbar}}\delta_{1i}|z|^{2\beta-2}\bar{z}+\beta^2H|z|^{2\beta-2}\delta_{1i}\delta_{1{\jbar}},
\]
\vglue-1cm
\begin{eqnarray*}
g_{i{\jbar}, k}&=&
\hat{g}_{i{\jbar}, k}
+H_{i{\jbar}k}|z|^{2\beta}
+\beta H_{ik}\delta_{1{\jbar}}|z|^{2\beta-2}z
+\beta(H_{k{\jbar}}\delta_{1i}
+H_{i{\jbar}}\delta_{1k})|z|^{2\beta-2}\bar{z}
\\
&&
+\beta^2 (H_i\delta_{1{\jbar}}\delta_{1k}+H_k\delta_{1i}\delta_{1{\jbar}}
+ H_{{\jbar}}\delta_{1i}\delta_{1k})|z|^{2\beta-2}\\
&&\beta^2(\beta-1)H|z|^{2\beta-4}\bar{z}\delta_{1i}\delta_{1{\jbar}}\delta_{1k},
\end{eqnarray*}
\vglue-1cm
\begin{eqnarray*}
g_{i{\jbar}, k\bar{\ell}}&=&\hat{g}_{i{\jbar}, k\bar{\ell}}+H_{i{\jbar}k\bar{\ell}}|z|^{2\beta}\\
&&+\beta\left[
(H_{ik\bar{\ell}}\delta_{1{\jbar}}+H_{ik{\jbar}}\delta_{1\bar{\ell}})|z|^{2\beta-2}z+(H_{{\jbar}\bar{\ell}i}\delta_{1k}+H_{{\jbar}\bar{\ell}k}\delta_{1i})|z|^{2\beta-2}\bar{z}
\right]
\\
&&
+\beta^2(H_{k{\jbar}}\delta_{1i}\delta_{1\bar{\ell}}+H_{i{\jbar}}\delta_{1k}\delta_{1\bar{\ell}}+H_{k\bar{\ell}}\delta_{1i}\delta_{1{\jbar}}+H_{i\bar{\ell}}\delta_{1{\jbar}}\delta_{1k})|z|^{2\beta-2}\\
&&+\beta(\beta-1)\left[
H_{ik}\delta_{1{\jbar}}\delta_{1\bar{\ell}}|z|^{2\beta-4}z^2+H_{{\jbar}\bar{\ell}}\delta_{1i}\delta_{1k}|z|^{2\beta-4}\bar{z}^2
\right]
\\
&&+\beta^2(\beta-1)\left[
(H_i\delta_{1k}+H_k\delta_{1i})\delta_{1{\jbar}}\delta_{1\bar{\ell}}|z|^{2\beta-4}z+(H_{{\jbar}}\delta_{1\bar{\ell}}+H_{\bar{\ell}}\delta_{1{\jbar}})\delta_{1i}\delta_{1k}|z|^{2\beta-4}\bar{z}
\right]\\
&&+\beta^2(\beta-1)^2H|z|^{2\beta-4}\delta_{1i}\delta_{1{\jbar}}\delta_{1k}\delta_{1\bar{\ell}}.
\end{eqnarray*}
Let $p\in\MsmD$ satisfy $\h{\rm dist}_{\hat{g}}(p, D)\le \epsilon_0$.
The lemma implies in particular $H(p)=1$, $H_i(p)=H_{ij}(p)=0$, and
the expressions above simplify to:
$$
\begin{aligned}
g_\ij( p)
= & \;
\hat g_{\ij}+
H_\ij|z|^{2\be}
+\be^2|z|^{2\be-2}\delta_{i1}\delta_{1\jbar},
\cr
g_{\ij,k}( p)
= & \;
\hat g_{\ij, k}+H_{\ij k}|z|^{2\be}
+\be (\delta_{i1}H_{k\jbar}+\delta_{k1}H_{i\jbar})|z|^{2\be-2}\bar{z}
+\be^2(\beta-1)\delta_{i1}\delta_{\jbar 1}\delta_{k1}|z|^{2\beta-4} \bar{z},
\cr
g_{\ij,\kl}( p)
= & \;
\hat g_{\ij, \kl}\!+\!H_{\ij \kl}|z|^{2\be}
\!+\!\be (\delta_{i1}H_{\jbar \kl}\!+\!\delta_{k1}H_{i\jbar \bell})|z|^{2\be-2}\bar{z}
\!+\!\be (\delta_{\jbar 1}H_{i \kl}\!+\!\delta_{\bell 1}H_{\ij k})|z|^{2\be-2} z
\cr & \quad\quad\;
+\be^2 (\delta_{i1}\delta_{\jbar 1} H_{\kl}+\delta_{i1}\delta_{\bell1}H_{k\jbar}+
\delta_{k1}\delta_{\jbar 1}H_{i\bell}+\delta_{k1}\delta_{\bell 1}H_{i\jbar}) |z|^{2\be-2}
\cr & \quad\quad\;
+\be^2(\be-1)^2\delta_{i1}\delta_{\jbar 1}\delta_{k1}\delta_{\bell 1}  |z|^{2\be-4}.
\end{aligned}
$$
It follows that
\begin{equation}
\label
{OffDiagVanEq}
 g^{r\b s}(p)=O(1), \quad 
g^{1\b s}(p)=O(\rho^{2-2\be}),\quad \hbox{ for\ } r,s>1,
\end{equation}
and,
\begin{equation}
\label
{gupperoneoneEq}
g^{1\b1}(p)=\be^{-2}\rho^{2-2\be}\frac{1}{1+b(p)\rho^{2-2\be}}+O(\rho^2),
\end{equation}
where $O(\rho^2)<C_3\rho^2$ and
$b(p):=\be^{-2}\det [\hat g_{\ij}]/\det  [\hat g_{r\bar s}]_{r,\bar s>1}|_p$
with $0<C_1<b(p)<C_2$, and $C_1,C_2,C_3$ independent of $p\in\MsmD$. 

Take two unit vectors $\eta=\eta^i\frac{\del}{\del z^i},
\nu=\nu^i\frac{\del}{\del z^i}\in T^{1,0}_pM$, so that
$g(\eta,\eta)|_p=g(\nu,\nu)|_p=1$. 
Then from the expression of $g_{i{\jbar}}$ we have
\begin{equation}\label{unitcomp}
\eta^1, \nu^1=O(\rho^{1-\be})
\quad  \eta^r,\nu^r=O(1), \mbox{ for } r>1. 
\end{equation}
Set 
$$
\h{Bisec}_\o(\eta,\nu)=
R(\eta,\bar\eta,\nu,\bar\nu)
=
R_{\ij\kl}\eta^i\overline{\eta^j}\nu^k\overline{\nu^\ell}
=
\sum_{i,j,k,l} \Lambda_{\ij\kl}+ \Pi_{\ij\kl},
$$
with 
$
\Lambda_{\ij\kl}
:=
-g_{\ij,\kl}\eta^i\overline{\eta^j}\nu^k\overline{\nu^\ell},
$
and
$
\Pi_{\ij\kl}
:=g^{s\b t}\, g_{i\b t,k} \, g_{s\jbar,\b \ell}\eta^i\overline{\eta^j}\nu^k\overline{\nu^\ell}
$
(no summations).
By \eqref{OffDiagVanEq}--\eqref{unitcomp} we have 
$|\Lambda_{\ij \kl}|\le C$ except for $\Lambda_{1\b 1 1\b 1}= 
-\be^2(\be-1)^2|z|^{2\be-4}|\eta^1|^2|\nu^1|^2$,
hence
\begin{equation}
\label{Cijoneone}
\sum_{i,j,k,l} \Lambda_{\ij\kl}(p)
=
O(1)+\Lambda_{1\b 1 1\b 1}(p)= O(1)-\be^2(\be-1)^2|z|^{2\be-4}|\eta^1|^2|\nu^1|^2.
\end{equation}

The Proposition follows
immediately by combining \eqref{Cijoneone} and
the following estimate.

\begin{lem}
\label{FirstAppendixLemma}
There exists a uniform constant $C>0$ such that for every $p\in\MsmD$, 
$$
\sum_{i,j,k,l} \Pi_{\ij\kl}(p)
\le
C+
\be^2(\be-1)^2|z|^{2\be-4}|\eta^1|^2|\nu^1|^2. 
$$
\end{lem}

\begin{proof}
Define a 
bilinear Hermitian form of two tensors $a=[a_{i\jbar k}]$, $b=[b_{p\b q r}]\in(\mathbb{C}^n)^3$
satisfying $a_{i\jbar k}=a_{k\jbar i}$ and $b_{p\b q r}=b_{r\b q p}$ by setting
$$
\langle [a_{i\jbar k}],[b_{p\b q r}]\rangle
:=\sum_{i,j,k,p,q,r}
g^{q\jbar}(\eta^i a_{i\jbar k}\nu^{k})  (\overline{\eta^{p} b_{p\b q r}\nu^{r}}).
$$
It is easy to see that this is a nonnegative bilinear form. We denote by $\|\,\cdot\,\|$
the associated norm.
Then 
$ 
\sum_{i,j,k,l} \Pi_{\ij\kl}
=
\|[g_{i\jbar, k}]\|^2
$.
Write,
$$
\begin{aligned}
g_{\ij,k}
= A_{\ij k}
+
B_{\ij k}
+
D_{\ij k}
+
E_{\ij k},
\end{aligned}
$$
with
$
A_{\ij k}:=\hat g_{\ij, k},\;
B_{\ij k}:=H_{\ij k}|z|^{2\be},\;
D_{\ij k}:=
\be (\delta_{i1}H_{k\jbar}+\delta_{k1}H_{i\jbar}) |z|^{2\be-2}\bar{z},\;
$
and
$
E_{\ij k}:=
\be^2(\be-1) \delta_{i1}\delta_{\jbar 1}\delta_{k1} |z|^{2\be-4}\bar{z}
$.
Denote $A=[A_{\ij k}]$ and similarly $B,D,E$.
Using \eqref{OffDiagVanEq},
$$
\langle D,E\rangle
\le C\Big|\sum_j g^{1 \jbar}|\eta^1|^2\overline{\nu^1}\rho^{2\be-1}\rho^{2\be-3}\Big|
\le C\rho^{1-\be},
$$
and similarly we conclude that
$\|[g_{i\jbar, k}]\|^2\le C+\|A+E\|^2$.
Now, since $\|\frac1{\sqrt\eps}A-\sqrt\eps E\|^2\ge0$,
we obtain 
$
\|A+E\|^2
\le
(1+\frac1\eps)\|A\|^2+(1+\eps)\|E\|^2.
$
Note now that by \eqref{gupperoneoneEq}
$$
\|E\|^2
=
g^{1\b1}|E_{1\b1 1}|^2|\eta^1|^2|\nu^1|^2
\le
C+\frac{\be^2(1-\be)^2}{1+b(p)\rho^{2-2\be}}\rho^{2\be-4}|\eta^1|^2|\nu^1|^2.
$$
Thus, letting $\eps=\eps(p)=b(p)\rho^{2-2\be}$, we will have proved
the lemma provided we can bound
$
(1+\rho^{2\be-2})\|A\|^2.
$
Now, by \eqref{gupperoneoneEq} and Lemma \ref{choose1} (ii),
$$
\rho^{2\be-2}\|A\|^2 =
\sum_{i,k,p,r}\rho^{2\be-2}\hat g_{i\b 1,k}\hat g_{1\jbar,\b \ell} 
g^{1\b 1}\eta^i\overline{\eta^j}\nu^k\overline{\nu^\ell}\le C.
$$
This concludes the proof of  Lemma \ref{FirstAppendixLemma}.
\end{proof}

\section{A local third derivative estimate (after Tian)}
A general result due to Tian  \cite{tian83}, proved in his M.Sc. thesis, gives a local a priori estimate in $W^{3,2}$
for solutions of both real and complex Monge--Amp\`ere equations under the assumption that the solution 
has bounded real or complex Hessian and the right hand side is at least H\"older. By the classical integral characterization 
of H\"older spaces this implies a uniform H\"older estimate on the Laplacian. This result can be seen as an alternative to 
the Evans--Krylov theorem (and in fact appeared independently around the same time). 

We present a very special case of this here which applies, in particular, to $\varphi(s)$ along the Ricci continuity path 
\eqref{RCMEq}. Unlike Calabi's estimates, this local estimate does not  require curvature bounds on the reference geometry 
(which works only when $\be<1/2$ \cite{Br}). The argument here is an immediate adaptation of \cite{tian83}
to the complex edge setting and is based entirely on the presentation in \cite{tian83} and Tian's unpublished notes 
\cite{Tian-unpublished}.  He understood the applicability of this method in the edge setting for some time, and had
described this in various courses and lectures over the years. The authors are indebted to Tian for 
generously sharing his notes
on this and for explaining the proof to them in great detail. 

\begin{thm} {\rm  (Tian \cite{tian83,Tian-unpublished})}
\label{TianThmB}
Let $\vp(s)\in \calD_w^{0,0}\cap \calC^4(M\sm D)\cap \PSH(M,\o)$ be a solution to \eqref{RCMEq} with $s>S$ 
and $0 < \beta <1$. For any $\gamma \in (0, \beta^{-1} - 1)\cap(0,1)$, there are constants $r_0\in (0,1)$ and $C> 0$ such that 
for any $x\in M$ and $0 < a < r_0$, 
\begin{equation}
\label{eq:main}
\int_{B_a(x)} |\nabla\o_\varphi|^2 \, \omega^n\,\le\, C\, a^{2n-2+2\gamma},
\end{equation}
where $B_a(x)$ denotes the geodesic ball with center $x$ and radius $a$, $\nabla$ the covariant derivative, and $|\,\cdot\,|$
the norm, all taken with respect to $\omega_\be$  \eqref{modelform}. 
The constant $C$ depend only on $\gamma,\beta,\o,n,||\Delta_\o\vp||_{L^\infty(M)},$
and $||\vp||_{L^\infty(M)}$. 
\end{thm}

For the proof, we may assume that $x\in D$ and fix some neighborhood $U$ of $x$ in $M$. We will also always assume $1/2<\be<1$ 
purely for simplicity of notation. Setting $t=1$ in \eqref{DefinehEq},
\begin{equation}
\label{Defineh2Eq}
\log\det[u_{i\b j}]= f_\o-s\vp+\log\det[\psi_{i\bar j}]=:\log h,
\end{equation}
and differentiating twice, multiplying by $h$,  and using that $(hu^{i\b j})_i=0$ (this, in turn, uses that $h=\det u_{i\b j}$), yields
\begin{equation}
\label{Defineh2Eq2}
-hu^{i\b s}u^{t\b j}u_{t\b s,\b l}u_{i\b j,k}+(hu^{i\b j}u_{k\b l,i})_{\b j}=h_{k\b l}-h_kh_{\b l}/h,
\end{equation}
Combining Lemma \ref{TYLaplacianEstimateLemma} and Lemma \ref{hLipschitzLemma} (iv) yields a uniform
bound for the right hand side of \eqref{Defineh2Eq2} (in fact, even with weaker bounds on $h$ one could replace terms
of order $a^{2n}$ that appear later by terms of order $a^{2n-\delta}$ and still run the argument with $a^{-\delta}+|\nabla\o|$ instead
of $1+|\nabla\o|$).
Here all the derivatives are with respect to $\zeta,z_2,\ldots,z_n$, equivalently, covariant derivatives with respect to $\omega_\beta$
defined in \eqref{modelform}.

\def\ze{\zeta}

Define $\bB_\beta(R)\subset U$ to be the domain in $\CC\times\CC^{n-1}$
consisting of all $(\ze,Z)$, where $Z=(z_2,\cdots,z_n)$, satisfying
$|\ze|^2+ |Z|^2 \le R^2$; recall $\ze=r e^{\sqrt{-1}\be\theta}, r\in [0,R], \theta\in [0, 2\pi].$
We often identify $\bB_\beta(R)\subset U$ with the standard ball $B_R$ in $\CC^n$.

\begin{lem}
\label{lemm:flat-1} 
(i)
Let $h$ be a harmonic function on $\bB_\beta(1)$ such that
\begin{equation}
\label{eq:boundary-1}
\begin{aligned}
h( r e^{\sqrt{-1} 2\pi \beta}, Z)\, &=\, e^{\sqrt{-1} 2\pi (1-\beta)} h(r,Z), \cr 
\del_{z_1} h( r e^{\sqrt{-1} 2\pi \beta}, Z)\, &=\, e^{\sqrt{-1} 2\pi (1-\beta)}  \del_{z_1} h(r,Z), 
~~||d h||_{L^2(\bB_\beta(1),\o_\be)} < \infty. 
\end{aligned}
\end{equation}
Then for any $a < 1$, there is a constant $C=C(\be,n)$,
\begin{equation}\label{eq:har-5}
||d h||^2_{L^2(\bB_\beta(a),\o_\be)} \le C a^{2n-4 + 2\beta^{-1}}
||d h||^2_{L^2(\bB_\beta(1),\o_\be)}.
\end{equation}
(ii)
Let $f$ be a smooth function on $\bB_\beta(1)$ satisfying \eqref{eq:boundary-1}. Then for some $C=C(\be,n)$,
\begin{equation}\label{eq:har-51}
||f||^2_{L^2(\bB_\beta(1),\o_\be)}\le C||df|^2_{L^{\frac{2n}{n+1}}(\bB_\beta(1),\o_\be)}.
\end{equation}
\end{lem}
This lemma can be proved by standard methods (e.g., Sobolev embedding $L^2\subset W^{1,\frac{2n}{n+1}}$, separation of variables 
and consideration of the indicial roots associated to the harmonic functions $z_1^{1-\be+k}=\ze^{\frac{k+1}\be-1}$, with $k=0,1,\ldots$; 
the exponent $2n - 4+ 2\beta^{-1}$ is sharp, corresponding to the first indicial root of the problem, i.e., the harmonic function 
$\ze^{\frac{1-\beta}{\beta}}$). The boundary condition \eqref{eq:boundary-1} corresponds to the
$d\zeta$-coefficient of a smooth 1-forms written in the $\zeta,z_2,\ldots,z_n$ coordinates.
To be more specific, in our application, we will consider a smooth 1-form defined on
a neighborhood in $M$ of a point $p\in D$, and write this 1-form with respect to
the aforementioned coordinates. The $d\zeta$-coefficient of this 1-form is then multivalued.
Choosing any branch, the coefficient is a function on the wedge $\bB_\beta(R)$ which satisfies
\eqref{eq:boundary-1}.

The lemma above is the only place where we need to modify \cite{tian83}. The rest of the proof below uses arguments identical to 
those of \cite[\S2]{tian83}. Since the proof was originally written for real Monge--Amp\`ere equations, we write out details
here for the complex Monge--Amp\`ere equation (this involves purely a change in notation).

\begin{lem}{\rm\cite[Lemma 2.2]{tian83}} \label{lemm:2-2}
Let $\lambda_{i\b i}, i=1,\ldots,n$ be positive numbers.
Then, 
 \begin{equation}
 \label{eq:83-1}
 \Big|\sum_k u_{k\bar k}\Pi_{i\neq k}\lambda_{i\bar i} - \det[u_{p\bar q}]- (n-1) 
\Pi_{i}\lambda_{i\bar i}\Big|\,\le\,C \sum_{i,j}|u_{i\bar j}- \delta_{ij}\lambda_{i\bar j}|^2,
 \end{equation}
where $C$ is a constant depending only on 
$\lambda_{i\bar i},u_{i\bar j}, i,j=1,\ldots,n$.
\end{lem}
\begin{proof}
First, by using the homogeneity and positivity of $[u_{i\bar j}]$, it suffices to prove the case when $[u_{i\bar j}] = [\delta_{ij}]=I$.
Next, if we denote the left side of \eqref{eq:83-1} by $f(\lambda_{1\bar 1},\ldots,\lambda_{n\bar n})$, then by a direct 
computation, $f(I)=0$, $\frac{\partial f}{\partial \lambda_{i\b i}} (I) =0$ for $i=1,\cdots, n$. Then \eqref{eq:83-1} follows from 
the Taylor expansion of $f$ at $I$.
\end{proof}

\begin{lem}{\rm\cite[Lemma 2.3]{tian83}}
 \label{lemm:2-3}
There are some uniform constants $q > 2$ and $C>0$, depending only on $\beta,n,\o,||u_{i\bar j}||_{L^\infty}, ||h_{i\bar j}||_{L^\infty}$, 
$i,j=1,\ldots,n$, 
such that for any $B_{2a}(y)\subset U$, 
 \begin{equation}
 \label{eq:83-2}
|| 1 + |\nabla \omega|^2  ||_{L^{q/2}(B_a(y),\o_\be)}
\le
C a^{2n(-1+2/q)}
|| 1 + |\nabla \omega|^2  ||_{L^1(B_{2a}(y),\o_\be)}.
 \end{equation}
\end{lem}
\begin{proof}
First we assume $y=x$. Set 
\begin{equation}
\label{eq:83-3}
\lambda _{i\bar j}:=a^{-2n}\,\int_{B_{a}(x)} u_{i\bar j}\, \omega_\beta^n,\qq
i,j=1,\ldots,n.
\end{equation}
By using unitary transformations if necessary, we may assume $\lambda_{i\bar j}=0$ for any $i\not= j$ and $i , j \ge 2$.
Let $C>0$ be such that $C^{-1} I \le \Lambda = [\lambda_{i\bar j}] \le C I.$ Choose a radial cut-off function 
$\eta: B_{a}(x) \ra \RR_+$ equal to $1 \h{\ on\ } B_{3a/4}(x)$ and supported in $B_{4a/5}(x)$, such that
$|\eta''| \le {C}/{a^2}$. Multiplying \eqref{Defineh2Eq2} by $\eta$, and integrating by parts gives
\begin{equation}
c \int_{B_{a}(x)} \eta |\nabla \omega_\vp|^2\omega^n_\beta-Ca^{2n}\le \int_{B_{a}(x)} hu^{i\bar j} \left (\sum_{k=1}^n 
u_{k\bar k}\Pi_{i\neq k}\lambda_{i\b i}- h-(n-1) \Pi\lambda_{i\b i}\right)\,\eta_{i\bar j}\,\omega^n_\beta\nonumber.
\end{equation}
Thus,
$\int_{B_{\frac{3a}{4}}(x)} |\nabla \omega_\vp|^2\,\omega^n_\beta \,\le\, C\,\left ( a^{2n}\,+ \,\int_{B_{a}(x)} \sum_{i,j=1}^n
|u_{i\bar j} - \lambda _{i\b i} \delta_{ij}|^2 \,\omega^n_\beta\right )$,
by Lemma \ref{lemm:2-2} applied to the matrix $\diag(\lambda_{1\b1},\ldots,\lambda_{n\b n})$.
Applying Lemma \ref{lemm:flat-1} (ii) to the terms $u_{1\b j}$ and the usual Sobolev inequality to the
term $u_{1\b 1}- \lambda _{1\b 1}$ and the terms $u_{i\b j}- \lambda _{i\b i} \delta_{ij}, i,j\ge2$, it follows that
$||1+|\nabla \omega_\vp|^2||_{L^1(B_{\frac{3a}{4}}(x),\o_\be)}\le Ca^{-2} ||1+|\nabla \omega_\vp|^2||_{L^{\frac{ n }{n+1}}(B_a(x),\o_\be)}.$
This inequality still holds if we replace $B_a(x)$ by any $B_a(y)$ which is disjoint from the singular set $\{z_1=0\}$. This can 
be proved by using the same arguments (without Lemma \ref{lemm:flat-1} (ii)). From this and a covering argument then
$ ||1+|\nabla \omega_\vp|^2||_{L^1(B_{a}(y),\o_\be)}\le C a^{-2} ||1+|\nabla \omega_\vp|^2||_{L^{\frac{n }{n+1}}(B_{2a}(y),\o_\be)}$
for any ball $B_{2a}(y)\subset U$. Then \eqref{eq:83-2} follows from Gehring's inverse H\"older inequality \cite[Lemma 3]{gehring73}.
\end{proof}

\begin{lem} {\rm\cite[Lemma 2.4]{tian83}}
\label{lemma:2-4}
For any $B_{4a}(y)\,\subset \,U$ and $\sigma < a$, we have
\begin{equation}
\label{eq:83-5}
||\nabla \omega_\vp||^2_{L^2(B_\sigma(y),\omega_\beta)}-Ca^{2n}
\le C\left [\left(\frac{\sigma}{a} \right )^{2n- 4+\frac2\beta}
+ a^{(2n-2)(\frac2q-1)} ||\nabla \omega_\vp||_{L^2(B_a(y),\omega_\beta)}^{\frac{2(q-2)}{q}}\right ]
||\nabla \omega_\vp||_{L^2(B_{2a}(y),\omega_\beta)}^2.
\nonumber
\end{equation}
The dependence of $C$ is as in the previous lemma.
\end{lem}
\begin{proof}
Let $v$ be the unique (1,1)-form on $B_a(y)$ solving
\begin{equation}
\label{eq:83-6}
\sum_{k=1}^n\,\Pi_{i\neq k}\lambda_{i\b i} v_{k\bar k}\,=\, 0~~{\rm on}~B_a(y),~~~v=\o_\vp ~~{\rm on}~\partial B_a(y).
\end{equation}
We emphasize that here $v_{k\bar k}=\nabla_k\nabla_{\bar k} v$ denotes
covariant derivatives with respect to $\o_\be$ of the full $(1,1)$-form $v$.
Set $\hat\o:=\ovp-v$. 
Then, $||\nabla\ovp||_{L^2(B_{\sigma}(y),\o_\be)}^2\le
2||\nabla\hat\o||_{L^2(B_{\sigma}(y),\o_\be)}^2+2||\nabla v||_{L^2(B_{\sigma}(y),\o_\be)}^2\
$.
Note that $v$ is harmonic with respect to a constant
coefficient metric equivalent to $\o_\be$.
Thus, $||\nabla v||_{L^2(B_{a}(y),\o_\be)}\le C||\nabla\ovp||_{L^2(B_{a}(y),\o_\be)}$, and
applying Lemma \ref{lemm:flat-1} (i) to $v$
(or more precisely to each of the components $v_{1\b j},v_{i\b 1}, i,j\ge2$, 
of the form where $v$ with respect to the coordinates $\zeta,z_2,\ldots,z_n$ gives,
\begin{equation}
\begin{aligned}
||\nabla\ovp||_{L^2(B_{\sigma}(y),\o_\be)}^2
&\le
2||\nabla\hat\o||_{L^2(B_{\sigma}(y),\o_\be)}^2
+2C\left(\sigma/a\right)^{2n-4 + 2\beta^{-1}}||\nabla v||_{L^2(B_a(y),\o_\be)}^2
\cr
&\le
2||\nabla\hat\o||_{L^2(B_{\sigma}(y),\o_\be)}^2
+2C'\left({\sigma}/{a}\right)^{2n-4 + 2\beta^{-1}}||\nabla\ovp||_{L^2(B_a(y),\o_\be)}^2
.
\end{aligned}
\end{equation}
It remains to estimate the first term on the right hand side.
Similarly to before, multiplying \eqref{Defineh2Eq} by $\hat \omega\w \o_\be^{n-1}$ and integrating by parts,
\begin{equation}
\label{eq:83-7}
\int_{B_r(y)} |\nabla \hat\omega|^2 \omega^n_\beta
\le C\Big(r^{2n}+\,\int_{B_r(y)}  (|\hat \omega|+
 \sum_{i,j}|u_{i\bar j}-\lambda_{i}\delta_{i j}|^2)(1+|\nabla \omega_\vp|^2)\omega^n_\beta\Big).
\end{equation}
By using Lemma \ref{lemm:2-3} and the Poincar\'e inequality (the usual one with matching boundary
data $f( r e^{\sqrt{-1} 2\pi \beta}, Z)= f(r,Z) $ as well as Lemma \ref{lemm:flat-1} (ii)),
\begin{equation}
\begin{aligned}
\label{eq:83-8}
||\,|u_{i\bar j}-\lambda_{i}\delta_{ij}|^2(1+|\nabla \omega_\vp|^2)&||_{L^1(B_a(y),\omega_\beta)}
\le 
||1+|\nabla \omega_\vp|^2\,||_{L^{q/2}(B_a(y),\omega_\beta)} ||\,|u_{i\bar j}-\lambda_{i}\delta_{i j}|\,||^2_{L^{\frac{2q}{q-2}}(B_a(y),\omega_\beta)}
\nonumber
\cr
&\le C 
a^{\frac{q-2}q(2-2n)}
||1+|\nabla \omega_\vp|^2||_{L^1(B_{2a}(y),\omega_\beta)}
||u_{i\bar j}-\lambda_{i}\delta_{i j}||^{\frac{2(q-2)}q}_{L^2(B_a(y),\omega_\beta)}
\cr
&\le C 
a^{\frac{q-2}q(2-2n)}
||1+|\nabla \omega_\vp|^2||_{L^1(B_{2a}(y),\omega_\beta)}
||1+|\nabla \omega_\vp|^2||^{\frac{q-2}q}_{L^1(B_{a}(y),\omega_\beta)}.
\end{aligned}
\end{equation}
Without loss of generality, we may assume that $q \ge 2(q-2)$. Since $\hat\omega$ vanishes on $\partial B_r(y)$, its $L^2$-norm is controlled by the $L^2(B_{a}(y),\omega_\beta)$-norm of
$|\nabla \hat \omega|$ and, consequently, of $|\nabla \omega_\vp|$
(as $||\nabla \hat \omega||^2_{L^2}\le 2||\nabla \ovp||^2_{L^2}+2||\nabla v||^2_{L^2}
\le 2(C+1)||\nabla \ovp||^2_{L^2}$). 
Also, $\hat\omega$ is uniformly bounded in $L^\infty$ as
both $\ovp$ (by the Laplacian estimate) and $v$ (by the maximum principle) are; thus
its $L^2$ norm is equivalent to its $L^{\frac{q}{q-2}}$ norm.
Then,
\begin{equation}
\begin{aligned} 
\label{eq:83-9}
||\,|\hat \omega|(1+|\nabla \omega_\vp|)^2||_{L^1(B_a(y),\omega_\beta)}
&\le 
||1+|\nabla \omega_\vp|^2||_{L^{q/2}(B_a(y),\omega_\beta)}
||\hat\omega||_{L^{\frac{q}{q-2}}(B_a(y),\omega_\beta)}
\nonumber
\cr
&\le  
Ca^{2n(-1+2/q)}||1+|\nabla \omega_\vp|^2\,||_{L^1(B_{2a}(y),\omega_\beta)}
||\hat\omega||^{\frac{2(q-2)}{q}}_{L^2(B_a(y),\omega_\beta)}
\nonumber
\cr
&\le  
Ca^{2n(-1+2/q)}||1+|\nabla \omega_\vp|^2\,||_{L^1(B_{2a}(y),\omega_\beta)}
a^{\frac{2(q-2)}q}||\nabla\hat\omega||^{\frac{2(q-2)}{q}}_{L^2(B_a(y),\omega_\beta)}
\nonumber
\cr
&\le  
Ca^{(2n-2)(-1+2/q)}||1+|\nabla \omega_\vp|^2\,||_{L^1(B_{2a}(y),\omega_\beta)}
||\nabla\omega_\vp||^{\frac{2(q-2)}{q}}_{L^2(B_a(y),\omega_\beta)}.
\nonumber
\end{aligned} 
\end{equation}
Combining all the estimates above concludes the proof.
\end{proof}

The next lemma gives an estimate on the smallness of the coefficient in the right hand
side of the previous lemma.

\begin{lem} {\rm \cite[Lemma 2.5]{tian83}}
\label{lemm:2-5}
For any $\epsilon_0 > 0$, there is an $\ell$ depending only on $\epsilon_0$, $||\Delta u||_{L^\infty}$ and $||h_{i\bar j}||_{L^\infty}$ satisfying: For any
$a > 0$ with $B_{a}(y) \subset U$, there is $\sigma \in [ 2^{-\ell} a , 2^{-1}a]$ such that
\begin{equation}
\label{eq:83-11}
||\nabla \omega_\vp||^2_{L^2(B_{\sigma}(y),\o_\be)}\le\epsilon_0\sigma^{2n-2}.
\end{equation}
\end{lem}

\begin{proof}
From \eqref{Defineh2Eq},
\begin{equation}
\label{eq:83-cma}
\Delta_{\ovp} \Delta u = 
\sum_k u^{i\bar j} u^{p\bar q} u_{i\bar q k} u_{\bar j p \bar k}+\Delta\log h,
\end{equation}
where $\Delta$ denotes, as before, the Laplacian of $\omega_\be$.
Let $\eta $ be a non-negative radial function on $B_a(y)$ equal to $1$ 
on $B_{a/2}(y)$, supported in $B_{3a/4}(y)$, and such that
$a^2|\eta''| \le C$. Let
$M_a:=\sup_{B_a(y)} \Delta u$. Then,
\begin{eqnarray}
\label{eq:83-12}
\int_{B_a(y)} \eta|\nabla\omega_\vp|^2 \,\omega_\beta^n
\le 
C a^{2n}- \int_{B_a(y) } \Delta_{\ovp} \eta \, (M_a - \Delta u)\,\omega^n\nonumber
\le C\, a^{2n}+C a^{-2}  \int_{B_a(y) }\, (M_a - \Delta u)\,\omega_\be^n.
\end{eqnarray}
From \eqref{eq:83-cma}, there exists $c > 0$, 
such that $Z:= M_a - \Delta u - ca^2$ satisfies
$\Delta_{\ovp}Z \le 0$. Thus, 
$$\frac1Ca^{-2n} \int_{B_a(y)} \, Z \,\omega^n\,\le \inf_{B_{{a}/{2}}(y)} Z+a^2,$$
by \cite[Thm8.18]{GT}.
Thus, $a^{2-2n}||\nabla\omega_\vp||^2_{L^2(B_{a/2}(y),\o_\be)}\le C( M_a - M_{a/2} + a^2).$
Hence, if \eqref{eq:83-11} does not hold for $\sigma=2^{-1}a, \cdots, 2^{-k}a$, then
$(k-1)\epsilon_0\le C \,(M_{a } - M_{2^{-k} a} + 2 a^2). $
This is impossible if $k$ is sufficiently large.
\end{proof}

We now complete the proof of Theorem \ref{TianThmB}.
Using the previous two lemmas, there exist uniform $\chi,\lambda\in(0,1)$ such that 
$$ (\lambda a)^{\frac2\be-2}+(\lambda a)^{2-2n}||\nabla\o_\vp||^2_{L^2(B_{\lambda a}(y),\o_\be)} 
\le \chi \Big[a^{\frac2\be-2}+ a^{2-2n}||\nabla\o_\vp||^2_{L^2(B_{a}(y),\o_\be)}\Big].
$$
Thus, as in \S 8, from \cite[Lemma 8.23]{GT}
it follows that there exists some $\gamma\in(0,\frac1\be-1)$ for which \eqref{eq:main} holds,
which is sufficient for the purposes of this article (it is easy to see that in fact
$\chi,\lambda$ can be chosen to give \eqref{eq:main} even for all $\gamma\in(0,\frac1\be-1)$).
This concludes the proof.

\begin{spacing}{0.1}

\end{spacing}

\bigskip

{\sc Wichita State University}

{\tt jeffres@math.wichita.edu}

\medskip

{\sc Stanford University}

{\tt mazzeo@math.stanford.edu}

\medskip

{\sc University of Maryland}

{\tt  yanir@member.ams.org}

\end{document}